\documentclass[10pt]{article}

\usepackage{amsmath,amssymb,amsthm}
\usepackage{dsfont}
\usepackage{xcolor}
\usepackage{amsfonts}
\usepackage{algorithm}
\usepackage{algpseudocode}
\usepackage[colorlinks,linkcolor={blue},
citecolor={blue},urlcolor={blue},]{hyperref}
\usepackage{hyperref}
\usepackage{tikz}
\usetikzlibrary{intersections} 
\usepackage{pgfplots}
\usepgfplotslibrary{fillbetween}
\usepackage[affil-it]{authblk}
\usepackage{booktabs}
\usepackage{algorithmicx}%

\allowdisplaybreaks

\makeatletter
\newenvironment{breakablealgorithm}
  {
   \begin{center}
     \refstepcounter{algorithm}
     \hrule height.8pt depth0pt \kern2pt
     \renewcommand{\caption}[2][\relax]{
       {\raggedright\textbf{\fname@algorithm~\thealgorithm} ##2\par}%
       \ifx\relax##1\relax 
         \addcontentsline{loa}{algorithm}{\protect\numberline{\thealgorithm}##2}%
       \else 
         \addcontentsline{loa}{algorithm}{\protect\numberline{\thealgorithm}##1}%
       \fi
       \kern2pt\hrule\kern2pt
     }
  }{
     \kern2pt\hrule\relax
   \end{center}
  }
\makeatother

\newtheorem{Theorem}{Theorem}[section]
\newtheorem{Proposition}[Theorem]{Proposition}
\newtheorem {Cor}[Theorem]{Corollary}
\newtheorem {pro}[Theorem]{Proposition}

\newtheorem {Lemma}[Theorem]{Lemma}
\newtheorem {rem}[Theorem]{Remark}
\newtheorem {rems}[Theorem]{Remarks}
\newtheorem {com}[Theorem]{Comment}
\newtheorem {coms}[Theorem]{Comments}
\newtheorem {warning}[Theorem]{Warning}
\newtheorem {notation}[Theorem]{Notation}
\newtheorem {Definition}[Theorem]{Definition}

\newtheorem {exer}[Theorem]{Exercise}
\newtheorem {Assumption}[Theorem]{Assumption}

\newcommand{\bnota}{\begin{notation} \rm } \newcommand{\enota}{\end{notation}}
\newcommand{\bw}{\begin{warning} \rm } \newcommand{\ew}{\end{warning}}
\newcommand{\bcom}{\begin{com} \rm } \newcommand{\ecom}{\end{com}}
\newcommand{\bcoms}{\begin{coms} \rm } \newcommand{\ecoms}{\end{coms}}
\newcommand {\bdefi}{\begin{Definition}}
\newcommand {\edefi}{\end{Definition}}
\newcommand {\bl}{\begin{Lemma}}
\newcommand {\el}{\end{Lemma}}
\newcommand {\bethe}{\begin{Theorem}}
\newcommand {\eethe}{\end{Theorem}}
\newcommand {\bp}{\begin{pro}}
\newcommand {\ep}{\end{pro}}
\newcommand {\bcor}{\begin{Cor}}
\newcommand {\ecor}{\end{Cor}}
 \newcommand {\brem }{\begin{rem} \rm }
\newcommand {\erem }{\end{rem}}
 \newcommand {\brems }{\begin{rems} \rm }
\newcommand {\erems }{\end{rems}}
\newcommand {\bexo}{\begin{exer} \rm }
\newcommand {\eexo}{\hfill $\lhd$ \end{exer}}

\renewcommand{\tilde}{\widetilde}

\let\ssection=\section
\renewcommand{\section}{\setcounter{equation}{0}\ssection}
\usepackage{graphicx}
\usepackage{subfig}

\usepackage[title]{appendix}
\usepackage{array,xtab,ragged2e}

\renewcommand{\footnoterule}{%
  \kern 2 pt
  \hrule width \textwidth width 2in
  \kern 2pt
}

\newcommand{\be}{\begin{equation}}
\newcommand{\ee}{\end{equation}}
\newcommand{\bde}{\begin{displaymath}}
\newcommand{\ede}{\end{displaymath}}
\newcommand{\beq}{\begin{eqnarray*}}
\newcommand{\eeq}{\end{eqnarray*}}
\newcommand{\beqa}{\begin{eqnarray}}
\newcommand{\eeqa}{\end{eqnarray}}
\newcommand{\bel }{\left\{\begin{array}{ll}}
\newcommand{\eel}{\cr \end{array} \right.}

\title{American Options with Last Exit Times: A Free-Boundary Approach } 
\author{Zhuoshu Wu\thanks{Corresponding author. Email address: zhuoshuwu@hotmail.com.}\thanks{The authors are grateful to Professor Ben Goldys for helpful discussions.}}
\author{Libo Li}

\affil{The University of New South Wales, Sydney}
\begin{document}
\maketitle

\begin{abstract}
We study an optimal stopping problem with a random horizon that is not a stopping time: the last exit time of a geometric Brownian motion from a fixed level $L$. The defaultable American put option serves as the motivating example. Using the theory of enlargements of filtration and the associated Az\'{e}ma supermartingale, we reduce the problem to a classical optimal stopping problem in a finite time horizon for a Markov process, whose gain function is only upper semi-continuous and has a concave kink at $L$ and a spatial derivative that blows up. These features break the standard free-boundary arguments, including the smooth-fit condition. We describe the continuation and stopping sets according to the position of $L$ relative to $B$, the optimal stopping boundary of the standard American put option. When $B$ lies completely above $L$, the random-horizon problem coincides with the standard American put option pricing problem on the continuation set. When $B$ lies below $L$ for some time interval, the optimal stopping boundary of the random-horizon problem lies below or stays at $L$ until $B$ reaches $L$ and agrees with $B$ when $B$ exceeds $L$. The smooth-fit condition fails at $L$, and we obtain upper and lower bounds for the one-sided derivative of the value function. Under numerically verifiable parameter conditions, we show that the boundary is monotone and continuous, and we derive a nonlinear integral equation and an early-exercise-premium representation with local time terms. \\


\textbf{Classifications:} 60G40, 91G20, 35R35 
\end{abstract}

\section{Introduction}
We consider an optimal stopping problem in which the reward becomes worthless at a random time that is not a stopping time. A natural and tractable example is the last exit time of a diffusion process from a level: it depends on the whole path up to the terminal time $T$, so the classical optimal stopping theory cannot be applied directly, see \cite[Chapter 8.2]{Nikeghbali} and, for related \textit{optimal prediction} problems, \cite{BoyceWM1970, DuToitandPeskir2009, DuToitandPeskirandShiryaev2008, GilbertJPandMostellerF2006, GriffeathDandSnellJL1974, KarlinS1962}. 

\medskip
\noindent\textbf{The problem.}
Let $(\Omega, \mathcal{F}, (\mathcal{F}_t)_{t \geq 0}, \mathbb{P})$ be a complete filtered
probability space satisfying the usual conditions, with $(\mathcal{F}_t)_{t\ge0}$ generated
by a standard Brownian motion $W=(W_t)_{t\in[0,T]}$. Let $X$ be the geometric Brownian
motion (GBM) satisfying
\begin{align*}
dX_t = rX_t\,dt+\sigma X_t\,dW_t,\qquad X_0=x>0,
\end{align*}
with the interest rate $r >0$ and the volatility coefficient $\sigma>0$, and let $\theta=\sup\{t\in[0,T]: X_t\ge L\}$ be the last time GBM $X$ is at or above a level $L$, where $0<L<K$ with $K>0$ being the strike price. We study
\begin{align}
V=\sup_{\tau\in[0,T]}\mathbb{E}\Big[e^{-r\tau}(K-X_\tau)^+\,I\{\tau<\theta\}\Big],
\label{VCIC1}
\end{align}
where the supremum is taken over all $(\mathcal{F}_t)$-stopping times $\tau$ with values in $[0, T]$. The last exit time $\theta$ depends on the whole path up to $T$ and is not a stopping time.

\emph{Motivating example.} Under the risk-neutral measure $\mathbb{P}$, \eqref{VCIC1} is the
value of an American put option with strike $K$ and maturity $T$ that is cancelled at a default time $\theta$: once the stock price has left $[L,\infty)$ for the last
time, the contract is worthless unless already exercised (see
\cite{ElliottRJandJeanblancMandYorM2000,HullandWhite,LW2024}, and \cite{LW2024} for the infinite horizon case). Our interest, however, is in the mathematical structure of
\eqref{VCIC1}.

\medskip
\noindent\textbf{Reformulation.}
Following Mansuy and Yor \cite{MansuyandYor} and Nikeghbali \cite{Nikeghbali}, we replace $I\{\tau<\theta\}$ by its conditional expectation given $\mathcal{F}_\tau$. This gives a classical optimal stopping problem with gain function
\begin{align}
G(t, x) = (K-x)^+ Z(t, x), \qquad Z(t, x) = \mathbb{P}(\theta>t | X_t = x), \label{GCIC1}
\end{align}
where $Z$ is the Az\'{e}ma supermartingale, which we compute explicitly (Proposition \ref{PSPML}). Throughout we assume $r<\frac{\sigma^2}{2}$ (Assumption \ref{Assumptiononalpha}), this ensures that $Z$ is capped at $1$ for $x\geq L$. With the value function being lower semi-continuous and the gain function being upper semi-continuous, the optimal stopping time exists (Lemma \ref{tauDisoptimal}).

\medskip
\noindent\textbf{Analytical contributions.}
The gain function is not covered by the framework of Peskir-Shiryaev \cite{PeskirandShiryaev} and its time-dependent extensions  \cite{DuToitandPeskirandShiryaev2008, DuToitandPeskir2009, GloverandPeskirandSamee2011}:
\begin{enumerate}
\item $G$ is not $C^1$ in $x$ at $L$ (and at $K$) and is not continuous at $(T, L)$;
\item $G$ is not monotone in $x$;
\item the standard comparison argument for connectedness requires the Hamiltonian $H = -rG + G_t + \mathbb{L}_X G$ to be monotone, which fails here without extra assumption.
\end{enumerate}

\medskip
\noindent\textbf{Main results.}
Let $B$ be the optimal stopping boundary of the standard American put option with strike price $K$ and let $t_*$ be the first time $B$ reaches $L$. Let $b$ be the optimal stopping boundary of the defaultable American put option with gain function \eqref{GCIC1}.
\begin{enumerate}
\item Structure of the continuation and stopping sets. If $B(0) > L$, then $b = B$ and $V$ coincides with the value function of the standard American put option on the continuation set, so \eqref{VCIC1} reduces to the classical problem (Lemma \ref{BICOLLB}, Theorem \ref{MTLGB} (i)). If $B(0)\leq L$, the stopping set contains the whole region below the maximiser $x^*(t)$ of $G(t, \cdot)$ on $(0, L]$, $b$ coincides with $B$ on $[t_*, T]$, and $b$ stays at $L$ over a specified time interval.
\item Monotonicity and continuity. Under Assumption \ref{TMOHTC5}, which is numerically checkable (Figures \ref{themapofxstar2} and \ref{OPSC4IFPICTURE}), $b$ is increasing (Propositions \ref{TMOVMGL} and \ref{TDOTBC5}). Continuity of $b$ is proved without the smooth-fit condition by an argument similar to \cite{Deangelis2015} (Lemma \ref{BISCus}).
\item Loss of the smooth-fit condition. When $b(t) = L$, smooth fit breaks down, and we show the upper and lower bounds for $V_x(t, L+)$ (Proposition \ref{TFBPFC5BCL}), in the spirit of \cite{DetempleandKitapbayev2018, Qiu2016}.
\item Integral equation. We derive an early-exercise-premium representation of $V$ and a nonlinear integral equation for $b$ containing local-time terms at $L$, and prove uniqueness (Theorem \ref{MTLGB}). The theorem in the case of $B(0)\leq L$ is conditional on Assumption \ref{TMOHTC5}.
\end{enumerate}

\medskip
\noindent\textbf{Related work.}
Optimal stopping problems with time-dependent rewards are treated in \cite{DuToitandPeskirandShiryaev2008, DuToitandPeskir2009}, and problems with non-smooth but time-independent rewards, including the loss of smooth fit condition, in \cite{Qiu2016, DetempleandKitapbayev2018}; The continuity of the free-boundaries without smooth fit condition is shown in \cite{Deangelis2015}. For the infinite horizon version of \eqref{VCIC1}, see \cite{LW2024}.

\medskip
\noindent\textbf{Organisation.}
Section \ref{survival} computes the Az\'{e}ma supermartingale and collects useful properties associated with it. Section \ref{TFBPC5} formulates the free-boundary problem, while Section \ref{TCASSC5} analyses the continuation and stopping sets in two cases $B(0) > L$ and $B(0)\leq L$. Section \ref{TOSPRC5} contains the main representation theorem and the numerical approximation of the optimal stopping boundary.

\section{The Optimal Stopping Problem}\label{stopping}
In this section, we reformulate the main optimal stopping problem \eqref{VCIC1} and collect some elementary facts for later use.
\subsection{The Az\'{e}ma Supermartingale}\label{survival}
We begin by an exploitation of the smoothing lemma in \eqref{VCIC1} and obtain
\begin{align}
V &=\sup_{0\leq\tau\leq{T}} \mathbb{E} \left[ e^{-r\tau} \left( K-X_{\tau} \right)^+ \mathbb{P}(\theta>\tau | \mathcal{F}_{\tau})  \right]. \label{OSPFL}
\end{align}
In view of \eqref{OSPFL}, the first thing is to calculate $P( \theta > \tau | \mathcal{F}_{\tau})$ and let us introduce the notation\footnote{Not to confuse this $x$ with $X_0=x$.} 
\begin{align*}
Z(t, x)= \left( \Phi \left( \frac{ - \log{\frac{L}{x}} + \left( r-\frac{\sigma^2}{2} \right)(T-t) }{\sigma\sqrt{T-t}} \right) + \left(\frac{L}{x}\right)^\alpha \Phi \left( \frac{-\log{\frac{L}{x}} - \left( r-\frac{\sigma^2}{2} \right)(T-t) }{\sigma\sqrt{T-t}} \right) \right)\wedge 1,
\end{align*}
where $\Phi(x)=\frac{1}{\sqrt{2\pi}} \int_{-\infty}^x e^{-\frac{z^2}{2}} dz$ and $\alpha=\frac{2r}{\sigma^2}-1$.

\begin{Proposition} \label{PSPML}
Let $\mathbb{P}\left(\theta>t  | \mathcal{F}_{t} \right)$ be the Az\'{e}ma supermartingale associated with the random time $\theta$. Then, for $r-\frac{\sigma^2}{2}\in\mathbb{R}$, $\sigma>0$ and $t\in[0, T)$,
\begin{align*}
\mathbb{P}\left(\theta>t  | \mathcal{F}_{t}\right)&= Z(t, X_t).
\end{align*} 
In addition, $Z(T, X_T)=1$ for $\log\frac{L}{X_T}\leq 0$ and $0$ otherwise.
\end{Proposition}

\begin{proof}
First of all, suppose that
\begin{align*}
{d}_t &= \inf \{ u\geq t: X_u \geq L \} = t + \inf \{ u \geq 0: \bar{X}_0= X_t, \bar{X}_u  \geq L \},
\end{align*}
where the second equality follows from the Markov property of $X$ and for easy reference, let
\begin{align*}
\bar{h}_L&=\inf \{ u\geq 0 :  \bar{X}_0= X_t,  \bar{X}_{u} \geq L  \}, 
\end{align*}
that is, $\bar{h}_L$ is the first time when a geometric Brownian motion $\bar{X}$, starting from $X_t$, hits the level $L$. Then, it follows that
\begin{align*}
\mathbb{P}(\theta>t | \mathcal{F}_{t}) &= 1 - \mathbb{P}\left(\theta \leq t |  \mathcal{F}_{t}\right)\\
&= 1- \mathbb{P}\left(d_t > T | \mathcal{F}_{t}\right)\\
&= 1- \mathbb{P}\left(\bar{h}_L> T-t | \mathcal{F}_{t}\right).
\end{align*}
Upon observing that, in fact,
\begin{align}
\{ \bar{h}_L \leq T-t\} &= \bigg\{ \max_{0\leq u \leq T-t} \bar{X}_{u} \geq L\ \bigg\}=\bigg\{ \max_{0\leq u \leq T-t} X_t e^{\left( r-\frac{\sigma^2}{2} \right) u + \sigma W_u } \geq L \bigg\},\label{seteqc5}
\end{align}
after which, an application of Lemma 1 in \cite[Pages 759 and 760]{Shiryaev1999}  proves that, for $\log{\frac{L}{X_t}} \geq 0$, $r-\frac{\sigma^2}{2} \in \mathbb{R}$ and $\sigma>0$, 
\begin{align*}
\mathbb{P}\left( \theta>t | \mathcal{F}_t \right)  &= \Phi \left( \frac{ - \log{\frac{L}{X_t}} + \left( r-\frac{\sigma^2}{2} \right)(T-t) }{\sigma\sqrt{T-t}} \right) \\
& \qquad \qquad \qquad + \left(\frac{L}{X_t}\right)^\alpha \Phi \left( \frac{-\log{\frac{L}{X_t}} - \left( r-\frac{\sigma^2}{2} \right)(T-t) }{\sigma\sqrt{T-t}} \right).
\end{align*}
while for $\log{\frac{L}{X_t}} < 0$, observe that event \eqref{seteqc5} happens almost surely, namely, 
\begin{align}
\mathbb{P}\left( \theta>t | \mathcal{F}_t \right) = 1. \label{XtgeqL}
\end{align}

As for $t=T$, from the set equality \eqref{seteqc5}, it follows that
\begin{align*} 
\mathbb{P}(\theta>T |\mathcal{F}_T)=
\begin{cases}
1,& \log{\frac{L}{X_T}}\leq 0,\\
0,& \log{\frac{L}{X_T}}>0.
\end{cases}
\end{align*}
\end{proof}

\begin{Assumption} \label{Assumptiononalpha}
Hereafter, it is always assumed that $\alpha<0$.
\end{Assumption}

\begin{Proposition} \label{PDEOZCL}
The function $Z(t, x)$ satisfies the following partial differential equation, for $(t, x)\in[0, T]\times\{ (0, L)\cup(L, \infty) \}$,
\begin{align}
\frac{\partial }{\partial t} Z(t, x) + rx \frac{\partial}{\partial x} Z(t, x) + \frac{1}{2} \sigma^2 x^2 \frac{\partial^2}{\partial x^2} Z(t, x)=0. \label{PDEOZL}
\end{align}
\end{Proposition}

\begin{proof}
Let us apply the \textit{change-of-variable} formula (see \cite{Peskir2005AC}) to obtain
\begin{align}
d Z(t, X_t) &= \left( \frac{\partial }{\partial t} Z + r X_t \frac{\partial}{\partial x} Z + \frac{1}{2} \sigma^2 {X_t}^2 \frac{\partial^2}{\partial x^2} Z \right) (t, X_t) I\{ X_t \neq L \} dt \label{COVOZF}\\
&\,\,\,\,\,\,\,+\sigma X_t \frac{\partial}{\partial x} Z(t, X_t)  I\{ X_t \neq L \} d W_t + \frac{1}{2} \left( \frac{\partial}{\partial x} Z(t, L+) - \frac{\partial}{\partial x} Z(t, L-)   \right) dl_t^L, \nonumber
\end{align}
where $l_t^L$ is the local time of $X$ at the level $L$ given by
\[ l_t^{L}=\mathbb{P}-\lim_{\epsilon \to 0} \frac{1}{2\epsilon} \int_0^t I\{ | X_s - L | < \epsilon \} d \langle X, X \rangle_s, \]
after which, we can lean on the Doob-Meyer decomposition $Z(t, X_t) = Z(0, x) + M_t - A_t$ to conclude that
\begin{align*}
M_t&= \int_0^t \sigma X_s \frac{\partial}{\partial x} Z(s, X_s) I\{ X_s \neq L \} d W_s, \\
A_t&=\frac{1}{2} \int_0^t \frac{\partial}{\partial x} Z(s, L-) dl_s^L,
\end{align*}
where $M_t$ is the c$\grave{a}$dl$\grave{a}$g martingale and $A_t$ is the predictable projection whose measure $dA_t$ is carried by the set $\{t: X_t=L\}$, and consequently, \eqref{PDEOZL} follows.
\end{proof}

\begin{Lemma}\label{MAPOZwrttax}  The map $t\mapsto Z(t, x)$ is decreasing on $[0, T]$ and  $x\mapsto Z(t, x)$ is increasing on $(0, \infty)$.
\end{Lemma}

\begin{proof}
The claim follows directly from computing the derivatives, see \eqref{A3} and \eqref{A4}.
\end{proof}

\subsection{The Free-boundary Problem} \label{TFBPC5}

By the strong Markov property of $X$, we generalise the original problem \eqref{OSPFL} as follows:
\begin{align}
V(t, x)&=\sup_{0\leq\tau\leq{T-t}} \mathbb{E}_{t, x} \left( e^{-r\tau} \left( K-X_{t+\tau} \right)^+ Z(t+\tau, X_{t+\tau})  \right), \label{OSPFLR}
\end{align}
where the supremum is taken over all stopping times $\tau$ of $X$ with values in $[0, T-t]$ and $\mathbb{E}_{t, x}$ is taken with respect to the measure $\mathbb{P}_{t, x}$ under which $X_t=x$. For simplicity, let $G:[0, T]\times(0, \infty)\mapsto[0,K]$ be the gain function such that $G(t, x)=(K-x)^+Z(t, x)$. The infinitesimal generator of the process $X$ is given by $\mathbb{L}_X=rx \frac{\partial}{\partial x}  + \frac{\sigma^2}{2} x^2 \frac{\partial^2}{\partial x^2}$.

Before formulating the \textit{free-boundary problem}, we introduce the continuation and stopping sets.
\begin{Definition}
The continuation set $\mathcal{C}$ and the stopping set $\mathcal{D}$ are defined as follows:
\begin{align}
\mathcal{C}&=\{ (t, x)\in[0, T) \times (0, \infty): V(t, x)>G(t, x) \}, \label{CSVGL} \\
\mathcal{D}&=\{ (t, x)\in[0, T) \times (0, \infty): V(t, x)=G(t, x) \} \cup \{ (T, x): x\in(0, \infty) \}, \label{DSVGL}
\end{align}
and the first entry time $\tau_\mathcal{D}$ of $X$ into $\mathcal{D}$ is defined by
\begin{align}
\tau_\mathcal{D} = \inf \{ 0\leq s \leq T-t: ( t+s, X_{t+s}^x )\in \mathcal{D}  \} \wedge (T-t).
\end{align}
\end{Definition}
According to \cite[Corollary 2.9, Page 46]{PeskirandShiryaev}, it is sufficient to show that $V$ is lower semi-continuous and $G$ is upper semi-continuous in order to establish the optimality of $\tau_{\mathcal{D}}$. It is worth mentioning that the partial derivative $Z_x(t, x)$ has a singularity at the point $(L, T)$, and therefore, establishing the continuity of the value function $V$ can only be pursued later. 

\begin{Lemma} \label{VILSCC5}
The value function $V$ is l.s.c on $[0, T)\times(0, \infty)$. 
\end{Lemma}

\begin{proof}
The proof is provided in Appendix \ref{ALAP}.
\end{proof}

\begin{Lemma}\label{tauDisoptimal}
The stopping time $\tau_{\mathcal{D}}$ is optimal for the optimal stopping problem \eqref{OSPFLR}.
\end{Lemma}

\begin{proof}
The proof is provided in Appendix \ref{ALAP}.
\end{proof}

It will be shown in the next section that the continuation and stopping sets can be defined as
\begin{align}
&\mathcal{C}=\{ (t, x) \in [0, T)\times(0, \infty): x>b(t) \}, \label{CSIBL}\\
&\mathcal{D}=\{ (t, x) \in [0, T)\times(0, \infty): x\leq b(t) \} \cup \{ (T, x): x\in(0, \infty) \}, \label{DSIBL}
\end{align}
where $b:[0, T]\mapsto \mathbb{R}$ is the unknown optimal stopping boundary such that $\tau_\mathcal{D}$ can be rewritten as
\begin{align}
\tau_\mathcal{D} = \inf \{ 0\leq s \leq T-t : X_{t+s}^x \leq b(t+s) \}\wedge (T-t).
\end{align}
Moreover, by Theorem 2.4 in \cite[Page 37 and Page 130 for the justification of $V$ being a solution to the PDE]{PeskirandShiryaev} and Lemma \ref{tauDisoptimal}, the \textit{free-boundary} problem can be formulated as:
\begin{align}
&{V}_t + \mathbb{L}_X {V} - r {V} = 0 \qquad\qquad\qquad\qquad\, \text{for $(t, x)\in\mathcal{C}$}, \label{FBPL1}\\
&{V}(t, x)={G}(t, x) \qquad\qquad\qquad\qquad\,\,\,\,\,\,\,\,\,\, \text{for $x=b(t)$, instantaneous stopping},\label{FBPL2}\\
&{V}_x(t, x) ={G}_x(t, x) \qquad\qquad\qquad\qquad\,\,\,\,\, \text{for $x=b(t)\neq L$, smooth-fit},\label{FBPL3}\\
&{V}(t, x)>{G}(t, x) \qquad\qquad\qquad\qquad\,\,\,\,\,\,\,\,\,\, \text{for $(t, x)\in\mathcal{C}$},\label{FBPL4}\\
&{V}(t, x)={G}(t, x) \qquad\qquad\qquad\qquad\,\,\,\,\,\,\,\,\,\, \text{for $(t, x)\in\mathcal{D}$}.\label{FBPL5}
\end{align}
Equations \eqref{CSIBL}, \eqref{DSIBL}, \eqref{FBPL2} and \eqref{FBPL3} are to be shown in Section \ref{TCASSC5}.

\subsection{The Continuation and Stopping Sets} \label{TCASSC5}

We begin by introducing some useful notations:
\begin{table}[H]
    \centering
    \caption{Notation}
    \label{tab:notation}
    \begin{tabular}{ll}
        \toprule
        \textbf{Notation} & \textbf{Description} \\
        \midrule
        $V$ & Value function for defaultable American put option with payoff $(K-x)^+ Z(t, x)$ \\
        $G$ & Gain function for defaultable American put option \\
        $V^A$ & Value function for standard American put option with payoff $(K-x)^+$ \\
        $G^A$ & Gain function for standard American put option\\
        $\mathcal{D}^A$ & Stopping set for standard American put option \\
        $\mathcal{C}^A$ & Continuation set for standard American put option \\
        $b(t)$ & Optimal stopping boundary for defaultable American put option\\
        $B(t)$ & Optimal stopping boundary for standard American put option \\
        $x^*(t)$ & the unique maximizer of the gain function $G(t,\cdot)$ on $(0, L]$, for each fixed $t\in[0, t_*]$   \\
        $t_*$ & First time at which $B(t) = L$ \\
        $t_{x^*}$ & First time at which $x^*(t) = L$\\
        $t_b$ & First time at which $b(t) = L$\\
        $t^b$ & Last time at which $b(t) = L$\\
        \bottomrule
    \end{tabular}
\end{table}

As further preparation for the detailed analysis, we lean on the change-of-variable formula (see \cite{Peskir2005AC}), as the function $G$ is not differentiable at $K$ and $L$, to obtain
\begin{align}
e^{-rs} {G}\left(t+s, X_{t+s}^x\right) & ={G}(t, x) \nonumber\\
&\,\,\,+ \int_0^s e^{-ru} \left( -r {G} + {G}_t + \mathbb{L}_X {G} \right) (t+u, X_{t+u}^x) I \{ X_{t+u}\neq K, X_{t+u}\neq L \}du \nonumber\\
&\,\,\,+ \int_0^s e^{-ru} \sigma X_{t+u}^x {G}_x  (t+u, X_{t+u}^x) I \{X_{t+u}\neq K, X_{t+u}\neq L  \}dW_u\nonumber\\
&\,\,\,+\frac{1}{2}  \int_0^s e^{-ru} \left( {G}_x(t+u, L+) - {G}_x(t+u, L-) \right) dl_u^L(X^x)\nonumber\\
&\,\,\,+\frac{1}{2}  \int_0^s e^{-ru} \left( {G}_x(t+u, K+) - {G}_x(t+u, K-) \right) dl_u^K(X^x), \label{COVGLK}
\end{align}
where $l_s^K$ and $l_s^L$ are the local times of $X$ at the level $K$ and $L$ respectively,
and for simplicity, we denote the martingale term as
\begin{align*}
M_s=\int_0^s e^{-ru} \sigma X_{t+u}^x {G}_x  (t+u, X_{t+u}^x) I \{X_{t+u}\neq K, X_{t+u}\neq L  \}dW_u,
\end{align*}
so that $\mathbb{E}_{t, x} \left[ M_s \right]=0$ under the measure $\mathbb{P}_{t, x}$ for each $s\in[0, T-t]$. A fairly easy calculation then shows that,
upon defining $H(t, x) := \left( -r {G} + {G}_t + \mathbb{L}_X {G} \right) (t, x)$ and using Proposition \ref{PDEOZCL},
\begin{align}
H(t, x)=
\begin{cases}
\left(-r K Z-\sigma^2 x^2 Z_x\right)(t, x),&\text{for $x<L<K$},\\
-r K, & \text{for $L\leq x<K$},\\
0, & \text{for $x>K$}. \label{tnnohH}
\end{cases}
\end{align}
By making use of the above computations and taking the expected value with respect to $\mathbb{P}_{t,x}$ in \eqref{COVGLK}, we find that
\begin{align}
&\mathbb{E}_{t, x} \left[e^{-rs} {G}\left(t+s, X_{t+s}\right)\right]= {G}(t, x)  + \mathbb{E}_{t, x} \left[ -rK \int_0^s e^{-ru} I\{ L< X_{t+u}<K \} du \right] \nonumber\\ 
&\,\,\,\, + \mathbb{E}_{t, x} \left[ \int_0^s e^{-ru} \left( -rKZ -\sigma^2 X_{t+u}^2 Z_x \right) (t+u, X_{t+u}) I\{ X_{t+u}<L \}du \right] \nonumber\\
&\,\,\,\,- \frac{1}{2} \mathbb{E}_{t, x} \left[  \int_0^s e^{-ru} (K-L) Z_x(t+u, L-) dl_u^L(X)  \right] + \frac{1}{2} \mathbb{E}_{t, x}\left[ \int_0^s e^{-ru} dl_u^K(X) \right], \label{GKBTC}
\end{align}
from which, it follows that
\begin{Lemma} \label{KBTSS}
All points $(t, x)\in[0, T)\times (K, \infty)$ belong to the continuation set $\mathcal{C}$.
\end{Lemma}
\begin{proof}
The proof follows directly from the observation that stopping on $[0, T)\times (K, \infty)$ gives one null payoff, whereas waiting yields a positive probability to collect a strictly positive payoff, see \cite[Page 558]{DetempleandKitapbayev2021}.
\end{proof}

\begin{rem}
Another message of equation \eqref{GKBTC}, as we shall see shortly, is that the optimal stopping set over $[0, T]$ depends on the relative positions of the level $L$ and of the optimal stopping boundary $B(t)$ of the standard American put option with payoff $G^A(x)=(K-x)^+$. (See \cite{BroadieandDetemple}, \cite{DetempleandKitapbayev2018}.)
\end{rem}

For notational convenience, we let
$V^A(t, x)$, $G^A(x)=(K-x)^+ $ and $B(t)$ denote the value function, the gain function and the optimal stopping boundary for the standard American put option with the same maturity and strike price as the current contract. Then, we know from \cite{PeskirandShiryaev} that
\begin{align*}
&V^A(t, x)>G^A(x) \qquad\text{for $x>B(t)$},\\
&V^A(t, x)=G^A(x) \qquad\text{for $x\leq B(t)$},
\end{align*}
where $B(t)$ is unique and $B(T-)=K$; and thus, its stopping set $\mathcal{D}^A$ and continuation set $\mathcal{C}^A$ equal:
\begin{align*}
&\mathcal{D}^A=\{ (t, x)\in [0, T]\times(0, \infty): x\leq B(t) \},\\
&\mathcal{C}^A=\{ (t, x)\in [0, T]\times(0, \infty): x>B(t)  \}.
\end{align*}

Also we define $t_*\in [0, T)$ as follows:
\begin{align*}
\begin{cases}
t_*=0, &\qquad\text{if $B(0) \geq L$},\\
t_*=B^{-1}(L), &\qquad\text{otherwise},
\end{cases}
\end{align*}
and since the map $t\mapsto B(t)$ is strictly increasing, $t_*$ is uniquely defined (see Figures \ref{TYBC51} and \ref{TYBC52}). Moreover, let us emphasise that $t_*<T$, as $L<K$.

\begin{Lemma} \label{SSINE}
All points $(t, x)\in [t_*, T]\times [L, B(t)] $ belong to the stopping set $\mathcal{D}$.
\end{Lemma}
\begin{proof}
Since $(t, x)\in[t_*, T]\times[L, B(t)]$, it follows from the construction of the stopping set for the standard American put option that
\begin{align*}
V^A(t, x)=G^A(x),
\end{align*}
and from equation \eqref{XtgeqL}, we note that in this set, $G^A(x)={G}(t, x)$. Now, knowing that
\begin{align*}
{V}(t, x)&=\sup_{0\leq \tau \leq T-t} \mathbb{E}_{t, x} \left[ e^{-r\tau} Z\left(t+\tau, X_{t+\tau}\right) (K- X_{t+\tau})^+ \right] \\
&\leq \sup_{0\leq \tau \leq T-t}  \mathbb{E}_{t, x} \left[ e^{-r\tau} (K- X_{t+\tau})^+ \right] = V^A(t, x),
\end{align*}
where the inequality is due to function $Z$ being a probability and consequently, 
\begin{align*}
0=V^A(t, x)-G^A(x)\geq {V}(t, x)-{G}(t, x),
\end{align*}
on the set $\{ (t, x)\in [t_*, T]\times [L, B(t)] \}$, which forces the equation ${V}={G}$ in the sense that by definition, $V - G \geq 0$ and thereby, proving the desired assertion.
\end{proof}

With a little additional work, we can extend Lemma \ref{SSINE} to the following result:

\begin{Lemma} \label{SSINEi}
All points $(t, x)\in [t_*, T]\times (0, B(t)] $ belong to the stopping set $\mathcal{D}$.
\end{Lemma}
\begin{proof}
The proof follows from the argument in \cite[Page 559, (i)]{DetempleandKitapbayev2021} and amounts to showing that $[t_*, T]\times [0, L]$ also belongs to the stopping set $\mathcal{D}$. Suppose, on the contrary, that there exists a point $(t, x)\in[t_*, T]\times(0, L)$ being in set $\mathcal{C}$ such that $V(t, x)>G(t, x)$ by definition. Then, we run the process $X$ starting from $(t, x)\in [t_*, T]\times(0, L)$ and we see  from equation \eqref{GKBTC} that the second, third and fourth terms in its right-hand side are negative, and in order to ensure $V>G$, the process shall continue till it hits level $K$ as the local time term is positive, but $L<K$, indicating the process cannot hit level $K$ before exercising at level $L$ due to Lemma \ref{SSINE}. This implies $V<G$, which is a contradiction. Therefore, it is optimal to stop immediately on this set and the conclusion follows.
\end{proof}

\begin{rem}\label{SSINEii}
Lemmas \ref{SSINE} and \ref{SSINEi} show that the stopping set is not empty and that if $t_*=0$, then all points $(t, x)\in[0, T]\times(0, L]$ belong to the stopping set $\mathcal{D}$, in the sense that 
\begin{align*}
[0, T]\times[L, B(t)] \subset  [0, T]\times (0, B(t)] ,
\end{align*}
such that in $[0, T]\times[L, B(t)]$,  the relation $V^A(t, x)=G^A(x)$ and $G^A(x)={G}(t, x)$ once again entails ${V}={G}$, after which, the similar argument as Lemma \ref{SSINEi} shows $(t, x)\in[0, T]\times(0, L]$ belongs to $\mathcal{D}$.
\end{rem}
\begin{figure}[H]
\centering
\begin{tikzpicture}
\draw[->, ultra thick] node[left]{$0$}(0, 0)-- (5, 0) node[right]{$T$};
\draw[->, ultra thick] (0, 0)--(0, 5) node[right]{$x$};
\draw[name path = A , thick](0, 1.95)..controls (1, 1.96) and (3.5, 2.3).. node[above]{$B(t)$} (5, 4.9);
\draw[dotted] (0, 4.9)-- node[below] {\tiny{$K$}} (5, 4.9);
\draw[name path = B](0, 1.5)--node[below] {$L$}(5, 1.5);
\draw[name path = C, dotted] (5, 0)--(5, 4.9);
\tikzfillbetween[of=A and B]{yellow, opacity=0.5};
\end{tikzpicture}
\begin{tikzpicture}
\draw[name path = D,->, ultra thick] node[left]{$0$}(0, 0)-- (5, 0) node[right]{$T$};
\draw[->, ultra thick] (0, 0)--(0, 5) node[right]{$x$};
\draw[name path = A , thick](0, 1.95)..controls (1, 1.96) and (3.5, 2.3).. node[above]{$B(t)$} (5, 4.9);
\draw[dotted] (0, 4.9)-- node[below] {\tiny{$K$}} (5, 4.9);
\draw[name path = B](0, 1.5)--node[below] {$L$}(5, 1.5);
\draw[orange, dotted, thick] (0, 1.7)--(5, 1.7) node[right]{\tiny{$L+\delta$}};
\draw[name path = C, dotted] (5, 0)--(5, 4.9);
\tikzfillbetween[of=A and D]{yellow, opacity=0.5};
\end{tikzpicture}
\caption{Two sketches to illuminate Lemmas \ref{SSINE} and \ref{SSINEi} for $t_*=0$. The area coloured in yellow is the subset of $\mathcal{D}$ and that $\mathcal{D}^A\subset\mathcal{D}$.}
\label{TYBC51}
\end{figure}
\begin{figure}[H]
\centering
\begin{tikzpicture}
\draw[->, ultra thick] node[left]{$0$}(0, 0)-- (5, 0) node[right]{$T$};
\draw[->, ultra thick] (0, 0)--(0, 5) node[right]{$x$};
\draw[name path = A , thick](0, 1.95)..controls (1, 1.96) and (3.5, 2.3).. node[below]{$B(t)$} (5, 4.9);
\draw[dotted] (0, 4.9)-- node[below] {\tiny{$K$}} (5, 4.9);
\draw[name path = B](0, 2.5)--(5, 2.5)node[above] {$L$};
\draw[name path = C, dotted] (5, 0)--(5, 4.9);
\draw[name path = E, dotted] (2.4,0)node[below]{$t_*$}--(2.4, 2.5);
\fill[green, opacity=0.5, intersection segments={of= A and B, sequence={L2--R2[reverse]}}];
\end{tikzpicture}
\begin{tikzpicture}
\draw[name path = D, ->, ultra thick] node[left]{$0$}(0, 0)--(5, 0) node[right]{$T$};
\draw[->, ultra thick] (0, 0)--(0, 5) node[right]{$x$};
\draw[name path = A , thick](0, 1.95)..controls (1, 1.96) and (3.5, 2.3).. node[below]{$B(t)$} (5, 4.9);
\draw[dotted] (0, 4.9)-- node[below] {\tiny{$K$}} (5, 4.9);
\draw[name path = B](0, 2.5)--(5, 2.5) node[right] {$L$};
\draw[orange, dotted, thick]  (0, 2.8)--(5, 2.8) node[right]{\tiny{$L+\delta$}};
\draw[name path = C, dotted] (5, 0)--(5, 4.9);
\draw[name path =E, dotted] (2.4,0) node[below]{$t_*$}--(2.4, 2.5);
\draw[orange, dotted, thick] (3, 0) node[below]{\tiny{$t_\delta$}}--(3, 2.8);
\fill[green, opacity=0.5, intersection segments={of= A and B, sequence={L2--R2[reverse]}}];
\fill[green, opacity=0.5] (2.4, 0) rectangle (5, 2.5);
\end{tikzpicture}
\caption{Two sketches to illuminate Lemmas \ref{SSINE} and \ref{SSINEi} for $t_*>0$. The area coloured in green is the subset of $\mathcal{D}$.}
\label{TYBC52}
\end{figure}
With Remark \ref{SSINEii} in mind, we now enter the discussions of the properties for sets $\mathcal{C}$ and $\mathcal{D}$ with $B(0)>L$ and $B(0)\leq L$.

\subsubsection{In the Case of $B(0)> L$.} \label{TCASSC5B0GL}

We begin by recalling the definitions of right-, left-, up-, and down-connectedness for the sets $\mathcal{D}$ and $\mathcal{C}$ from \cite[Page 559]{DetempleandKitapbayev2021}.
\begin{Definition}
The set $\mathcal{D}$ is right-connected and the set $\mathcal{C}$ is left-connected if, for $0\leq t_1<t_2<T$ and $x\in(0, \infty)$,
\begin{align*}
{V}(t_1, x) - {G}(t_1, x) \geq  {V}(t_2, x) - {G}(t_2, x) \geq 0.
\end{align*}
The set $\mathcal{D}$ is down-connected if, for $t\in[0, T]$ and $0<y<x$, $(t, x)\in\mathcal{D}$ implies $(t, y)\in\mathcal{D}$; similarly, $\mathcal{C}$ is up-connected if for $t\in[0, T]$ and $0<y<x$, $(t, y)\in\mathcal{C}$ implies $(t, x)\in\mathcal{C}$.
\end{Definition}

\begin{Proposition} \label{SSINEiii}
The stopping set $\mathcal{D}$ is right-connected and the continuation set $\mathcal{C}$ is left-connected.
\end{Proposition}

\begin{proof}
Let $0\leq t_1\leq t_2\leq T$ and consider the stopping time $\tau_L=\inf\{0\leq u \leq T-t_2: X_{t_2+u}^x\leq L\}$ and suppose that $\tau_2$ is the optimal stopping time for ${V}(t_2, x)$ so that $\tau_2\leq \tau_L$, which is a consequence of Remark \ref{SSINEii}. Then, by the time-homogeneous property of GBM,
\begin{align*}
&{V}(t_1, x) - {V}(t_2, x) \geq \mathbb{E} \left[ e^{-r\tau_2} {G}\left( t_1+\tau_2, X_{ \tau_2}^x\right) - e^{-r\tau_2} {G}\left( t_2+\tau_2, X_{ \tau_2}^x\right) \right]\\
&\,\,\,= {G}(t_1, x) - {G}(t_2, x) + \mathbb{E} \left[ \int_0^{\tau_2} e^{-ru} \left( H\left( t_1+u, X_{u}^x \right) - H\left( t_2+u, X_{u}^x \right)  \right) I\{L<X_{u}^x< K \} du \right] \\
&\,\,\, = {G}(t_1, x) - {G}(t_2, x),
\end{align*}
where the second equality follows from $H\left(t_1+u, X_{u}^x\right) = H\left(t_2+u, X_{u}^x\right)=-rK$.
This establishes the following relation:
\begin{align*}
{V}(t_1, x) - {G}(t_1, x) \geq  {V}(t_2, x) - {G}(t_2, x) \geq 0,  
\end{align*}
from which, we see that $(t_1, x)\in\mathcal{D}$ implies $(t_2, x)\in\mathcal{D}$ (i.e., $V(t_1, x) - G(t_1, x) = 0$ implies ${V}(t_2, x) - {G}(t_2, x) =0$) and that $(t_2, x)\in\mathcal{C}$ implies $(t_1, x)\in\mathcal{C}$ (i.e., $V(t_2, x) - G(t_2, x) > 0$ implies ${V}(t_1, x) - {G}(t_1, x) >0$).
\end{proof}

\begin{Proposition}
The stopping set $\mathcal{D}$ is down-connected and the continuation set $\mathcal{C}$ is up-connected.
\end{Proposition}
\begin{proof}
Immediate from the Proposition \ref{SSINEiii} and a similar argument as in Lemma \ref{SSINEi} and \cite[Page 559, (ii)]{DetempleandKitapbayev2021}.
\end{proof}

\begin{rem}
The connectedness of sets $\mathcal{C}$ and $\mathcal{D}$ entails the monotonicity of the free-boundary $b$.
\end{rem}

To prepare for the main theorem presented in Section \ref{TOSPRC5}, we further justify the continuity of the optimal stopping boundary.

\begin{Lemma}\label{BICOLLB}
(i) For all $t\in[0, T]$, $b(t) = B(t)$;
(ii) The optimal stopping boundary $b$ is continuous on $[0, T]$ and $b(T-)=K$.
\end{Lemma}
\begin{proof}
Note that statement (ii) is a direct consequence of statement (i), see \cite[Page 383]{PeskirandShiryaev}. From Lemmas \ref{SSINE} and \ref{SSINEi}, we know that $b\geq B$ for all $t\in[0, T]$ as $B(0) > L$. Suppose on the contrary that $b(t) > B(t)$ for some $t\in[0, T]$. Take $x\in(B(t), b(t))$ for some $t\in[0, T]$ such that $(t, x)\in\mathcal{C}^A$, and $(t, x)\in\mathcal{D}$. Then, we consider the stopping time
\begin{align}
\tau_B = \inf\{s\in[0, T-t]: X_{t+s} \leq B(t+s)\}
\end{align}
such that as $B(t+\tau_B)\geq B(0) > L$ and $Z(t+\tau_B, X_{t+\tau_B}) = 1$
\begin{align}
V(t, x) \geq \mathbb{E}_{t, x}\left[ e^{-r\tau_B}G(t+\tau_B, X_{t+\tau_B}) \right] = \mathbb{E}_{t, x}\left[ e^{-r\tau_B}(K - X_{t+\tau_B})^+ \right] = V^A(t, x). \label{BICOLLBII2}
\end{align}
where the last equality holds true as $\tau_B$ is the optimal stopping time for standard American put option. Since $(t, x)\in\mathcal{C}^A$, it must be $V^{A}(t, x) > G^A(x)$, which, together with \eqref{BICOLLBII2} shows that for $x>B(t)\geq L$,
\begin{align*}
V(t, x) \geq V^A(t, x) > G^A(x) = G(t, x),
\end{align*}
but $(t, x)\in\mathcal{D}$, leading to a contradiction. Hence, $b(t) = B(t)$ for all $t\in[0, T]$.
\end{proof}

We now have sufficient tools to go around the singularity of $Z_x(t, x)$ and establish the continuity of the value function for $t_*=0$.

\begin{Cor}\label{GIUCIZL}
The gain function is uniformly continuous on the state space $[0, T]\times[L, \infty)$.
\end{Cor}
\begin{proof}
See Appendix \ref{ALAPii}.
\end{proof}

\begin{Lemma}\label{VICC51}
The value function $V$ is continuous on $[0, T)\times(0, \infty)$. 
\end{Lemma}
\begin{proof}
Note that for $(t, x)\in[0, T)\times[B(t), \infty)$, $V^A(t, x) = V(t, x)$ in the sense that if we consider the stopping time 
\begin{align}
\tau_B = \inf\{s\in[0, T-t]: X_{t+s} \leq B(t+s)\},
\end{align}
then due to Lemma \ref{BICOLLB} and $B(0) > L$,
\begin{align}
V(t, x) = \mathbb{E}_{t, x}\left[ e^{-r\tau_B}G(t+\tau_B, X_{t+\tau_B}) \right] = \mathbb{E}_{t, x}\left[ e^{-r\tau_B}(K - X_{t+\tau_B})^+ \right] = V^A(t, x). 
\end{align}
As for $(t, x)\in[0, T)\times(0, B(t)]$, $V(t, x) = G(t, x)$. The conclusion follows from \cite[Page 380]{PeskirandShiryaev}, the continuity of $B$ and function $Z$.
\end{proof}

By taking advantage of the above analysis, we obtain the following result:
\begin{Proposition}[The Free-boundary Problem]\label{CORRLGB}
The stopping set and the continuation set are of the forms:
\begin{align}
&\mathcal{C}=\{ (t, x) \in [0, T)\times(0, \infty): x>b(t) \}, \label{CLGB}\\
&\mathcal{D}=\{ (t, x) \in [0, T)\times(0, \infty): x\leq b(t) \} \cup \{ (T, x): x\in(0, \infty)\}, \label{DLGB}
\end{align}
where the map $t\mapsto b(t)$ is increasing.

The \textit{free-boundary} problem is rearranged as follows:
\begin{align}
&{V}_t + \mathbb{L}_X {V} - r {V} = 0 \qquad\qquad \text{for $(t, x)\in\mathcal{C}$}, \label{FBPL1e0}\\
&{V}(t, x)=K-x \qquad\qquad\,\,\,\,\,\,\,\,\text{for $x=b(t)$, instantaneous stopping},\label{FBPL2e0}\\
&{V}_x(t, x) =-1 \qquad\qquad\,\,\,\,\,\,\,\,\,\,\,\,\,\, \text{for $x=b(t)$, smooth-fit},\label{FBPL3e0}\\
&{V}(t, x)>{G}(t, x) \qquad\qquad\,\,\,\,\,\, \text{for $(t, x)\in\mathcal{C}$},\label{FBPL4e0}\\
&{V}(t, x)={G}(t, x) \qquad\qquad\,\,\,\,\,\, \text{for $(t, x)\in\mathcal{D}$}.\label{FBPL5e0}
\end{align}
\end{Proposition}

Just to complete the picture of the free-boundary problem, we now justify equation \eqref{FBPL3e0}.
\begin{proof}
From Lemma \ref{BICOLLB} and $B(0) > L$, we know that $V(t, B(t)) = V^A(t, B(t)) = G^A(B(t)) = G(t, B(t))$ and that $V_x^A(t, B(t)) = G_x^A(t, B(t)) = -1$, see \cite[Page 381]{PeskirandShiryaev}.
\end{proof}

 \begin{rem}
The noteworthy feature of the value function is that it is not smooth in the stopping set $\mathcal{D}$, in particular at level $L$. This fact implies the presence of a local time term in the contract pricing formula, see \cite[Page 7]{DetempleandKitapbayev2018} and \cite{Qiu2016} for similar results.
 \end{rem}

\subsubsection{In the Case of $B(0)\leq L$.} \label{TCASSC5B0LL}

\begin{Proposition}\label{PLB01}
On the state space $[0, T]\times[L, \infty)$, the stopping set $\mathcal{D}$ is right-connected and the continuation set $\mathcal{C}$ is left-connected.
\end{Proposition}

\begin{proof}
Since, on the state space $[0, T]\times[L, \infty)$, the gain function $G(t, x)=(K-x)^+$ by equation \eqref{XtgeqL}, the map $t\mapsto V(t, x)$ is decreasing and consequently the map $t\mapsto V(t, x)-G(t, x)$ is also decreasing. That is, for any $t_1, t_2\in[0, T]$, if $t_1<t_2$, then
\begin{align*}
V(t_1, x)-G(t_1, x) \geq V(t_2, x)-G(t_2, x),
\end{align*}
such that $(t_1, x)\in\mathcal{D}$ implies $(t_2, x)\in\mathcal{D}$ and $(t_2, x)\in\mathcal{C}$ implies $(t_1, x)\in\mathcal{C}$, which is precisely the claim.
\end{proof}

To fill the last slot of the diagram, we shall show that the stopping set $\mathcal{D}$ is not empty on state space $[0, t_*]\times(0, L]$, but first we present some important facts:
\begin{Lemma}\label{TMOGIX} 
(i) The gain function $G$ is uniformly continuous on the state space $[0, t_*]\times(0, L]$;

(ii) The map $t\mapsto\max\limits_{x\in(0, L]} G(t, x)$ is continuous on $[0, t_*]$.

(iii) For each $t\in[0, t_*]$, $\max\limits_{x\in(0, L]} G(t, x)$ is reached at a unique point $x^*(t)$, in particular,
\begin{align}
\max\limits_{x\in(0, L]}G(t, x)=G(t, x^*(t)),
\end{align}
the map $t\mapsto x^*(t)$ is continuous on $[0, t_*]$.
\end{Lemma}

Let $t_{x^*} = \inf\{ t\in[0, t_*]: x^*(t) = L\}$ with $t_{x^*} = t_*$ for the infimum of the empty set.

\begin{rem}\label{TROX}
Knowing from derivative computation that for each fixed $t$, the maximum of $G$ attained on $(0, L]$ is also the global maximum over the domain $(0, \infty)$, we can write 
\[\max\limits_{x\in(0, \infty)}G(t, x)=\max\limits_{x\in(0, L]}G(t, x).\]
As a matter of fact, we can select parameters $r, \sigma, L, K, T$ such that $t\mapsto x^*(t)$ can be non-decreasing (see Figure \ref{themapofxstar}).
\end{rem}

\begin{proof}[Proof of Lemma \ref{TMOGIX} ]
See Appendix \ref{ALAPiii} and Lemma \ref{xisugxxin}.
\end{proof}


What Lemma \ref{TMOGIX} tells us is that:
\begin{Lemma} \label{TMOGIXIND}
All points $(t, x)\in[0, t_*]\times (0, x^*(t)]$ belong to the stopping set $\mathcal{D}$.
\end{Lemma}

\begin{proof}
According to the definition of $x^*$, Remark \ref{TROX} and the fact that $t\mapsto G(t, x)$ is decreasing,
\begin{align*}
G(t, x^*(t))\geq \max_{x\in(0, \infty)} G(t+\tau, x ) \geq G(t+\tau, x),
\end{align*}
and by domination, taking expectation with respect to the measure $\mathbb{P}_{t, x^*(t)}$,
\begin{align*}
G(t, x^*(t))\geq \mathbb{E}_{t, x^*(t)} \left[ e^{-r\tau} G(t+\tau, X_{t+\tau} ) \right],
\end{align*}
so that $G(t, x^*(t))\geq V(t, x^*(t))$, but $V\geq G$ forces the equation $G(t, x^*(t)) = V(t, x^*(t))$, and thus showing that for all $t\in[0, t_*]$, $\{(t, x^*(t))\}\in\mathcal{D}$ (that is, the curve $t\mapsto x^*(t)$ lies entirely in $\mathcal{D}$ for $t\in[0, t_*]$). Furthermore, due to Lemma \ref{TMOGIX}, an analogous argument as Lemma \ref{SSINEi} proves the conclusion.
\end{proof}

\begin{figure}[H]
\centering
    \subfloat[\centering]{{\includegraphics[width=7cm]{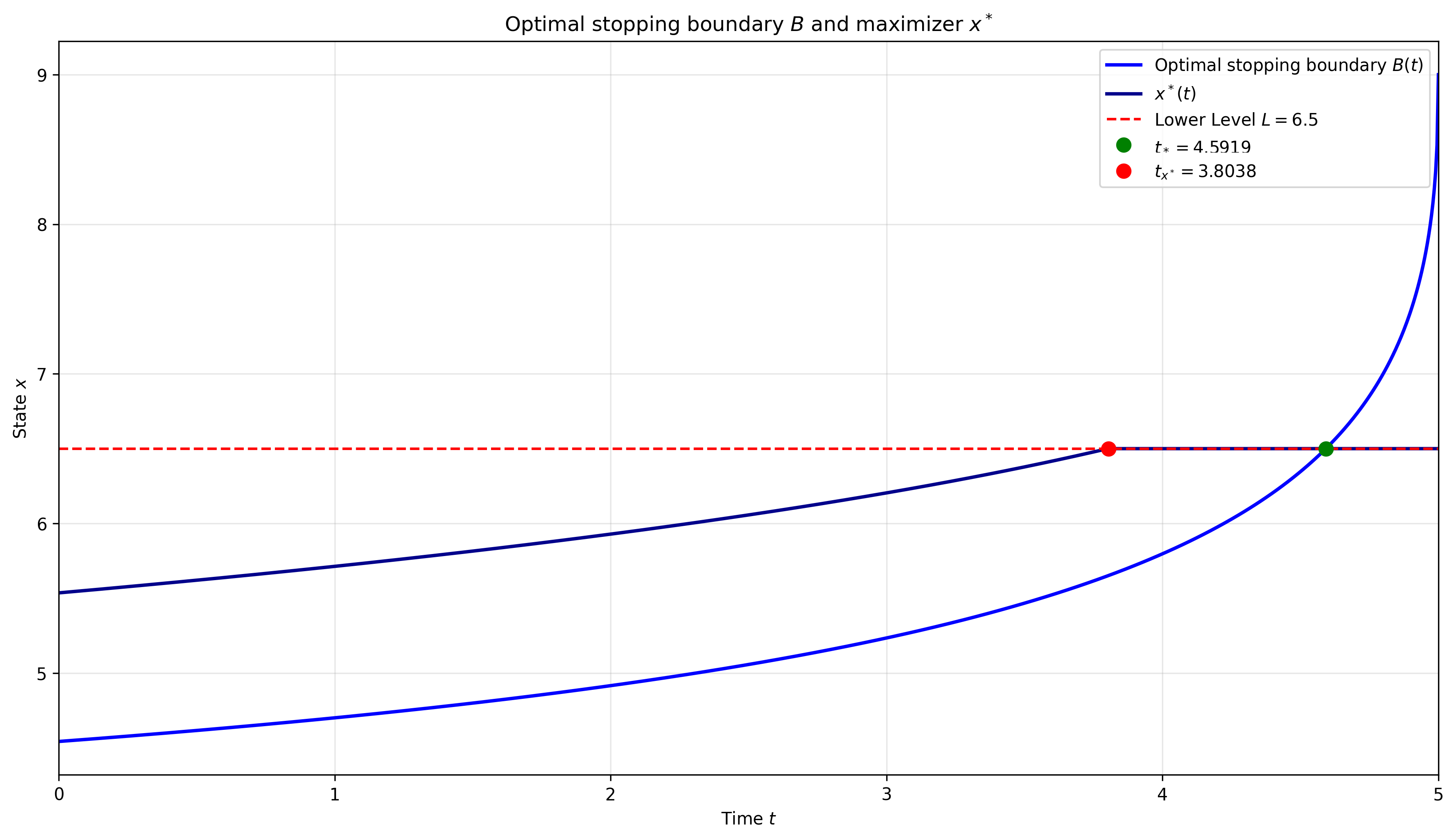} \label{themapofxstar}}}%
    \qquad
    \subfloat[\centering]{{\includegraphics[width=7cm]{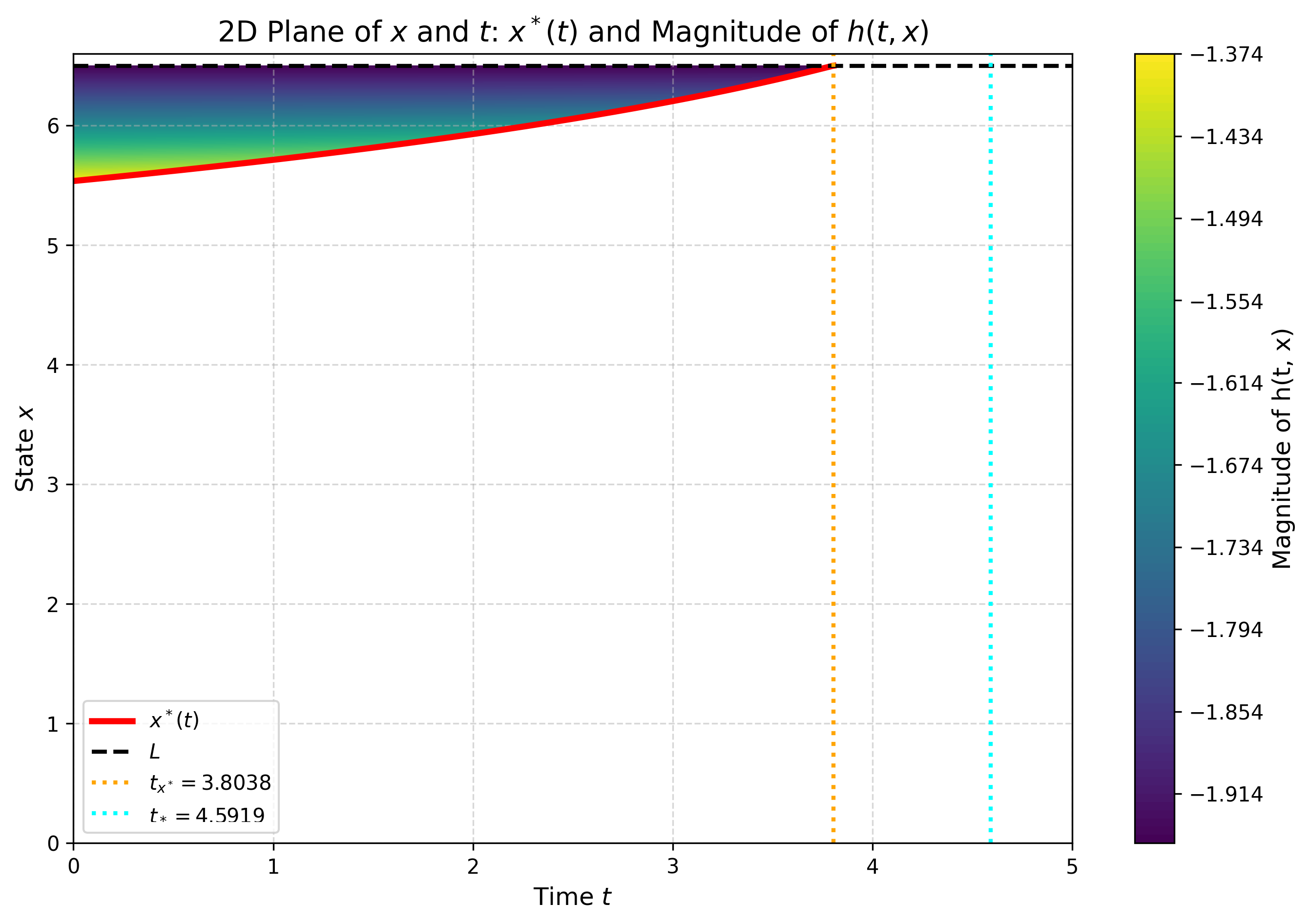} \label{themapofxstar2} }}%
    \caption{With chosen parameters $r=0.03$, $\sigma=0.3$, $L=6.5$, $K=9$, $T=5$, figure (a) displays the maps $t\mapsto x^*(t)$; figure (b)displays $h(t, x)$ in \eqref{CAM1}}%
   \end{figure}

\begin{figure}[H]
\centering
\begin{tikzpicture}
\draw[name path = D, ->, ultra thick] node[left]{$0$}(0, 0)--(5, 0) node[right]{$T$};
\draw[->, ultra thick] (0, 0)--(0, 5) node[right]{$x$};
\draw[name path = A , thick](0, 1.95).. node[below]{$B(t)$} controls (1, 1.96) and (3.5, 2.3)..(5, 4.9);
\draw[dotted] (0, 4.9)--(5, 4.9) node[right] {\tiny{$K$}} ;
\draw[name path = B](0, 3.5)--(5, 3.5) node[right] {$L$};
\draw[name path = C, dotted] (5, 0)--(5, 4.9);
\draw[name path =E, dotted] (3.9,0) node[below]{$t_*$}--(3.9, 3.5);
\draw[thick, orange, dotted] (0, 3.8)--(5, 3.8) node[right]{\tiny{$L+\delta$}};
\draw[thick, orange, dotted](4.2, 0) node[below]{\tiny{$t_\delta$}}--(4.2, 3.8);
\draw[name path = F, blue, dashed](0, 2.18) node[above]{\textcolor{black}{$x^*(t)$}}..controls (1, 2.2) and (3.5, 2.5)..(3.9, 3);
\draw[name path=G, dashed](0, 0)--(3.9, 0);
\fill[green, opacity=0.5, intersection segments={of= A and B, sequence={L2--R2[reverse]}}];
\fill[green, opacity=0.5] (3.9, 0) rectangle (5, 3.5);
\tikzfillbetween[of=F and G]{green, opacity=0.5};
\end{tikzpicture}
\begin{tikzpicture}
\draw[name path = D, ->, ultra thick] node[left]{$0$}(0, 0)--(5, 0) node[right]{$T$};
\draw[->, ultra thick] (0, 0)--(0, 5) node[right]{$x$};
\draw[name path = A , thick](0, 1.95)..controls (1, 1.96) and (3.5, 2.3).. node[below]{$B(t)$} (5, 4.9);
\draw[dotted] (0, 4.9)-- (5, 4.9) node[right] {\tiny{$K$}};
\draw[name path = B](0, 3.5)--(5, 3.5) node[right] {$L$};
\draw[name path = C, dotted] (5, 0)--(5, 4.9);
\draw[name path =E, dotted] (3.9,0) node[below]{$t_*$}--(3.9, 3.5);
\draw[thick, orange, dotted] (0, 3.8)--(5, 3.8) node[right]{\tiny{$L+\delta$}};
\draw[thick, orange, dotted](4.2, 0) node[below]{\tiny{$t_\delta$}}--(4.2, 3.8);
\draw[name path = F, blue, dashed](0, 2.2) node[above]{\textcolor{black}{$x^*(t)$}}..controls (1, 2.25) and (2, 2.5)..(3, 3.5)--(3.9, 3.5);
\draw[name path=G, dashed](0, 0)--(3.9, 0);
\fill[green, opacity=0.5, intersection segments={of= A and B, sequence={L2--R2[reverse]}}];
\fill[green, opacity=0.5] (3.9, 0) rectangle (5, 3.5);
\tikzfillbetween[of=F and G]{green, opacity=0.5};
\end{tikzpicture}
\caption{Two sketches to illuminate Lemma \ref{TMOGIXIND} for $t_*>0$. The areas coloured in green is $\mathcal{D}_0\subset\mathcal{D}$ and the blue dashed lines are plotting function $x^*:[0, t_*]\mapsto(0, L]$. The right figure is to illuminate the possible case stated in Remark \ref{TCOBC5}.}
\label{TTSTILC5}
\end{figure}

\begin{rem}
It has been highlighted in multiple pieces of literature (if not all of them, such as \cite[Page 113]{Kitapbayev2014}) that if either the maps $t\mapsto H(t, x)$ or $x\mapsto H(t, x)$ fails to be monotone (the latter monotonicity is used to establish the up-connectedness of the sets), we will not be able to compare two integrals with different signs and prove most of the properties possessed by the optimal stopping boundary, which is probably the main drawback of such analytical method but opens the door to many interesting research possibilities; this is, unfortunately, the story of $K\leq L$. By narrowing down the scale of the continuation set (the idea behind this is that, in the American put option pricing problem, the lower bound for the continuation set when $T<\infty$ is the optimal stopping level obtained from pricing the perpetual American put option), we can make an additional assumption to further develop some insights towards the shape of sets $\mathcal{C}$ and $\mathcal{D}$.
\end{rem}

\begin{Assumption} \label{TMOHTC5}
A set of parameters $(r, \sigma, L, K, T)$ is selected such that for $(t, x)\in [0, t_*] \times [x^*(t), L]$,
\begin{align}
h(t, x) : = rK\sigma^{-1} \log{\frac{L}{x}}  - \sigma x + \left( \log{\frac{L}{x}}\right)^2 (T-t)^{-1} \sigma^{-1} x -  x \left( \frac{r}{\sigma} - \frac{\sigma}{2} \right) \log{\frac{L}{x}} \leq 0. \label{CAM1}
\end{align}
See Figure \ref{themapofxstar2}.
\end{Assumption}

\begin{rem}[On the nature of Assumption \ref{TMOHTC5}]
By Lemmas \ref{TMOGIX} and \ref{xisugxxin}, for each $t\in[0, t_*]$, the maximiser $x^*(t)$ of $G(t, \cdot)$ on $(0, L]$ exists uniquely and depends continuously on $t$. Consequently, the region $[x^*(t), L]$ on which $h$ is required to be non-positive in Assumption \ref{TMOHTC5} is well defined. The  assumption is therefore a well-posed condition on parameters $(r, \sigma, L, K, T)$. However, it is not an explicit condition as $x^*(t)$ is characterised only implicitly by $G_x(t, x^*(t)) = 0$ for $t\in[0, t_{x^*})$ and $x^*(t) = L$ for $t\in[t_{x^*}, t_*]$, so the admissible parameter set must be determined by first computing $x^*(t)$ numerically and then checking the sign of $h$ on the resulting region. This is how we obtain Figure \ref{themapofxstar2}. We do not claim that the assumption holds for all parameters; however, it is a verifiable hypothesis, and Theorem \ref{MTLGB} (ii) should be read as conditional on it. 
\end{rem}

Assumption \ref{TMOHTC5} ensures the monotonicity of $x^*(t)$ on $[0, t_*]$:
\begin{Lemma}\label{MNTOOXSTAR}
Suppose that Assumption \ref{TMOHTC5} holds true. Then, the map $t\mapsto x^*(t)$ is non-decreasing on $[0, t_*]$.
\end{Lemma}
\begin{proof}
For $t\in[0, t_{x^*})$, we know that $x^*(t)<L$ and that by the definition of $x^*(t)$, $G_{x}(t, x^*(t)) = 0, G_{xx}(t, x^*(t)) < 0$. By taking partial derivative, the implicit function theorem shows that
\begin{align}
\frac{\partial^2}{\partial x \partial t} G(t, x^*(t)) + \frac{\partial^2}{\partial x^2 } G(t, x^*(t)) \frac{d x^*(t)}{dt} = 0,
\end{align}
which is 
\begin{align}
 \frac{d x^*(t)}{dt} = -\frac{\frac{\partial^2}{\partial x \partial t} G(t, x^*(t)) }{\frac{\partial^2}{\partial x^2} G(t, x^*(t)) }.
\end{align}
Then, by Assumption \ref{TMOHTC5}, we have for $t\in[0, t_*]$,
\begin{align*}
\left( -rK \frac{\partial}{\partial t}Z - \sigma^2 (x^*(t))^2 \frac{\partial^2}{\partial t \partial x}Z \right)(t, x^*(t)) = \sqrt{\frac{1}{2\pi}} e^{-\frac{d_1^2}{2}} (T-t)^{-\frac{3}{2}} h(t, x^*(t)) \leq 0,
\end{align*}
which, together with Lemma \ref{MAPOZwrttax}, entails that 
\begin{align}
Z_{xt}(t, x^*(t)) \geq -\frac{rK Z_t(t, x^*(t))}{\sigma^2 (x^*(t))^2} \geq 0
\end{align}
and thereby, proving $G_{xt}(t, x^*(t))\geq 0$. Lemma \ref{xisugxxin} statement (iii) suggests that $ \frac{d x^*(t)}{dt} \geq 0$, and thereby, proving the non-decreasing property of $x^*$ on $[0, t_{x^*})$. Finally, note that once $x^*(t)$ reaches $L$ at $t=t_{x^*}$, it remains at $L$, that is $x^*(t) = L$ for $t\in[t_{x^*}, t_*]$; which, together with the continuity of $x^*$, proves that $t\mapsto x^*(t)$ is non-decreasing on $[0, t_*]$.
\end{proof}

\begin{Proposition} \label{TMOHTCp}
If Assumption \ref{TMOHTC5} holds, then for each fixed $x$, the map $t\mapsto H(t, x)$ is decreasing on the set $\{t: (t, x) \in \mathcal{S}\}$,
where
\begin{align*}
\mathcal{S}: &= \{(t, x):t\in[0, t_{x^*}), x\in[x^*(t), L)\} \cup \{(t, x): t\in[0, t_*], x\geq L \}\\
&\qquad\,\,\,\,\,\,\cup \{ (t, x): t\in[t_*, T], x\geq B(t) \} \cup \{(t_*, x): x\in[x^*(t_*), L)\}.
\end{align*}
\end{Proposition}
\begin{proof}
Note that for $(t, x)\in[0, T]\times[L, \infty)$, by the definition of function $H$ in \eqref{tnnohH}, $\frac{\partial}{\partial t}H(t, x) = 0$ and as a result, the map $t\mapsto H(t, x)$ is decreasing on the set $\{t: (t, x)\in\mathcal{S}_1\}$ with
\begin{align*}
\mathcal{S}_1 : =\{(t, x): t\in[t_*, T], x\geq B(t)\} \cup \{(t, x): t\in[0, t_*], x\geq L\}.
\end{align*}

For $(t, x)\in[0, t_{x^*}) \times [x^*(t), L)$, we have by \eqref{tnnohH} and Assumption \ref{TMOHTC5},
\begin{align*}
&\frac{\partial}{\partial t}H(t, x) = \left( -rK \frac{\partial}{\partial t}Z - \sigma^2 x^2 \frac{\partial^2}{\partial t \partial x}Z \right)(t, x) = \sqrt{\frac{1}{2\pi}} e^{-\frac{d_1^2}{2}} (T-t)^{-\frac{3}{2}} h(t, x) \leq 0,
\end{align*} 
after which, the decreasing property follows on the set $\{t:(t, x) \in\mathcal{S}_2\}$ with
\begin{align*}
\mathcal{S}_2 : =\{(t, x): t\in[0, t_{x^*}), x\in [x^*(t), L)\} \cup  \{(t_*, x): x\in[x^*(t_*), L)\}.
\end{align*}
Then, by combining the sets $\mathcal{S}_1$ and ${\mathcal{S}}_2$, the conclusion follows.
\end{proof}

By Lemma \ref{TMOGIXIND}, the continuity and monotonicity of $x^*$ in Lemma \ref{TMOGIX} and Lemma \ref{MNTOOXSTAR}, we have $\mathcal{D}_0\subseteq\mathcal{D}$, where
\begin{align}
\mathcal{D}_0 & := \{(t, x):t\in[0, t_{x^*}], x\in(0, x^*(t)]\} \cup \{(t, x): t\in[t_{x^*}, t_*], x \in (0, L] \}\\
&\qquad\,\,\,\,\,\,\cup \{ (t, x): t\in[t_*, T], x\in (0, B(t)] \} \cup \{(t_*, x): x\in[x^*(t_*), L]\}.
\end{align}
Let $s:[0, T]\mapsto[x^*(0), K]$ be the increasing boundary of $\mathcal{D}_0$.

\begin{Proposition}\label{TMOVMGL}
Suppose Assumption \ref{TMOHTC5} holds true. The map $t\mapsto V(t,x) - G(t,x)$ is decreasing on $[0, T]$.
\end{Proposition}
\begin{proof}
Let $0\leq t_1 < t_2 \leq T$ and consider the stopping times:
\begin{align*}
\tau_S=\inf\{ 0 \leq u \leq T-t_2: X_u^x\leq s(t_2+u) \} \wedge(T-t_2),
\end{align*}
and let $\tau_2$ be the optimal stopping time for $V(t_2, x)$ such that $t_1+\tau_2<t_2+\tau_2\leq t_2 + \tau_S$ and thereby $X_{\tau_2}^x\geq s(t_2+\tau_2)\geq s(t_1+\tau_2)$ a.s due to the fact that $\mathcal{D}_0\subseteq\mathcal{D}$ and the boundary of $\mathcal{D}_0$ is non-decreasing. Then, we have
\begin{align*}
&V(t_1, x) - V(t_2, x)  \geq \mathbb{E} \left[ e^{-r\tau_2} G\left(t_1+\tau_2, X_{\tau_2}^x\right) \right] - \mathbb{E} \left[ e^{-r\tau_2} G\left(t_2+\tau_2, X_{\tau_2}^x\right) \right] \\
&= G(t_1, x)-G(t_2, x) \\
&\,\,\,+ \mathbb{E} \left[ \int_0^{\tau_2} e^{-ru} \left(H \left( t_1+u, X_u^x \right) - H(t_2+u, X_u^x)\right) I\{X_u^x<L\}du \right]\\
&\,\,\,+ \mathbb{E} \left[ \frac{1}{2}\int_0^{\tau_2} e^{-ru} (K-L) \left(Z_x(t_2+u, L-)- Z_x(t_1+u, L-)\right) dl_u^L(X^x) \right] \geq G(t_1, x) - G(t_2, x),
\end{align*}
where the second inequality holds true via the decreasing map $t\mapsto H(t, x)$ on set $\{ t: (t, x)\in[0, t_{x^*})\times[x^*(t), L) \}\cup \{ t: (t, x) \in \{t_*\}\times[x^*(t), L) \}$ and $t\mapsto Z_x(t, L-)$ is increasing (see \ref{APFC5}.1). In particular, we arrive at the inequality:
\begin{align*}
V(t_1, x)-G(t_1, x) \geq V(t_2, x)- G(t_2, x),
\end{align*}
which is precisely the claim.
\end{proof}

The usefulness of Proposition \ref{TMOVMGL} is well demonstrated by the following results:
\begin{Proposition} \label{TDOTBC5}
(i) The stopping set $\mathcal{D}$ is right-connected, i.e. increasing w.r.t time $t$;

(ii) The continuation set $\mathcal{C}$ is left-connected, i.e. decreasing w.r.t time $t$;

(iii) The stopping set $\mathcal{D}$ is down-connected;

(iv) The continuation set $\mathcal{C}$ is up-connected.
\end{Proposition}
\begin{proof}
Assertions (i) and (ii) are immediate from Proposition \ref{TMOVMGL} and assertions (iii) and (iv) follow from the similar argument as Lemma \ref{SSINEi}.
\end{proof}

\begin{rem}
Once again, because the gain function is not smooth in the stopping set $\mathcal{D}$ at level $L$, the presence of the local-time term in the pricing formula should not come to us as a surprise; above all, the twist comes when the optimal stopping boundary crosses or remains for some time at $L$ (consider the situation for $B(t)<L<K$ and $x^*(t)=L$ for all $t\in[0, t_*]$, then $\{ (t, x) \in [0, T]\times[0, L] \}\subset \mathcal{D}$), the \textit{smooth-fit} condition fails to hold. See \cite[Page 19]{DetempleandKitapbayev2018} for similar treatment of such problem.
\end{rem}

\begin{rem}
The minimal conditions under which the \textit{smooth-fit} condition can hold in greater generality are the \textit{regularity} of the diffusion process $X$ and the \textit{differentiability} of the gain function $G$. For further contribution and examples, we refer to \cite[Page 155]{PeskirandShiryaev}.
\end{rem}

To further construct the continuity of the value function, we need a definition first:
\begin{Definition}
Let times $t_b$ and $t^b$ be defined as follows:
\begin{align*}
\begin{cases}
b(t)<L,& \text{for $t\in[0, t_b)$,} \\
b(t)=L,& \text{for $ t\in[t_b, t^b]$,}\\
b(t)>L,& \text{for $t\in(t^b, T]$},
\end{cases}
\end{align*}
and if $b(t)\geq L$ for all $t\in[0, T]$, we set $t_b=0$. Note also that if $t_b=t^b$, then $t^b<T$.
\end{Definition}


\begin{Cor}\label{GIUCIZI}
The gain function is uniformly continuous on the state space $[0, t_*]\times(0, \infty)$.
\end{Cor}
\begin{proof}
Immediate from Corollary \ref{GIUCIZL} and Lemma \ref{TMOGIX}.
\end{proof}

\begin{Lemma} \label{VICC552}
(i) For $t\in[t_*, T]$, $b(t) = B(t)$. 
(ii) For $(t, x)\in[t_*, T]\times[B(t), \infty)$, $V(t, x) = V^A(t, x)$. 
\end{Lemma}
\begin{proof}
From Lemmas \ref{SSINE} and \ref{SSINEi}, we know that $b(t) \geq B(t)$ for $t\in[t_*, T]$. Suppose on the contrary that $b(t) > B(t)$ for some $t\in[t_*, T]$. Fix $x\in (B(t), b(t))$ such that $(t, x) \in\mathcal{D}$ and $(t, x)\in\mathcal{C}^A$. Then, we consider stopping time $\tau_B = \inf\{s\in[0, T-t]: X_{t+s}^x \leq B(t+s)\}$ such that
\begin{align*}
V(t, x) \geq \mathbb{E}_{t, x}\left[ e^{-r\tau_B}G(t+\tau_B, X_{t+\tau_B}) \right] = \mathbb{E}_{t, x}\left[ e^{-r\tau_B}(K - X_{t+\tau_B})^+ \right] = V^A(t, x) > G^A(x),
\end{align*}
where the first inequality follows from the definition of the value function in \eqref{OSPFLR} and equalities are due to $X_{t+\tau_B}^x = B(t+\tau_B)\geq B(t_*)\geq L$ and the optimality of $\tau_B$ for standard American put option, the last inequality follows as $(t, x)\in \mathcal{C}^A$. Since $G^A(x) = G(t, x)$ as $x>B(t) \geq L$, it reaches a contradiction, which is $V(t, x) > G(t, x)$ for $(t, x)\in\mathcal{D}$. We therefore conclude statement (i). As for statement (ii), it follows the same argument from the proof of Lemma \ref{VICC51} and statement (i). 
\end{proof}

Lemma \ref{VICC552} entails that $t_b\leq t_{x^*}$ and $t^b = t_*$.

\begin{Cor} \label{ZIABSOC}
For $s\in[0, t_\delta]$ with $t_\delta\in[t_*, T)$ and $0<y_1<y_2$, 
\begin{align}
0\leq Z(s, y_2) - Z(s, y_1) = \int_{y_1}^{y_2} Z_x(s, y) dy\leq C_2(y_2 - y_1),
\end{align}
where $C_2$ is defined in \eqref{BODERX11}.
\end{Cor}
\begin{proof}
By Lemma \ref{MAPOZwrttax}, the first inequality follows. Note that $Z(s, \cdot)$ is continuous at $L$ ($Z=1, Z_x = 0, Z_t = 0$ for $x\in[L, \infty))$ and $C^1$ on $(0, L)\cup(L, \infty)$, and hence, it is absolutely continuous, and the point $L$ has measure $0$, which concludes the desired statement.
\end{proof}

\begin{Lemma}\label{VICC52}
Suppose Assumption \ref{TMOHTC5} holds true. The value function is continuous on $[0, T)\times(0, \infty)$. 
\end{Lemma}
\begin{proof}
(a) The value function is continuous on $[t_*, T)\times(0, \infty)$. The continuity of $V$ follows the continuity of $V^A$ on $[t_*, T]\times[B(t), \infty)$ by Lemma \ref{VICC552}. As for $(t, x)\in[t_*, T)\times(0, B(t)]$, we know $V(t, x) = G(t, x)$, and statement (a) follows.

(b) Its continuity on $[0, t_*]\times(0, x^*(t)]$ once again follows the continuity of the gain function due to the definition of the stopping set $\mathcal{D}$. 

(c) Next, we show the continuity of $V$ on state space $[0, t_*]\times[x^*(t), \infty)$.

(i) the map $t\mapsto V(t, x)$ is continuous for each fixed $x\in[x^*(0), \infty)$.
Take any $t_1<t_2$ in $[0, t_*]$ such that $0<|t_1-t_2| < t_\delta - t_*$ where $t_\delta$ is chosen such that $B(t_\delta) = L+\delta$ for $0 <\delta < K-L$.
Let $\tau_1$ be the optimal stopping time for $(t_1, x)$ and set $\tau_2 = \tau_1\wedge (T-t_2)$ such that
\begin{align}
&0\leq V(t_1, x)-V(t_2, x) \nonumber\\
&\leq   \mathbb{E} \left[ e^{-r\tau_1} G\left(t_1+\tau_1, X_{t_1+\tau_1}^x\right)\right] - \mathbb{E}\left[ e^{-r\tau_2} G\left(t_2+\tau_2, X_{t_2+\tau_2}^x\right)\right] \nonumber\\
&\leq \mathbb{E}\left[ e^{-r\tau_2} Z(t_1 + \tau_1, X_{\tau_1}^x) \left( X_{\tau_2}^{x} - X_{\tau_1}^{x} \right)^+ \right] \nonumber\\
&\qquad + \mathbb{E} \left[e^{-r\tau_2} \left(K-X^x_{\tau_2}\right)^+ \left(  Z\left( t_1+\tau_1, X^x_{\tau_1} \right) -  Z\left( t_2+\tau_2, X^x_{\tau_2} \right) \right) \right], \nonumber\\
& = \mathbb{E}\left[ e^{-r\tau_2} Z(t_1 + \tau_1, X_{\tau_1}^x) \left( X_{\tau_2}^{x} - X_{\tau_1}^{x} \right)^+ \right] \nonumber\\
&\qquad + \mathbb{E} \left[e^{-r\tau_1} \left(K-X^x_{\tau_2}\right)^+ I\{X_{\tau_1}^x < L\} \left(  Z\left( t_1+\tau_1, X^x_{\tau_1} \right) -  Z\left( t_2+\tau_2, X^x_{\tau_2} \right) \right) \right] \nonumber\\
&\qquad + \mathbb{E} \left[e^{-r\tau_1} \left(K-X^x_{\tau_2}\right)^+ I\{X_{\tau_1}^x \geq L\} \left(  Z\left( t_1+\tau_1, X^x_{\tau_1} \right) -  Z\left( t_2+\tau_2, X^x_{\tau_2} \right) \right) \right] \nonumber\\
& = \mathbb{E}\left[ e^{-r\tau_2} Z(t_1 + \tau_1, X_{\tau_1}^x) \left( X_{\tau_2}^{x} - X_{\tau_1}^{x} \right)^+ \right] \nonumber\\
& \qquad +  \mathbb{E} \left[e^{-r\tau_1} \left(K-X^x_{\tau_2}\right)^+ I\{X_{\tau_1}^x < L\}   Z_t\left( t_3, X^x_{\tau_1} \right) (t_1 - t_2) \right] \nonumber \\
&\leq  \mathbb{E}\left[ e^{-r\tau_2} Z(t_1 + \tau_1, X_{\tau_1}^x) \left( X_{\tau_2}^{x} - X_{\tau_1}^{x} \right)^+ \right] + C_3 K  (t_2 - t_1)  \label{VCLME}
\end{align}
where $t_3\in[t_1+\tau_1, t_2 + \tau_2]$. In \eqref{VCLME}, the first inequality is because of the map $t\mapsto V(t, x)$ being decreasing, the third inequality holds via
\begin{align}
(K-y)^+ - (K-z)^+ \leq (z-y)^+  \qquad \text{for $y, z\in \mathbb{R}$}. \label{IEQVCL}
\end{align}
Additionally, the last equality in \eqref{VCLME} holds true as (1) on the case of $\{X_{\tau_1}^x < L\}$, $I\{X_{\tau_1}^x < L\} = I\{t_1+\tau_1 < t_{x^*} \leq t_*\}$: if $t_1 + \tau_1 < t_*$, then by $t_2-t_1\leq t_\delta - t_*$, $t_2+\tau_2\leq t_2 + \tau_1 < t_\delta < T$ and $\tau_2 < T-t_2$ (i.e. $\tau_2 = \tau_1$) and hence, the bound $C_3$ in \eqref{BoundforZt} is applicable; (2) on the case of  $\{X_{\tau_1}^x \geq L\}$: if $\tau_1 \leq T-t_2$ and $X_{\tau_1}^x\geq L$, then $X_{\tau_1}^x = X_{\tau_2}^x\geq L$ and $Z(t_1+\tau_1, X_{\tau_1}^x) = Z(t_2+\tau_2, X_{\tau_2}^x) = 1$; if $\tau_1 > T-t_2$ and $X_{\tau_1}^x\geq L$, then by definition of $\tau_1$ (being the first enter time of the stopping set), we have for $u\in[0, \tau_1)$, $X_u^{x} \geq b(t_1 + u)$ and let $u = T-t_2$, it must be $X_{T-t_2}^{x} \geq b(T-(t_2 - t_1))$, after which, we recall that $t_2-t_1 < t_\delta - t_* \leq T-t_*$ and by the monotonicity of $b$, it follows that $b(T-(t_2-t_1)) \geq b(T-T+t_*) = b(t_*) = L$ and thereby, proving that $X_{T-t_2}^{x} \geq L$ and $Z(t_1+\tau_1, X_{\tau_1}^x) = Z(t_2+\tau_2, X_{\tau_2}^x) = 1$.

By letting $t_1\to t_2$, using $\tau_1 \to  \tau_2$ and the dominated convergence theorem allows us to conclude that 
\begin{align*}
V(t_1, x) - V(t_2, x) \to 0,
\end{align*}
and assertion (i) follows.

(ii) The map $x\mapsto V(t, x)$ is continuous uniformly in $t\in[0, t_*]$. 
Let $\tau_2$ be optimal for $V(t, x_2)$ so that
\begin{align}
& V(t, x_2) - V(t, x_1)\nonumber\\
&\leq \mathbb{E}\left[ e^{-r\tau_2} G\left(t+\tau_2, X_{\tau_2}^{x_2}\right) \right] - \mathbb{E} \left[ e^{-r\tau_2} G\left(t+\tau_2, X_{\tau_2}^{x_1}\right) \right] \nonumber\\
&\leq \mathbb{E}\Big[ e^{-r\tau_2}\left( G\left(t+\tau_2, X_{\tau_2}^{x_2}\right) - G\left(t+\tau_2, X_{\tau_2}^{x_1}\right) \right) \Big] \nonumber \\
& = \mathbb{E} \left[ e^{-r\tau_2} \left( \left( K- X_{\tau_2}^{x_2} \right)^+ - \left( K- X_{\tau_2}^{x_1} \right)^+  \right) Z\left(t+\tau_2, X_{\tau_2}^{x_2}\right) \right] \nonumber\\
&\qquad + \mathbb{E} \left[ e^{-r\tau_2}  \left( K- X_{\tau_2}^{x_1} \right)^+ \left( Z\left( t+\tau_2, X_{\tau_2}^{x_2}  \right) - Z\left( t+\tau_2, X_{\tau_2}^{x_1}  \right) \right) \right] \nonumber \\
&\leq K \mathbb{E}\Big[ e^{-r\tau_2} \left(  Z\left( t+\tau_2, X_{\tau_2}^{x_2}  \right) - Z\left( t+\tau_2, X_{\tau_2}^{x_1}  \right) \right) \Big]\nonumber\\
& \leq  K \mathbb{E}\Big[ e^{-r\tau_2} I\{ X_{\tau_2}^{x_2} > L+\delta\}  \left(  Z\left( t+\tau_2, X_{\tau_2}^{x_2}  \right) - Z\left( t+\tau_2, X_{\tau_2}^{x_1}  \right) \right) \Big] \nonumber\\
& \qquad +  K \mathbb{E}\Big[ e^{-r\tau_2} I\{ X_{\tau_2}^{x_2} \leq L+\delta\} \left(  Z\left( t+\tau_2, X_{\tau_2}^{x_2}  \right) - Z\left( t+\tau_2, X_{\tau_2}^{x_1}  \right) \right) \Big], \label{VICC52E1}
\end{align}
where the fourth inequality is because of $X_{\tau_2}^{x_2} > X_{\tau_2}^{x_1}$ and inequality $(K-y)^+ - (K-z)^+ \leq (z-y)^+$.

Before we proceed, recall that $B(t)\in[L, K]$ for $t\in[t_*, T]$ (and $L<K$). By employing the continuity and monotonicity of $B$, intermediate value theorem tells us that there exists a $t_\delta\in(t_*, T)$ so that $B(t_\delta) = L+\delta$ and because $\tau_2\leq \tau_{\mathcal{D}_0}$, where $\tau_{\mathcal{D}_0} = \inf\{ u\in[0, T-t]: (t+u, X_u^{x_2})\in\mathcal{D}_0 \}$ (Set $\mathcal{D}_0$ is the area coloured in green in Figure \ref{TTSTILC5}), we know that $X_{\tau_2}^{x_2}\leq L+\delta$ implies $t+\tau_2\leq t_\delta$ a.s. (see Figure \ref{TYBC52}), equation \eqref{VICC52E1} therefore equals
\begin{align}
&V(t, x_2) - V(t, x_1) \leq  K \mathbb{E}\Big[ e^{-r\tau_2} I\{ X_{\tau_2}^{x_2} > L+\delta\}  \left(  Z\left( t+\tau_2, X_{\tau_2}^{x_2}  \right) - Z\left( t+\tau_2, X_{\tau_2}^{x_1}  \right) \right) \Big] \nonumber\\
&\qquad \qquad\qquad \qquad\,\,\,\,\, +  K \mathbb{E}\Big[ e^{-r\tau_2} I\{ t+\tau_2 \leq t_\delta \} \left(  Z\left( t+\tau_2, X_{\tau_2}^{x_2}  \right) - Z\left( t+\tau_2, X_{\tau_2}^{x_1}  \right) \right) \Big]\nonumber\\
& \qquad  \leq  K \mathbb{E}\Big[ e^{-r\tau_2} I\{ X_{\tau_2}^{x_2} > L+\delta\}  \big(  Z\left( t+\tau_2, X_{\tau_2}^{x_2}  \right) - Z\left( t+\tau_2, X_{\tau_2}^{x_1}  \right) \big) \Big]\nonumber\\
&\qquad \qquad +  K \mathbb{E}\Big[ e^{-r\tau_2} I\{ t+\tau_2 \leq t_\delta \}C_2 \left( X_{\tau_2}^{x_2} - X_{\tau_2}^{x_1} \right) \Big]\nonumber\\
&\qquad  \leq K \mathbb{E}\Big[ e^{-r\tau_2} I\{ X_{\tau_2}^{x_2} > L+\delta\}  \left(  Z\left( t+\tau_2, X_{\tau_2}^{x_2}  \right) - Z\left( t+\tau_2, X_{\tau_2}^{x_1}  \right) \right) \Big]+  K C_2  (x_2 - x_1) \label{VICC52E2}
\end{align}
where the last two steps are due to Corollary \ref{ZIABSOC} and the strong solution of GBM. 

Then, we can choose a $\delta>0$ such that $X_{\tau_2}^{x_1} > L$ almost surely whenever $X_{\tau_2}^{x_2} > L+\delta$ and $0<x_2-x_1\leq \frac{x^*(0)\delta}{2L}< \frac{x_1 \delta}{2 L}$ (the justification follows from the strong solution of GBM and is the same as that in the proof of \eqref{FBPL22ee0}, see also \eqref{GBMSS}), by setting such a $\delta$ in \eqref{VICC52E2}, joining with the fact that $Z(t, x) = 1$ for $x\geq L$,
\begin{align*}
V(t, x_2) - V(t, x_1) \leq K C_2 (x_2 - x_1).
\end{align*}

Next, we let $\tau_1$ be optimal for $V(t, x_1)$ such that
\begin{align*}
&V(t, x_2) - V(t, x_1) \geq \mathbb{E}\left[ e^{-r \tau_1} \left(G(t+{\tau_1}, X_{\tau_1}^{x_2}) - G(t+{\tau_1}, X_{\tau_1}^{x_1})\right) \right]\\
& = \mathbb{E} \left[ e^{-r\tau_1} \left( \left( K- X_{\tau_1}^{x_2} \right)^+ - \left( K- X_{\tau_1}^{x_1} \right)^+  \right) Z\left(t+\tau_1, X_{\tau_1}^{x_2}\right) \right] \nonumber\\
&\qquad + \mathbb{E} \left[ e^{-r\tau_1}  \left( K- X_{\tau_1}^{x_1} \right)^+ \left( Z\left( t+\tau_1, X_{\tau_1}^{x_2}  \right) - Z\left( t+\tau_1, X_{\tau_1}^{x_1}  \right) \right) \right] \nonumber \\
&\geq  \mathbb{E}\left[ e^{-r\tau_1} \left( X_{\tau_1}^{x_1} -  X_{\tau_1}^{x_2} \right) \right] \geq (x_1 - x_2)
\end{align*}
where the second inequality follows from the fact that the increasing map of $x\mapsto Z(t, x)$ and that $X_{\tau_1}^{x_2}\geq X_{\tau_1}^{x_1}$. Then, given $\epsilon>0$, we can choose a $\delta'=\min\big\{\frac{x^*(0)\delta}{2L}, \frac{\epsilon}{{1+KC_2}}\big\}$ (which is independent of $t$), such that $|x_1 - x_2|<\delta'$ implies $|V(t, x_1) - V(t, x_2)|<\epsilon$. This concludes the proof.
\end{proof}

To this point, we have been working hard on the discussion of the stopping and continuation sets and now we can finally return to our main purpose, i.e. formulation of the \textit{free-boundary problem}.
\begin{Proposition}[The Free-boundary Problem] \label{TFBPFC5BCL}
Suppose Assumption \ref{TMOHTC5} holds true. The stopping set and the continuation set are of the forms:
\begin{align*}
\mathcal{C} &= \{ (t, x)\in[0, T)\times(0, \infty): x>b(t) \},\\
\mathcal{D} &= \{ (t, x)\in[0, T)\times(0, \infty): x\leq b(t) \}  \cup \{ (T, x): x\in(0, \infty)\},
\end{align*}
where the map $t\mapsto b(t)$ is increasing. The free-boundary problem in Section 2.2 is therefore reshuffled accordingly as
\begin{align}
&{V}_t + \mathbb{L}_X {V} - r {V} = 0 \qquad\qquad\qquad \text{for $(t, x)\in\mathcal{C}$}, \label{FBPL1ee0}\\
&{V}(t, x)=G(t, x) \qquad\qquad\qquad\,\,\,\,\,\,\,\text{for $x=b(t)$ (instantaneous stopping)},\label{FBPL2ee0}\\
&{V}_x(t, x)=G_x(t, x) \qquad\qquad\,\,\,\,\,\,\,\,\,\,\,\,\,\,\,\text{for $x=b(t)$ with $b(t)\neq L$ (smooth-fit)},\label{FBPL23ee0}\\
&{V}_x(t, L-)=G_x(t, L-) \qquad\qquad\,\,\,\, \text{for $x=b(t)=L$, (smooth-fit breaking down),} \label{FBPL22ee0}\\
&-1 \leq {V}_x(t, L+)\leq G_x(t, L-)\,\,\,\,\,\,\,\,\,\, \text{for $x=b(t)=L$, (smooth-fit breaking down),} \label{FBPL22ee02}\\
&{V}(t, x)>{G}(t, x) \qquad\qquad\qquad\,\,\,\,\,\,\, \text{for $(t, x)\in\mathcal{C}$},\label{FBPL3ee0}\\
&{V}(t, x)={G}(t, x) \qquad\qquad\qquad\,\,\,\,\,\,\, \text{for $(t, x)\in\mathcal{D}$}.\label{FBPL4ee0}
\end{align}
\end{Proposition}


\begin{proof} [Proof of the Smooth-fit Condition \eqref{FBPL23ee0}]
The proof for $b(t)>L$ is essentially the same as that of Proposition \ref{CORRLGB}. It thus remains to show \eqref{FBPL23ee0} holds true for $b(t)<L$. Let a point $(t, x)\in[0, T)\times (0, \infty)$ lying on the boundary $b$ be fixed, i.e. $x=b(t)<L$. Then, let $\epsilon>0$ be such that $x+\epsilon<L$, and
\begin{align*}
\frac{V(t, x+\epsilon)- V(t, x)}{\epsilon}\geq \frac{G(t, x+\epsilon) - G(t, x)}{\epsilon}, 
\end{align*}
where the inequality is due to $(t, x+\epsilon)\in\mathcal{C}$ and $(t, x)\in\mathcal{D}$. Note that by taking the limit as $\epsilon\to0$, 
\begin{align}
\lim_{\epsilon\to0}\frac{G(t, x+\epsilon)- G(t, x)}{\epsilon} = - Z(t, x) + (K-x) \frac{\partial}{\partial x} Z(t, x). \label{SMFLGBG}
\end{align}

In order to prove the converse inequality, we first note that (the exact same argument is presented in the proof of Lemma \ref{VICC52}, so we omit some details here) there exists $\delta>0$ s.t. $B(t_\delta)=L+\delta$ for $t_\delta\in(t_*, T)$. Then, let $\tau_\epsilon$ be optimal for $V(t, x+\epsilon)$, where we can choose $\epsilon\in\left(0, \frac{x\delta}{2L}\right)$ so that $X_{\tau_\epsilon}^{x}>L$ almost surely whenever $X_{\tau_\epsilon}^{x+\epsilon}>L+\delta$ (the justification follows the same pattern as that in the proof of Lemma \ref{VICC51}), 
\begin{align*}
& \frac{V(t, x+\epsilon)- V(t, x)}{\epsilon}  \leq  \frac{1}{\epsilon} \mathbb{E}\left[ e^{-r\tau_\epsilon}\left( G\left(t+\tau_\epsilon, X_{\tau_\epsilon}^{x+\epsilon}\right) - G\left(t+\tau_\epsilon, X_{\tau_\epsilon}^x\right) \right)\right]\\
&\,\,\, \leq - \mathbb{E} \left[  e^{-\frac{\sigma^2}{2} \tau_\epsilon + \sigma W_{\tau_\epsilon}}  I\{ X_{\tau_\epsilon}^{x+\epsilon} < K \} Z\left( t+\tau_\epsilon, X_{\tau_\epsilon}^{x+\epsilon}  \right)  \right]\\
&\,\,\,\,\,\, + \frac{1}{\epsilon}  \mathbb{E}\left[e^{-r\tau_\epsilon} \left( K- X_{\tau_\epsilon}^x \right)^+ I\{ X_{\tau_\epsilon}^{x+\epsilon}\leq L+\delta \}  \left( Z\left( t+\tau_\epsilon, X_{\tau_\epsilon}^{x+\epsilon}  \right)  -  Z\left( t+\tau_\epsilon, X_{\tau_\epsilon}^{x}  \right) \right) \right] \\
&\,\,\,\,\,\, +  \frac{1}{\epsilon}  \mathbb{E} \left[e^{-r\tau_\epsilon} \left( K- X_{\tau_\epsilon}^x \right)^+ I\{ X_{\tau_\epsilon}^{x+\epsilon}> L+\delta \}  \left( Z\left( t+\tau_\epsilon, X_{\tau_\epsilon}^{x+\epsilon}  \right)  -  Z\left( t+\tau_\epsilon, X_{\tau_\epsilon}^{x}  \right) \right) \right]\\
&\,\,\, = - \mathbb{E} \left[  e^{-\frac{\sigma^2}{2} \tau_\epsilon + \sigma W_{\tau_\epsilon}}  I\{ X_{\tau_\epsilon}^{x+\epsilon} < K \} Z\left( t+\tau_\epsilon, X_{\tau_\epsilon}^{x+\epsilon}  \right)  \right]\\
&\,\,\,\,\,\, + \frac{1}{\epsilon}  \mathbb{E} \left[e^{-r\tau_\epsilon} \left( K- X_{\tau_\epsilon}^x \right)^+ I\{ X_{\tau_\epsilon}^{x+\epsilon}\leq L+\delta \}  \left( Z\left( t+\tau_\epsilon, X_{\tau_\epsilon}^{x+\epsilon}  \right)  -  Z\left( t+\tau_\epsilon, X_{\tau_\epsilon}^{x}  \right) \right) \right] \\
&= - \mathbb{E} \left[  e^{-\frac{\sigma^2}{2} \tau_\epsilon + \sigma W_{\tau_\epsilon}}  I\{ X_{\tau_\epsilon}^{x+\epsilon} < K \} Z\left( t+\tau_\epsilon, X_{\tau_\epsilon}^{x+\epsilon}  \right)  \right]\\
&\,\,\,\,\,\, + \frac{1}{\epsilon}  \mathbb{E} \left[e^{-r\tau_\epsilon} \left( K- X_{\tau_\epsilon}^x \right)^+ I\{ L\leq X_{\tau_\epsilon}^{x+\epsilon}\leq L+\delta \}  \left( Z\left( t+\tau_\epsilon, X_{\tau_\epsilon}^{x+\epsilon}  \right)  -  Z\left( t+\tau_\epsilon, X_{\tau_\epsilon}^{x}  \right) \right) \right] \\
&\,\,\,\,\,\, + \frac{1}{\epsilon}  \mathbb{E} \left[e^{-r\tau_\epsilon} \left( K- X_{\tau_\epsilon}^x \right)^+ I\{ X_{\tau_\epsilon}^{x+\epsilon}<L \}  \left( Z\left( t+\tau_\epsilon, X_{\tau_\epsilon}^{x+\epsilon}  \right)  -  Z\left( t+\tau_\epsilon, X_{\tau_\epsilon}^{x}  \right) \right) \right] \\
&\leq - \mathbb{E} \left[  e^{-\frac{\sigma^2}{2} \tau_\epsilon + \sigma W_{\tau_\epsilon}}  I\{ X_{\tau_\epsilon}^{x+\epsilon} < K \} Z\left( t+\tau_\epsilon, X_{\tau_\epsilon}^{x+\epsilon}  \right)  \right]\\
&\,\,\,\,\,\, +C_2  \mathbb{E} \left[e^{-\frac{\sigma^2}{2} \tau_\epsilon + \sigma W_{\tau_\epsilon}}  \left( K- X_{\tau_\epsilon}^x \right)^+ I\{ L\leq X_{\tau_\epsilon}^{x+\epsilon}\leq L+\delta \} \right] \\
&\,\,\,\,\,\, +   \mathbb{E} \left[e^{-\frac{\sigma^2}{2} \tau_\epsilon + \sigma W_{\tau_\epsilon}}  \left( K- X_{\tau_\epsilon}^x \right)^+ I\{ X_{\tau_\epsilon}^{x+\epsilon}<L \}   Z_x\left( t+\tau_\epsilon, \xi_\epsilon \right)   \right] 
\end{align*}
where the second inequality holds true because 
\begin{align}
(K - X_{\tau_\epsilon}^{x+\epsilon})^+ - (K -X_{\tau_\epsilon}^x)^+ \leq (X_{\tau_\epsilon}^x - X_{\tau_\epsilon}^{x+\epsilon}) I\{X_{\tau_\epsilon}^{x+\epsilon} < K\}, \label{TIEQIC5SF}
\end{align}
 and the first equality follows from the facts that for our choice of $\epsilon$, $I\{ X_{\tau_\epsilon}^{x+\epsilon}> L+\delta \} = 1$ entails $I\{X_{\tau_\epsilon}^{x}> L\}=1$, the last inequality is derived from the mean value theorem for $\xi_\epsilon\in(X_{\tau_\epsilon}^{x}, X_{\tau_\epsilon}^{x+\epsilon})$ and Corollary \ref{ZIABSOC}. 
 
Recalling that $X$ is dominated by a positive integrable random variable, letting $\epsilon\to0$ (so that $\tau_\epsilon\to0$ almost surely), and exploiting the dominated convergence theorem (the boundedness of $Z_x(t, x)$ for $(t, x)\in [0, t_\delta)\times(0, L)$ is verified in \eqref{BODERX11}) yield
\begin{align}
&\lim_{\epsilon\to0} - \mathbb{E} \left[  e^{-\frac{\sigma^2}{2} \tau_\epsilon + \sigma W_{\tau_\epsilon}}  I\{ X_{\tau_\epsilon}^{x+\epsilon} < K \} Z\left( t+\tau_\epsilon, X_{\tau_\epsilon}^{x+\epsilon}  \right)  \right]\nonumber\\
&\,\,\,\,\,\, + \lim_{\epsilon\to0} \mathbb{E} \left[e^{\frac{-\sigma^2}{2}\tau_\epsilon + \sigma W_{\tau_\epsilon}} \left( K- X_{\tau_\epsilon}^x \right)^+ I\{ X_{\tau_\epsilon}^{x+\epsilon}<L \}  Z_x \left( t+\tau_\epsilon, \xi_\epsilon  \right) \right]\nonumber\\
&\,\,\,\,\,\, =  - Z(t, x) + (K-x) \frac{\partial}{\partial x} Z(t, x), \label{SMFLGBL}
\end{align}
since $I\{ L\leq X_{\tau_\epsilon}^{x+\epsilon}\leq L+\delta\}\to 0$ as $x=b(t) < L$; which, together with \eqref{SMFLGBG} and the squeeze theorem, yields that $\frac{\partial^+}{\partial x} V(t, x) = - Z(t, x) + (K-x) \frac{\partial}{\partial x} Z(t, x)$ and the fact that $(t, x), (t, x-\epsilon)\in\mathcal{D}$ tells us that 
\[\frac{V(t, x) - V(t, x-\epsilon)}{\epsilon} = \frac{G(t, x) - G(t, x-\epsilon)}{\epsilon},\]
which allows us to obtain \eqref{FBPL23ee0}.
\end{proof}

\begin{proof}[Proof of \eqref{FBPL22ee0} and \eqref{FBPL22ee02}]
(i) Let $x=b(t)=L$ such that there exists $\epsilon>0$, $x-\epsilon<L$ and $(t, x-\epsilon)\in\mathcal{D}$. The first part of \eqref{FBPL22ee0} is a consequence of 
\begin{align*}
\frac{V(t, x) - V(t, x-\epsilon)}{\epsilon} & = \frac{G(t, x) - G(t, x-\epsilon)}{\epsilon} \\
& = \frac{(K-x)Z(t, x) - (K-x+\epsilon)Z(t, x-\epsilon) }{\epsilon}\\
& = -Z(t, x)+ \frac{(K-x+\epsilon) \left[ Z(t, x) - Z(t, x-\epsilon) \right] }{\epsilon},
\end{align*}
and via letting $\epsilon\to0$, we see $\lim\limits_{\epsilon\to0}\frac{V(t, x) - V(t, x-\epsilon)}{\epsilon}=-1 + (K-L)Z_x(t, L-)$

(ii) Again, let $x=b(t)=L$ so that there exists $0<\epsilon<\frac{x\delta}{2L}$ (where $\delta>0$ is defined as so that $B(t_\delta)=L+\delta$ for $t_\delta\in(t_*, T)$), $L<x+\epsilon<K$ and $(t, x+\epsilon)\in\mathcal{C}$ and that
\begin{align*}
\frac{V(t, x+\epsilon) - V(t, x)}{\epsilon} \geq \frac{G(t, x+\epsilon) - G(t, x)}{\epsilon}=-1.
\end{align*}

Furthermore, let $\tau_\epsilon$ be the optimal stopping time for $V(t, x+\epsilon)$ so that $X_{\tau_\epsilon}^{x+\epsilon} \geq L$ and that if $X_{\tau_\epsilon}^{x+\epsilon} > L+\delta$, then $X_{\tau_\epsilon}^x > L$ almost surely in the sense that by the strong solution of GBM,
\begin{align}
(x+\epsilon) e^{\left(r - \frac{\sigma^2}{2}\right)\tau_\epsilon + \sigma W_{\tau_\epsilon}} \geq b(t+\tau_\epsilon) > L+\delta, \label{GBMSS}
\end{align} 
which implies that $e^{\left(r - \frac{\sigma^2}{2}\right)\tau_\epsilon + \sigma W_{\tau_\epsilon}} > \frac{L+\delta}{x + \epsilon}$ almost surely, s.t.
\begin{align*}
x e^{\left(r - \frac{\sigma^2}{2}\right)\tau_\epsilon + \sigma W_{\tau_\epsilon}} > x \frac{L + \delta}{x + \epsilon} > \frac{x(L + \delta)}{x + \frac{x\delta}{2L}} = L + \frac{L\delta}{2L + \delta}.
\end{align*}
Then, 
\begin{align*}
&\frac{V(t, x+\epsilon) - V(t, x)}{\epsilon}  \leq \frac{1}{\epsilon} \mathbb{E}\left[ e^{-r\tau_\epsilon} \left( G\left(t+\tau_\epsilon, X_{\tau_\epsilon}^{x+\epsilon}\right) - G\left(t+\tau_\epsilon, X_{\tau_\epsilon}^{x}\right) \right) \right] \\
&\,\,\, \leq - \mathbb{E}\Big[ e^{-\frac{\sigma^2}{2}\tau_\epsilon + \sigma W_{\tau_\epsilon}} Z\left(t+\tau_\epsilon, X_{\tau_\epsilon}^{x}\right) I\{ X_{\tau_\epsilon}^{x+\epsilon} < K \} \Big] \\
&\,\,\,\,\,\, +  \frac{1}{\epsilon} \mathbb{E}\left[ e^{-r\tau_\epsilon} I\{ X_{\tau_\epsilon}^{x+\epsilon} > L+\delta \} \left( K -  X_{\tau_\epsilon}^{x+\epsilon} \right)^+ \left( Z\left(t+\tau_\epsilon, X_{\tau_\epsilon}^{x+\epsilon}\right) - Z\left(t+\tau_\epsilon, X_{\tau_\epsilon}^{x}\right)  \right) \right]\\
&\,\,\,\,\,\, +  \frac{1}{\epsilon} \mathbb{E}\left[ e^{-r\tau_\epsilon} I\{L\leq X_{\tau_\epsilon}^{x+\epsilon} \leq L+\delta \} \left( K -  X_{\tau_\epsilon}^{x+\epsilon} \right)^+ \left( Z\left(t+\tau_\epsilon, X_{\tau_\epsilon}^{x+\epsilon}\right) - Z\left(t+\tau_\epsilon, X_{\tau_\epsilon}^{x}\right)  \right) \right]\\
&\,\,\, = - \mathbb{E} \Big[ e^{-\frac{\sigma^2}{2}\tau_\epsilon + \sigma W_{\tau_\epsilon}} Z\left(t+\tau_\epsilon, X_{\tau_\epsilon}^{x}\right) I\{ X_{\tau_\epsilon}^{x+\epsilon} < K \} \Big] \\
&\,\,\,\,\,\, + \frac{1}{\epsilon}\mathbb{E}\Big[ e^{-r\tau_\epsilon}  I\{L\leq X_{\tau_\epsilon}^{x+\epsilon} \leq L+\delta \} I\{X_{\tau_\epsilon}^x < L\} \left( K -  X_{\tau_\epsilon}^{x+\epsilon} \right)^+ (Z\left(t+\tau_\epsilon, L\right) - Z\left(t+\tau_\epsilon, X_{\tau_\epsilon}^{x}\right)) \Big],\\
& =  - \mathbb{E} \Big[ e^{-\frac{\sigma^2}{2}\tau_\epsilon + \sigma W_{\tau_\epsilon}} Z\left(t+\tau_\epsilon, X_{\tau_\epsilon}^{x}\right) I\{ X_{\tau_\epsilon}^{x+\epsilon} < K \} \Big] \\
&\,\,\,\,\,\, +  \frac{1}{\epsilon} \mathbb{E}\Big[ e^{-r\tau_\epsilon} I\{L\leq X_{\tau_\epsilon}^{x+\epsilon} \leq L+\delta \} I\{X_{\tau_\epsilon}^x < L\} \left( K -  X_{\tau_\epsilon}^{x+\epsilon} \right)^+ Z_x\left(t+\tau_\epsilon, \xi_\epsilon \right) (L - X_{\tau_\epsilon}^x) \Big],\\
&\leq  - \mathbb{E} \Big[ e^{-\frac{\sigma^2}{2}\tau_\epsilon + \sigma W_{\tau_\epsilon}} Z\left(t+\tau_\epsilon, X_{\tau_\epsilon}^{x}\right) I\{ X_{\tau_\epsilon}^{x+\epsilon} < K \} \Big] \\
&\,\,\,\,\,\, +  \frac{1}{\epsilon} \mathbb{E}\Big[ e^{-r\tau_\epsilon} I\{X_{\tau_\epsilon}^x < L\leq X_{\tau_\epsilon}^{x+\epsilon} \leq L+\delta\} \left( K -  X_{\tau_\epsilon}^{x+\epsilon} \right)^+ Z_x\left(t+\tau_\epsilon, \xi_\epsilon \right) (X_{\tau_\epsilon}^{x+\epsilon} - X_{\tau_\epsilon}^x) \Big],
\end{align*}
where the second inequality is due to \eqref{TIEQIC5SF} and the first equality follows from $X_{\tau_\epsilon}^{x+\epsilon} > L+\delta$ implies that $X_{\tau_\epsilon}^x > L$ almost surely, whereas the last equality follows because for $L\leq x\leq L+\delta$, $Z(t, x) = 1 = Z(t, L)$ and the mean value theorem for  $\xi_\epsilon\in[ X_{\tau_\epsilon}^x, L]$, the last inequality is because of the indicator function implies $0\leq L-X_{\tau_\epsilon}^x \leq X_{\tau_\epsilon}^{x+\epsilon} - X_{\tau_\epsilon}^x$. This leads to 
\begin{align*}
\frac{V(t, x+\epsilon) - V(t, x)}{\epsilon} & \leq  - \mathbb{E} \Big[ e^{-\frac{\sigma^2}{2}\tau_\epsilon + \sigma W_{\tau_\epsilon}} Z\left(t+\tau_\epsilon, X_{\tau_\epsilon}^{x}\right) I\{ X_{\tau_\epsilon}^{x+\epsilon} < K \} \Big] \\
& +  \mathbb{E}\Big[ e^{-\frac{\sigma^2}{2}\tau_\epsilon + \sigma W_{\tau_\epsilon}} I\{X_{\tau_\epsilon}^x < L\leq X_{\tau_\epsilon}^{x+\epsilon} \leq L+\delta\}  \left( K -  X_{\tau_\epsilon}^{x+\epsilon} \right)^+ Z_x\left(t+\tau_\epsilon, \xi_\epsilon \right) \Big]
\end{align*}
By letting $\epsilon \to 0$, we see $\tau_\epsilon\to 0$, $Z(t+\tau_\epsilon, X_{\tau_\epsilon}^x)\to Z(t, L) = 1$ and $\xi_\epsilon\to L$. Therefore,
\begin{align}
\lim_{\epsilon\to 0 } \frac{V(t, x+\epsilon) - V(t, x)}{\epsilon} \leq -1 + (K-L) Z_x(t, L-).
\end{align}
\end{proof}

\begin{rem} \label{VXTLPLUS}
Note that for $(t, x)\in [0, t_*]\times(L, \infty)$, the value function $V(t, x)$ satisfies PDE \eqref{FBPL1ee0}.
\end{rem}

\begin{figure}[H]
\centering
\subfloat[\centering]{{\includegraphics[width=8cm]{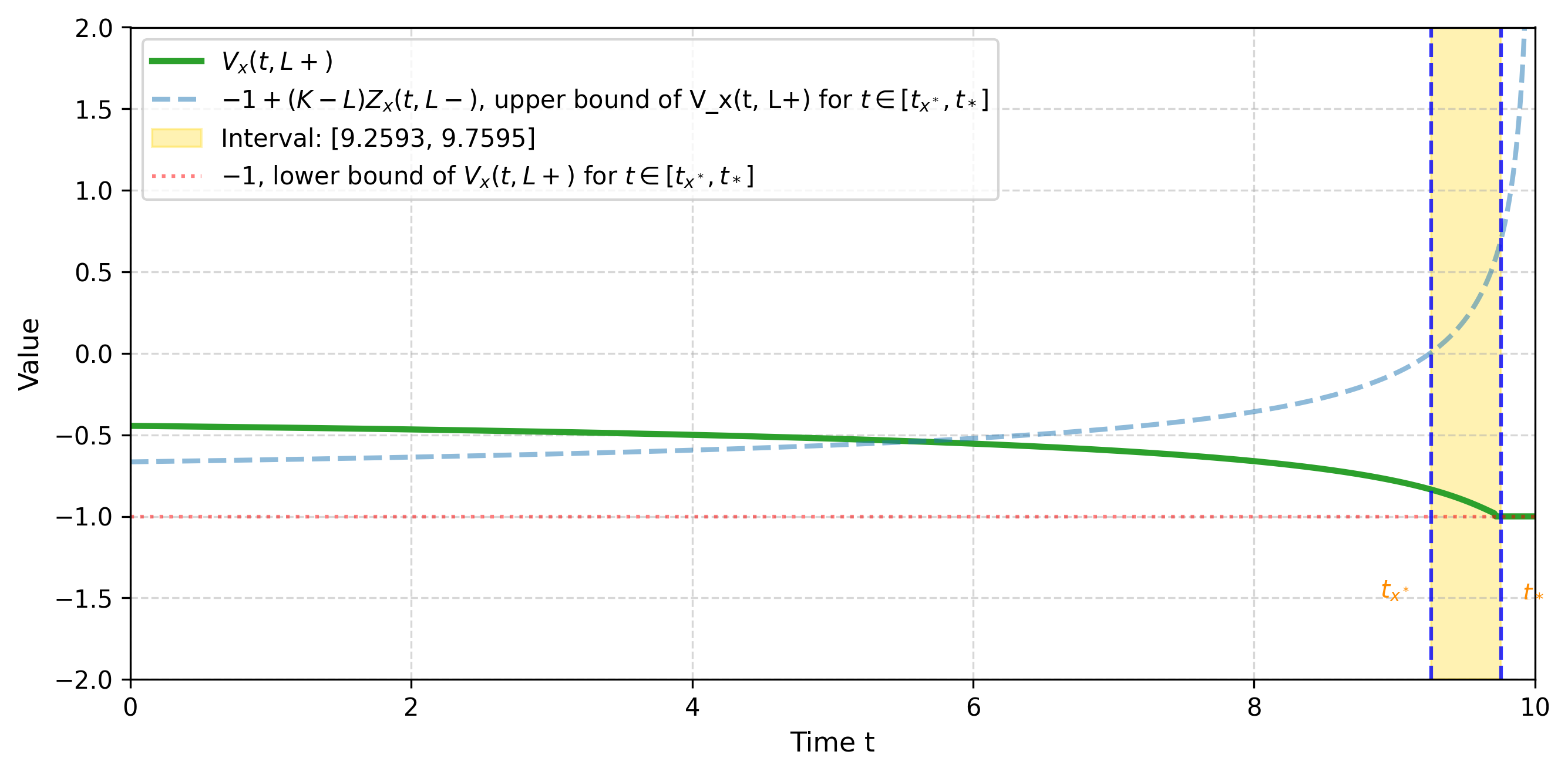} \label{VxL5}}}
\caption{Finite difference approach to obtain $V_x(t, L+)$ numerically by using PDE \eqref{FBPL1ee0} with conditions \eqref{FBPL3ee0} and \eqref{FBPL4ee0}. $K = 7$, $L=5$, $r = 0.05$ and $\sigma = 0.4$, $T=10$.}
\end{figure}

Before we bring this section to a close, we establish the continuity of the boundary. 
\begin{Lemma} \label{BISCus}
Suppose Assumption \ref{TMOHTC5} holds true. The optimal stopping boundary $b$ is continuous on $[0, T]$ and $b(T-)=K$.
\end{Lemma}

\begin{rem} \label{TCOBC5}
(i) The proof of the continuity of $b$ on $[t^b, T]$ can be done by following the exact same pattern as that in the proof of Lemma \ref{BICOLLB}, and if $b$ remains at level L in $[t_b, t^b]$, we know from the left-connectedness of $\mathcal{D}$ that during such time period, $b$ must be continuous. Our task is therefore reduced to prove that $b$ is continuous on $[0, t_b]$.

(ii) The essential tool to prove the continuity of $b$ is the \textit{smooth-fit} condition, which may at times either be difficult to verify or fail to hold. The following proof provides a rather convenient technique to prove the continuity based on the regularity of functions $V$ and $G$ and the monotonicity of $b$. 
\end{rem}

\begin{proof}

To avoid using \textit{smooth-fit} condition, we use the argument provided in \cite[Page 173-175]{Deangelis2015}.

(i) The boundary $b$ is right-continuous on $[0, t_b]$. The conclusion is immediate from the fact that the stopping set is closed and $b$ is increasing on $[0, t_b]$.

(ii) The boundary $b$ is left-continuous on $[0, t_b]$. We assume that, for contradiction, there exists $t_0\in(0, t_b]$ so that a discontinuity of $b$ occurs; that is, at $t_0$, one has $b(t_0-)<b(t_0)$, where $b(t_0-)$ denotes the left limit of the boundary at $t_0$, which always exists as $b(t)$ is monotonically increasing in $(0, t_b)$.

Let $x_1, x_2$ be fixed such that $b(t_0-)<x_1<x_2<b(t_0)$. Then, for fixed $t'\in(0, t_0)$, we define an open bounded domain $\mathcal{R}\subset \mathcal{C}$ with $\mathcal{R}=\{ (t', t_0)\times(x_1, x_2) \}$. Hence, by the structure of $\mathcal{C}$, we  know the value function $V$, which is $C^{1, 2}$ in $\mathcal{R}$, solves the Cauchy problem
\begin{align}
V_t + \mathbb{L}_X V - r V = 0. \label{cauchypbc}
\end{align}

Then, we denote by $C^{\infty}_c( (x_1, x_2))$ the space of infinitely differentiable functions with compact support contained in $(x_1, x_2)$. Let function $\psi\in C^{\infty}_c( (x_1, x_2))$ satisfy $\psi\geq 0$ and $\int_{x_1}^{x_2} \psi(y) dy=1$. By further multiplying \eqref{cauchypbc} by function $\psi$, and integrating the product over $(t, t_0)\times(x_1, x_2)$ for some $t\in(t', t_0)$, 
\begin{align}
\int_{t}^{t_0} \int_{x_1}^{x_2} \left( V_t + \mathbb{L}_X V \right)(s, y) \psi(y) dy ds + \left(- \int_{t}^{t_0} \int_{x_1}^{x_2}rV (s, y) \psi(y) dy ds\right) = 0, \label{pocpapi}
\end{align}
where the second double integral on the left-hand side is strictly negative, as $V>G\geq0$ on $\mathcal{C}$. Therefore, we want to investigate the sign of first double integral by estimating its upper bound, which if is negative, the proof will be completed by reaching a contradiction. To do so, let us first observe that
\begin{align}
&\int_{t}^{t_0} \int_{x_1}^{x_2} V_t (s, y) \psi(y) dy ds  = \int_{x_1}^{x_2} \bigl( V(t_0, y) - V(t, y) \bigr) \psi(y) dy\nonumber\\
&\qquad \leq  \int_{x_1}^{x_2} \bigl( G(t_0, y) - G(t, y) \bigr)\psi(y) dy= \int_{t}^{t_0} \int_{x_1}^{x_2} G_t (s, y) \psi(y) dy ds, \label{ubftft}
\end{align}
where the inequality is the consequence of $(t_0, y)\in\mathcal{D}$ and $(t, y)\in\mathcal{C}$.

Next in line is integration-by-parts (three times), 
\begin{align}
\int_{t}^{t_0} \int_{x_1}^{x_2} \mathbb{L}_{X} V(s, y) \psi(y) dy ds = \int_{t}^{t_0} \int_{x_1}^{x_2}  V(s, y) \mathbb{L}_{X}^* \psi(y) dy ds, \label{ajolx}
\end{align}
where the boundary terms disappear as $\psi(x_1)=\psi(x_2)=\psi'(x_1)=\psi'(x_2)=0$ and $\mathbb{L}_{X}^*$ is the formal adjoint of $\mathbb{L}_X$ given as
\begin{align*}
\mathbb{L}_{X}^* \psi(x) = \frac{1}{2} \frac{\partial^2}{\partial x^2} (\sigma^2 x^2 \psi(x)) - \frac{\partial}{\partial x} (r x \psi(x)).
\end{align*}

Since $G\in C^{1, 2} (\bar{\mathcal{R}})$, 
\begin{align}
\lvert  V(s, y) - G(s, y) \rvert &= \lvert  V(s, y) - V(t_0, y) + G(t_0, y) - G(s, y) \rvert \nonumber\\
&\leq \lvert  V(s, y) - V(t_0, y)  \rvert + \lvert G(t_0, y) - G(s, y) \rvert \nonumber\\
&\leq C_4|y| (t_0 - s)^{\frac{\beta}{2}} + C_5|y| (t_0-s), \qquad \text{for all $(s, y)\in\bar{\mathcal{R}}$}\label{tstsub}
\end{align}
where $C_4, C_5$ are constants independent of $s, y$ and $\beta>0$, the first equality follows via again $(t_0, y)\in\mathcal{D}$ and the second inequality is due to the fulfilment of the conditions proposed on \cite[Page 172,  (C.3)]{Deangelis2015}.

In addition, we consider $\lvert  t_0 - t \rvert<1$ and $C_6$ as a constant independent of $s, y$, upon using $V\geq G$, in $\bar{\mathcal{R}}$,
\begin{align*}
0\leq G(s, y)\leq V(s, y) \leq G(s, y) + C_6|y| (t_0 - s)^{\frac{\beta}{2}},
\end{align*}
of which, as a result, for any $s\in(t, t_0)$, in the view of \eqref{ajolx}, 
\begin{align}
& \int_{t}^{t_0} \int_{x_1}^{x_2}  V(s, y) \mathbb{L}_{X}^* \psi(y) dy ds =  \int_{t}^{t_0} \int_{x_1}^{x_2} {I}\{ \mathbb{L}_{X}^* \psi(y) \geq 0 \}  V(s, y) \mathbb{L}_{X}^* \psi(y) dy ds\nonumber \\
 &\qquad\qquad\qquad\qquad\qquad\qquad\qquad + \int_{t}^{t_0} \int_{x_1}^{x_2} {I}\{ \mathbb{L}_{X}^* \psi(y) < 0 \}  V(s, y) \mathbb{L}_{X}^* \psi(y) dy ds\nonumber\\
 &\qquad\leq \int_{t}^{t_0} \int_{x_1}^{x_2} {I}\{ \mathbb{L}_{X}^* \psi(y) \geq 0 \}  \left(G(s, y)+C_6|y| (t_0 - s)^{\frac{\beta}{2}} \right)\mathbb{L}_{X}^* \psi(y) dy ds \nonumber\\
 &\qquad\qquad + \int_{t}^{t_0} \int_{x_1}^{x_2} {I}\{ \mathbb{L}_{X}^* \psi(y) < 0 \}  G(s, y) \mathbb{L}_{X}^* \psi(y) dy ds\nonumber\\
 &\qquad=  \int_{t}^{t_0} \int_{x_1}^{x_2}  G(s, y) \mathbb{L}_{X}^* \psi(y) dy ds\nonumber \\
 &\qquad\qquad + \int_{t}^{t_0} \int_{x_1}^{x_2} {I}\{ \mathbb{L}_{X}^* \psi(y) \geq 0 \}  C_6|y| (t_0 - s)^{\frac{\beta}{2}} \mathbb{L}_{X}^* \psi(y) dy ds\nonumber \\
 &\qquad=  \int_{t}^{t_0} \int_{x_1}^{x_2}  \mathbb{L}_{X} G(s, y)  \psi(y) dy ds \nonumber\\
 & \qquad\qquad + \int_{t}^{t_0} \int_{x_1}^{x_2} {I}\{ \mathbb{L}_{X}^* \psi(y) \geq 0 \}  C_6|y| (t_0 - s)^{\frac{\beta}{2}} \mathbb{L}_{X}^* \psi(y) dy ds, \label{ubflfx}
\end{align}
where the last equality follows by partial integrations again.

Joining \eqref{ubftft} with \eqref{ubflfx}, in view of equation \eqref{pocpapi}, 
\begin{align}
& 0 = \int_{t}^{t_0} \int_{x_1}^{x_2} \left(-rV + V_t + \mathbb{L}_X V \right)(s, y) \psi(y) dy ds \\
&\leq\int_{t}^{t_0} \int_{x_1}^{x_2} \left( G_t + \mathbb{L}_X G - rG \right) (s, y) \psi(y) dy ds \nonumber \\
 & \qquad + C_6  \int_{t}^{t_0} (t_0 - s)^{\frac{\beta}{2}} \int_{x_1}^{x_2}  I\{ \mathbb{L}_{X}^* \psi(y) \geq 0 \}  \mathbb{L}_{X}^* \psi(y)  dy ds\nonumber\\
 & = \int_{t}^{t_0} \int_{x_1}^{x_2} \left( G_t + \mathbb{L}_X G - rG \right)(s, y) \psi(y) dy ds \nonumber \\
 & \qquad + C_6  \int_{t}^{t_0} (t_0 - s)^{\frac{\beta}{2}} \int_{x_1}^{x_2}  I\{ \mathbb{L}_{X}^* \psi(y) > 0 \}  \mathbb{L}_{X}^* \psi(y)  dy ds\label{pocpapiii}\\
 &= \int_{t}^{t_0} \int_{x_1}^{x_2} \left( G_t + \mathbb{L}_X G - rG \right)(s, y) \psi(y) dy ds + \frac{2\gamma}{2+\beta} (t_0-t)^{1+\frac{\beta}{2}}\nonumber\\
 &\leq \int_{t}^{t_0} \int_{x_1}^{x_2} (-l_\epsilon) \psi(y) dy ds + \frac{2\gamma}{2+\beta} (t_0-t)^{1+\frac{\beta}{2}}\nonumber\\
 &= -l_\epsilon (t_0 - t) + \frac{2\gamma}{2+\beta} (t_0-t)^{1+\frac{\beta}{2}}, \label{pocpapii}
\end{align}
where the second inequality follows in the sense that in such interval $[0, t^*]\times [x_1, x_2]$, there exists some constant $l_\epsilon$ such that,
\begin{align*}
(G_t + \mathbb{L}_X G - rG)(t, x) = (- r K Z - \sigma^2 x^2 Z_x)(t, x) \leq -l_\epsilon<0
\end{align*}
since $x\mapsto Z(t, x)$ is increasing and $Z_x$ exists as $0 < x_1 < x_2 < b(t_0) < L$ and that $\int_{x_1}^{x_2} \psi(y) dy = 1$; since the second term in \eqref{pocpapiii} is strictly positive by the arbitrariness of $\psi$, we set
\begin{align*}
\gamma=\gamma(\psi; x_1, x_2):= C_6  \int_{x_1}^{x_2} I\{ \mathbb{L}_{X}^* \psi(y) > 0 \} \mathbb{L}_{X}^* \psi(y) dy >0,
\end{align*}
from which, the second equality holds.

Taking the limit $t\to t_0$ thus leads to a contradiction in the sense that the second term on the last equality of \eqref{pocpapii} is vanishing more rapidly than its first term, hence the jump may not occur.
\end{proof}

\subsection{The Optimal Stopping Rule} \label{TOSPRC5}

In order to prepare for the main result, we shall first verify the conditions for the change-of-variable formula to be applicable.

\begin{Lemma} \label{LSPTFFLLB}
Let $F(t, x)=e^{-r t} V(t, x)$ be defined on $[0, T)\times(0, \infty)$, 
\begin{align*}
&\mathcal{D}_1=\{(t, x)\in[0, t_b]\times(0, b(t))\}, \\
&\mathcal{D}_2=\{(t, x)\in[t_b, T]\times(0, L)\},\\
&\mathcal{D}_3=\{(t, x)\in[t^b, T]\times (L, b(t))\}. 
\end{align*}
Then, the function $F$ fulfils the following conditions:

(a) $F(t, x)$ is $C^{1, 2}$ on $\mathcal{C}\cup \mathcal{D}_1 \cup \mathcal{D}_2 \cup \mathcal{D}_3$;

(b) $F_t + \mathbb{L}_X F$ is locally bounded;

(c) $t\mapsto F_x(t, b(t)\pm)$ is continuous;

(d) $t\mapsto F_x(t, L\pm)$ is continuous;

(e) $F_{xx}=F_1+F_2$ on $\mathcal{C}\cup\mathcal{D}_1\cup\mathcal{D}_2\cup\mathcal{D}_3$, where $F_1$ is non-negative and $F_2$ is continuous.
\end{Lemma}
\begin{proof}
See Appendix \ref{ALAPiv}.
\end{proof}

\begin{rem}
The change-of-variable formula is thus applicable to $F(t, x)=e^{-rt}V(t, x)$, see \cite{Peskir2005AC}.
\end{rem}

\begin{Theorem} \label{MTLGB}
(i) For $B(0)> L$, the optimal stopping boundary in problem \eqref{OSPFLR} can be characterised as the unique solution of the free-boundary equation:
\begin{align}
K-b(t)&=  \mathbb{E}_{t, b(t)} \left[ e^{-r(T-t)} G\left( T, X_{T} \right)+ r K \int_0^{T-t} e^{-ru} I \{ L<X_{t+u} < b(t+u) \} du \right]\nonumber\\
& + \mathbb{E}_{t, b(t)} \left[  \int_0^{T-t} e^{-ru} \left( rKZ + \sigma^2 X_{t+u}^2 Z_x\right) \left(t+u, X_{t+u}\right) I \{ X_{t+u} < L \} du   \right] \nonumber\\
&+ \mathbb{E}_{t, b(t)} \left[ \frac{1}{2} \int_0^{T-t} e^{-ru} (K-L) Z_x(t+u, L-) dl_u^L(X) \right], \label{RRFBLB}
\end{align} 
and the function $b:[0, T)\mapsto[L, K)$ is continuous and increasing with $b(T-)=K$. The value function admits the following \textit{early exercise premium} representation:
\begin{align}
V(t, x) &= \mathbb{E}_{t, x} \left[ e^{-r(T-t)} G(T, X_T)+ r K \int_0^{T-t} e^{-ru} I \{ L<X_{t+u} < b(t+u) \} du \right] \nonumber\\
&+ \mathbb{E}_{t, x} \left[  \int_0^{T-t} e^{-ru} \left( rKZ + \sigma^2 X_{t+u}^2 Z_x\right)\left(t+u, X_{t+u}\right) I \{ X_{t+u} < L \} du   \right] \nonumber\\
&+ \mathbb{E}_{t, x} \left[\frac{1}{2} \int_0^{T-t} e^{-ru} (K-L) Z_x(t+u, L-) dl_u^L(X) \right], \label{RRFVLB}
\end{align}
for all $(t, x)\in[0, T]\times(0, \infty)$.

(ii) Suppose Assumption \ref{TMOHTC5} holds true. For $B(0)\leq L$, the optimal stopping boundary in problem \eqref{OSPFLR} is given as $b(t):[0, T)\mapsto(0, K)$ is the non-decreasing, continuous function with (a) $b(t)\geq x^*(t)$ for $t\in[0, t_b]$, (b) $b(t) = L$ for $t\in[t_b, t_*]$, (c) $b(t) = B(t)$ for $t\in[t_*, T)$ and (d) $b(T-)=K$ , which is characterised as the unique solution of the following nonlinear integral equation\footnote{The computation of expectation of the local time term is on Appendix \ref{ALAP1} or see \cite[Page 1523]{DetempleandKitapbayev2021}, \cite[Page 9]{DetempleandKitapbayev2018}.}:
\begin{align}
G(t, b(t))& =  \mathbb{E}_{t, b(t)} \left[  e^{-r(T-t)} G\left( T, X_{T} \right) + rK  \int_0^{T-t} e^{-ru} {I} \{L <  X_{t+u} < b(t+u) \} du\right]  \nonumber\\
&+\mathbb{E}_{t, b(t)}\left[ \int_0^{T-t} e^{-ru}\left( rKZ + \sigma^2 (X_{t+u})^2 Z_x \right)(t+u, X_{t+u}) {I} \{ X_{t+u} < b(t+u), X_{t+u}<L \} du \right]\nonumber\\
&+\frac{1}{2} \int_0^{ T-t } e^{-ru} {I}\{ b(t+u)> L \} \bigl( K-L \bigr)Z_x(t+u, L-) d \mathbb{E}_{t, b(t)} \left[l_u^L(X) \right]\nonumber\\
&- \frac{1}{2} \int_0^{ T-t } e^{-ru} {I}\{ b(t+u) = L \} \bigl( V_x\left( t+u, L+ \right) - G_x\left( t+u, L- \right) \bigr)  d \mathbb{E}_{t, b(t)} \left[l_u^L(X) \right],\label{RRFBGB}
\end{align}
The value function admits the following \textit{early exercise premium} representation:
\begin{align}
V(t, x)& =  \mathbb{E}_{t, x} \left[  e^{-r(T-t)} G\left( T, X_{T} \right) + rK  \int_0^{T-t} e^{-ru} {I} \{ L < X_{t+u} < b(t+u) \} du\right]  \nonumber\\
&+\mathbb{E}_{t, x}\left[ \int_0^{T-t} e^{-ru}\left( rKZ + \sigma^2 (X_{t+u})^2 Z_x \right)(t+u, X_{t+u}) {I} \{ X_{t+u} < b(t+u), X_{t+u}<L \} du \right]\nonumber\\
&+\frac{1}{2} \int_0^{ T-t } e^{-ru} {I}\{ b(t+u)> L \}\bigl( K-L \bigr)Z_x(t+u, L-) d \mathbb{E}_{t, x} \left[l_u^L(X) \right]\nonumber\\
&- \frac{1}{2} \int_0^{ T-t } e^{-ru} {I}\{ b(t+u) = L \} \bigl( V_x\left( t+u, L+ \right) - G_x\left( t+u, L- \right) \bigr)  d \mathbb{E}_{t, x} \left[l_u^L(X) \right],\label{RRFVGB}
\end{align}
for all $(t, x)\in[0, T]\times(0, \infty)$.
\end{Theorem}

\begin{rem} \label{beqalB}
(i) In view of \eqref{GKBTC}, we let $s = T-t$ such that \eqref{RRFBLB} becomes
\begin{align*}
K - b(t) & = G(t, b(t)) - rK \mathbb{E}_{t, b(t)} \left[ \int_0^{T-t} e^{-ru} I\{b(t+u) < X_{t+u} < K\} du\right] + \frac{1}{2} \mathbb{E}_{t, b(t)}\left[ \int_0^{T-t} e^{-ru}dl_u^K \right],
\end{align*} 
after which, we apply the change-of-variable formula for $G^A(x)$ and see that
\begin{align*}
\mathbb{E}_{t, b(t)}\left[ e^{-r(T-t)} G^A(X_{T}) \right] = G^A(b(t)) - rK \mathbb{E}_{t, b(t)} \left[ \int_0^{T-t} e^{-ru} I\{X_{t+u} < K\} du \right] + \frac{1}{2} \mathbb{E}_{t, b(t)}\left[ \int_0^{T-t} e^{-ru}dl_u^K \right].
\end{align*}
Note that $G(t, b(t)) = G^A(b(t))$ and hence,
\begin{align}
K - b(t) = \mathbb{E}_{t, b(t)}\left[ e^{-r(T-t)} G^A(X_{T}) \right] + rK \mathbb{E}_{t, b(t)} \left[ \int_0^{T-t} e^{-ru} I\{ X_{t+u} \leq b(t+u)\} du \right]
\end{align}
which is precisely \cite[Equation (25.2.38)]{PeskirandShiryaev}.

(ii)  Note that $V_x(t, L+)$ can be obtained by solving PDE \eqref{FBPL1ee0} on $[0, t_*]\times(L, \infty)$. See also Remark \ref{VXTLPLUS}.
\end{rem}

%

\begin{proof}[Proof of \eqref{RRFBGB} and \eqref{RRFVGB}]
We follow the argument in \cite[Pages 386-391]{PeskirandShiryaev}. With the proofs of \eqref{RRFBLB} and \eqref{RRFVLB} being similar, we confine ourselves to proving \eqref{RRFBGB} and \eqref{RRFVGB}. Since the function $F$ fulfils the conditions on Lemma \ref{LSPTFFLLB}, an application of change-of-variable formula yields
\begin{align}
& e^{-rs} V\left( t+s, X_{t+s}^x \right) = V(t, x) +M_s \nonumber\\
&\qquad\qquad\,\,\, + \int_0^s e^{-ru} \left( -r V + V_t + \mathbb{L}_X V \right)\left( t+u, X_{t+u}^x \right) {I} \{ X_{t+u}^x \neq b(t+u), X_{t+u}^x\neq L \} du \nonumber\\
&\qquad\qquad\,\,\,+ \frac{1}{2} \int_0^s e^{-ru} \bigl( V_x\left( t+u, L+ \right) - V_x\left( t+u, L- \right) \bigr) dl_u^L(X^x)\nonumber\\
&\qquad\qquad\,\,\,+\frac{1}{2} \int_0^s e^{-ru} \bigl( V_x\left( t+u, b(t+u)+ \right) - V_x\left( t+u, b(t+u)- \right) \bigr) dl_u^b(X^x)\nonumber\\
&\qquad=V(t, x) +M_s\nonumber \\
&\qquad\qquad\,\,\, + \int_0^s e^{-ru} \left( -r G + G_t + \mathbb{L}_X G \right)\left( t+u, X_{t+u}^x \right) {I} \{ X_{t+u}^x < b(t+u), X_{t+u}^x\neq L \} du \nonumber\\
&\qquad\qquad\,\,\,+\frac{1}{2} \int_0^{ s } e^{-ru} {I}\{ b(t+u)>L \} \bigl( G_x\left( t+u, L+ \right) - G_x\left( t+u, L- \right) \bigr) dl_u^L(X^x)\nonumber\\
&\qquad\qquad\,\,\,+\frac{1}{2} \int_0^{ s } e^{-ru} {I}\{ b(t+u)=L \} \bigl( V_x\left( t+u, L+ \right) - V_x\left( t+u, L- \right) \bigr) dl_u^L(X^x), \\
&\qquad=V(t, x) +M_s\nonumber \\
&\qquad\qquad\,\,\, -rK \int_0^s e^{-ru} {I} \{L < X_{t+u}^x < b(t+u) \} du \nonumber\\
&\qquad\qquad\,\,\,- \int_0^s e^{-ru}\left( rKZ + \sigma^2 (X_{t+u}^x)^2 Z_x\right)(t+u, X_{t+u}^x)  {I} \{ X_{t+u}^x < b(t+u), X_{t+u}^x < L \} du \nonumber\\
&\qquad\qquad\,\,\,+\frac{1}{2} \int_0^{ s } e^{-ru} {I}\{ b(t+u)>L \} \bigl( G_x\left( t+u, L+ \right) - G_x\left( t+u, L- \right) \bigr) dl_u^L(X^x)\nonumber\\
&\qquad\qquad\,\,\,+\frac{1}{2} \int_0^{ s } e^{-ru} {I}\{ b(t+u)=L \} \bigl( V_x\left( t+u, L+ \right) - G_x\left( t+u, L- \right) \bigr) dl_u^L(X^x),
\label{TRPLGB}
\end{align}
where $M_s=\int_0^s e^{-ru} \sigma X_{t+u} V_x\left( t+u, X_{t+u} \right) {I} \{ X_{t+u}^x \neq b(t+u), X_{t+u}^x\neq L \} dW_u$ is a martingale under measure $\mathbb{P}_{t, x}$ for each $s\in[0, T-t]$, and the second equality is due to the fact that the \textit{smooth-fit} condition fails for $b(t)=L$ and the gain function is not smooth as $b(t)>L$ for $t\in[0, T]$. Note that $G_x(t, L+) - G_x(t, L-) = - (K-L)Z_x(t, L-)$ as $K>L$ by the assumption.

Then, upon taking the expectation under the measure $\mathbb{P}_{t, x}$ of equation \eqref{TRPLGB} and invoking the optional sampling theorem, we obtain
\begin{align}
& \mathbb{E}_{t, x} \left[e^{-rs} V\left( t+s, X_{t+s} \right)\right) = V(t, x) - rK \mathbb{E}_{t, x}\left[ \int_0^s e^{-ru} {I} \{ L < X_{t+u} < b(t+u) \} du\right] \nonumber\\
&\qquad\qquad\,\,\,-\mathbb{E}_{t, x}\left[ \int_0^s e^{-ru}\left( rKZ + \sigma^2 (X_{t+u})^2 Z_x \right)(t+u, X_{t+u}) {I} \{ X_{t+u} < b(t+u), X_{t+u} < L \} du \right]\nonumber\\
&\qquad\qquad\,\,\,- \frac{1}{2} \int_0^{ s } e^{-ru} {I}\{ b(t+u)> L \} (K-L)Z_x(t+u, L-) d \mathbb{E}_{t, x} \left[l_u^L(X) \right],\nonumber\\
&\qquad\qquad\,\,\,+ \frac{1}{2} \int_0^{ s } e^{-ru} {I}\{ b(t+u) = L \} \bigl( V_x\left( t+u, L+ \right) - G_x\left( t+u, L- \right) \bigr)  d \mathbb{E}_{t, x} \left[l_u^L(X) \right] \label{ETRPLGB}
\end{align}

Next, let $s=T-t$ such that $\mathbb{E}_{t, x} \left[e^{-r(T-t)} V\left( T, X_{T} \right)\right]  = \mathbb{E}_{t, x} \left[ e^{-r(T-t)} G(T, X_T) \right]$,
which, together with \eqref{ETRPLGB} shows that
\begin{align*}
V(t, x)& =  \mathbb{E}_{t, x} \left[  e^{-r(T-t)} G\left( T, X_{T} \right) + rK  \int_0^{T-t} e^{-ru} {I} \{ L < X_{t+u} < b(t+u)\} du\right]  \nonumber\\
&+\mathbb{E}_{t, x}\left[ \int_0^{T-t} e^{-ru}\left( rKZ + \sigma^2 (X_{t+u})^2 Z_x \right)(t+u, X_{t+u}) {I} \{ X_{t+u} < b(t+u), X_{t+u}<L \} du \right]\nonumber\\
&+ \frac{1}{2} \int_0^{ T-t } e^{-ru} {I}\{ b(t+u)> L \}(K-L)Z_x(t+u, L-)  d \mathbb{E}_{t, x} \left[l_u^L(X) \right]\nonumber\\
&- \frac{1}{2} \int_0^{ T-t } e^{-ru} {I}\{ b(t+u) = L \} \bigl( V_x\left( t+u, L+ \right) - G_x\left( t+u, L- \right) \bigr)  d \mathbb{E}_{t, x} \left[l_u^L(X) \right]
\end{align*}
which is precisely \eqref{RRFVGB}; and finally, let $x=b(t)$ and since $V(t, b(t))=G(t, b(t))$, \eqref{RRFBGB} follows. 
\end{proof}

\begin{proof}[Proof of uniqueness of the boundary $b$] From Lemma \ref{VICC552} and Remark \ref{beqalB}, we know that $b(t) = B(t)$ for $t\in[t_*, T]$ and thereby, proving the uniqueness of $b$ for $t\in[t_*, T]$. We then follow the argument in \cite[Pages 386-391]{PeskirandShiryaev}. The key idea is to show that if there exists a continuous increasing function $c:[0, T]\mapsto (0, \infty)$ that satisfies $0<c(t)<K$, $c(t) = B(t)$ for $t\in[t_*, T]$ and solves \eqref{RRFBGB}, then such $c$ must coincide with the optimal stopping boundary $b$. We proceed in steps:

(i) In view of \eqref{RRFVGB}, we introduce the following functions: for all $(t, x)\in[0,T)\times(0,\infty)$
\begin{align}
U(t, x)&= \underbrace{\mathbb{E}_{t, x} \left[  e^{-r(T-t)} G\left( T, X_{T} \right)\right]}_{U_1(t, x)} + rK \underbrace{\mathbb{E}_{t, x}\left[  \int_0^{T-t} e^{-ru} {I} \{ L < X_{t+u} < c(t+u) \} du\right]}_{U_2(t, x)}  \nonumber\\
&+ \underbrace{\mathbb{E}_{t, x}\left[ \int_0^{T-t} e^{-ru}\left( rKZ + \sigma^2 (X_{t+u})^2 Z_x \right)(t+u, X_{t+u}) {I} \{ X_{t+u} < c(t+u), X_{t+u}< L \} du \right]}_{U_3(t, x)}\nonumber\\
&+ \underbrace{\frac{1}{2} \int_0^{ T-t } e^{-ru} {I}\{ c(t+u)> L \}(K-L)Z_x(t+u, L-)  d \mathbb{E}_{t, x} \left[l_u^L(X) \right]}_{U_4(t, x)} \nonumber\\
&- \underbrace{ \frac{1}{2} \int_0^{ T-t } e^{-ru} {I}\{ c(t+u) = L \} \bigl( U_x\left( t+u, L+ \right) - G_x\left( t+u, L- \right) \bigr)  d \mathbb{E}_{t, x} \left[l_u^L(X) \right]}_{U_5(t, x)}, \label{UOBUC}
\end{align}
where $-1 \leq U_x\left( t+u, L+ \right) \leq -1 + (K-L)Z_x(t+u, L-) $; and set
\begin{align*}
V^c(t, x)=
\begin{cases}
U(t, x), \qquad x>c(t),\\
G(t, x), \qquad x\leq c(t),
\end{cases}
\end{align*}
which is continuous on $[0,T)\times(0,\infty)$ in the sense that $c$ solves \eqref{RRFBGB} such that $U(t, c(t))=G(t, c(t))$. Note that \(G(T,X_T)\) and
\[
\int_0^{T-t} e^{-ru}I{\{L<X_{t+u}<c(t+u)\}}\,du
\]
are uniformly square-integrable by $X$ being geometric Brownian motion. Moreover, we have
\begin{align*}
&\int_0^{T-t} e^{-ru}\left(rKZ+\sigma^2(X_{t+u})^2Z_x\right)(t+u,X_{t+u})
I{\{X_{t+u}<c(t+u), X_{t+u} <L\}}\,du\\
&=\int_0^{T-t}e^{-ru}\Bigg(
rKZ(t+u,X_{t+u})
+\sigma\sqrt{\frac{2}{\pi (T-t-u)}}e^{-d_1^2/2}X_{t+u}-\alpha L^\alpha X_{t+u}^{1-\alpha}\sigma^2\Phi(d_2)
\Bigg)\\
&\qquad \times I{\{X_{t+u}<c(t+u), X_{t+u} <L\}}\,du,
\end{align*}
which is also uniformly square-integrable, since \(1-\alpha>0\). It then follows from \cite[Theorem 3.8]{BD2018} that \(U_1(t,x)\), \(U_2(t,x)\), and \(U_3(t,x)\) are continuously differentiable with respect to \(x\) on \([0,T)\times (0, \infty)\). On the other hand, since
\[
G_x(t,L-)
=-1+(K-L)Z_x(t,L-),\,\,G_x(t,L+)=-1,\,\, -1 \leq U_x\left( t+u, L+ \right) \leq -1 + (K-L)Z_x(t, L-),
\]
the spatial derivative of $G, U$ have, in general, a jump at $x=L$. The corresponding local-time contribution in the change-of-variable formula is precisely represented by $U_4, U_5$. In particular, $U_4, U_5$ account for the possible discontinuity of the spatial derivative at $x=L$. Therefore, $U$ is differentiable with respect to $x$ on \([0,T)\times \{(0, L) \cup (L, \infty)\}\).

We then reconstruct the continuation and stopping sets by the means of $c$:
\begin{align*}
&\mathcal{C}_c=\{(t, x)\in[0, T)\times(0, \infty): x>c(t)  \},\\
&\mathcal{D}_c=\{ t, x)\in[0, T]\times(0, \infty): x\leq c(t) \}.
\end{align*}
From the definition, $V^c$ has the following properties: 

(a) $V^c=U$ is $C^{1, 2}$ and satisfies the equation $V_t^c + \mathbb{L}_X V^c = r V^c$ on $\mathcal{C}_c$; 

(b) $V_x^c$ is continuous on $\mathcal{C}_c$;

(c) $V^c=G$ is $C^{1, 2}$ on $\mathcal{D}_c \setminus \{(t, x)\in[0, T]\times\{L, c(t)\}\}$.

(ii) Following exactly the same pattern on Lemma \ref{LSPTFFLLB}, we can verify that the change-of-variable formula is applicable to $e^{-rt}V^c(t, X_t)$ so that
\begin{align*}
& e^{-rs} V^c \left( t+s, X_{t+s}^x \right) = V^c(t, x) +M_s^c \\
&\,\,\,\,\,\,\,\,\,\, + \int_0^s e^{-ru} \left( -r V^c + V_t^c + \mathbb{L}_X V^c \right)\left( t+u, X_{t+u}^x \right)I \{ X_{t+u}^x \neq c(t+u), X_{t+u}^x\neq L \} du \nonumber\\
&\,\,\,\,\,\,\,\,\,\,+ \frac{1}{2} \int_0^s e^{-ru} \bigl( V_x^c \left( t+u, L+ \right) - V_x^c\left( t+u, L- \right) \bigr) I\{c(t+u)\neq L\}dl_u^L(X^x)\nonumber\\
&\,\,\,\,\,\,\,\,\,\,+\frac{1}{2} \int_0^s e^{-ru} \bigl( V_x^c\left( t+u, c(t+u)+ \right) - V_x^c\left( t+u, c(t+u)- \right) \bigr)I\{c(t+u)\neq L\} dl_u^c(X^x)\nonumber\\
&\,\,\,\,\,\,\,\,\,\,+\frac{1}{2} \int_0^s e^{-ru} \bigl( V_x^c\left( t+u, c(t+u)+ \right) - V_x^c\left( t+u, c(t+u)- \right) \bigr)I\{c(t+u) = L\} dl_u^c(X^x)\nonumber\\
&=V^c(t, x) +M_s^c\\
&\,\,\,\,\,\,\,\,\,\,  + \int_0^s e^{-ru} \left( -r G + G_t + \mathbb{L}_X G \right)\left( t+u, X_{t+u}^x \right) I\{ X_{t+u}^x < c(t+u), X_{t+u}^x\neq L \} du \nonumber\\
&\,\,\,\,\,\,\,\,\,\, +\frac{1}{2} \int_0^s e^{-ru} \bigl( V_x^c\left( t+u, L+ \right) - V_x^c\left( t+u, L- \right) \bigr)I\{c(t+u) < L\} dl_u^L(X^x)\\
&\,\,\,\,\,\,\,\,\,\, +\frac{1}{2} \int_0^s e^{-ru} \bigl( V_x^c\left( t+u, L+ \right) - V_x^c\left( t+u, L- \right) \bigr)I\{c(t+u) > L\} dl_u^L(X^x)\\
&\,\,\,\,\,\,\,\,\,\,+\frac{1}{2} \int_0^s e^{-ru} \bigl( V_x^c\left( t+u, c(t+u)+ \right) - V_x^c\left( t+u, c(t+u)- \right) \bigr) dl_u^c(X^x),
\end{align*}
where $M_s^c=\int_0^s e^{-ru} \sigma X_{t+u} V_x^c(t+u, X_{t+u}) I\{ X_{t+u}\neq c(t+u), X_{t+u}\neq L \} dW_u$ and it is a martingale under the measure $\mathbb{P}_{t, x}$ so that $\mathbb{E}_{t, x}\left[ M_s^c \right]= 0$ for each $0\leq s \leq T-t$. Let us take the expectation under the measure $\mathbb{P}_{t, x}$ such that
\begin{align}
& \mathbb{E}_{t, x}\left[e^{-rs} V^c \left( t+s, X_{t+s} \right) \right] = V^c(t, x) \nonumber\\
&\qquad  +   \mathbb{E}_{t, x}\left[ \int_0^s e^{-ru} \left( -r G + G_t + \mathbb{L}_X G \right)\left( t+u, X_{t+u} \right) I\{ X_{t+u} < c(t+u), X_{t+u} \neq L \} du\right] \nonumber\\
&\qquad + \mathbb{E}_{t, x} \left[ \frac{1}{2} \int_0^s e^{-ru} \bigl( G_x\left( t+u, L+ \right) - G_x\left( t+u, L- \right) \bigr) I\{c(t+u) > L\} dl_u^L(X)\right]\nonumber\\
&\qquad +  \mathbb{E}_{t, x} \left[ \frac{1}{2} \int_0^s e^{-ru} \bigl( U_x\left( t+u, L+ \right) - U_x\left( t+u, L- \right) \bigr) I\{c(t+u) < L\} dl_u^L(X)\right]\nonumber \\
&\qquad + \mathbb{E}\left[ \frac{1}{2} \int_0^s e^{-ru} \bigl( V_x^c\left( t+u, c(t+u)+ \right) - V_x^c\left( t+u, c(t+u)- \right) \bigr) I\{ c(t+u) \neq L\} dl_u^c(X^x)\right]\nonumber\\
&\qquad + \mathbb{E}\left[ \frac{1}{2} \int_0^s e^{-ru} \bigl( V_x^c\left( t+u, c(t+u)+ \right) - V_x^c\left( t+u, c(t+u)- \right) \bigr) I\{ c(t+u) = L\} dl_u^c(X^x)\right]\nonumber\\
&=  V^c(t, x) \nonumber\\
&\qquad  +   \mathbb{E}_{t, x}\left[ \int_0^s e^{-ru} \left( -r G + G_t + \mathbb{L}_X G \right)\left( t+u, X_{t+u} \right) I\{ X_{t+u} < c(t+u), X_{t+u} \neq L \} du\right] \nonumber\\
&\qquad + \mathbb{E}_{t, x}\left[ \frac{1}{2} \int_0^s e^{-ru} \bigl( V_x^c\left( t+u, c(t+u)+ \right) - V_x^c\left( t+u, c(t+u)- \right) \bigr) I\{ c(t+u) \neq L\} dl_u^c(X^x)\right]\nonumber\\
&\qquad + \mathbb{E}_{t, x}\left[ \frac{1}{2} \int_0^s e^{-ru} \bigl( V_x^c\left( t+u, c(t+u)+ \right) - V_x^c\left( t+u, c(t+u)- \right) \bigr) I\{ c(t+u) = L\} dl_u^c(X^x)\right]\nonumber\\
&\qquad + \mathbb{E}_{t, x} \left[ \frac{1}{2} \int_0^s e^{-ru} \bigl( G_x\left( t+u, L+ \right) - G_x\left( t+u, L- \right) \bigr) I\{c(t+u) > L\} dl_u^L(X)\right], \label{LSPTFOVC}
\end{align}
where the last equality is due to the definition of $V^c$ and statements (a) and  (b) in step (i).

Let $s = T-t$ such that by comparing \eqref{LSPTFOVC} with \eqref{UOBUC}, it is clear that
\begin{align*}
&2(U(t, x) - V^c(t, x))I\{x\leq c(t)\} =2 (U(t, x) - G(t, x))I\{x\leq c(t)\}\nonumber\\
&=  \mathbb{E}_{t, x} \left[  \int_0^{T-t} e^{-ru} \bigl( U_x\left( t+u, L+ \right) - G_x\left( t+u, L- \right) \bigr) I\{c(t+u) = L\} dl_u^L(X)\right]\nonumber\\
&\qquad - \mathbb{E}_{t, x}\left[  \int_0^{T-t} e^{-ru} \bigl( V_x^c\left( t+u, c(t+u)+ \right) - V_x^c\left( t+u, c(t+u)- \right) \bigr) I\{ c(t+u) \neq L\} dl_u^c(X)\right]\nonumber\\
&\qquad - \mathbb{E}_{t, x}\left[  \int_0^{T-t} e^{-ru} \bigl( V_x^c\left( t+u, c(t+u)+ \right) - V_x^c\left( t+u, c(t+u)- \right) \bigr) I\{ c(t+u) = L\} dl_u^c(X)\right],
\end{align*}
and therefore, to prove $x\mapsto V^c(t, x)$ is $C^1$ at $x=c(t)$ for all $c(t)\neq L$ for each $0\leq t < T$ given and fixed, we need to prove that for all $0<x \leq c(t)$,
\begin{align}
&U(t, x) = G(t, x). \label{UCEGID1}
\end{align}
Note that by the continuously differentiable $U_x(t, x)$ on $[0, T)\times\{(0, L)\cup(L, \infty)\}$, \eqref{UCEGID1} implies that
\begin{align*}
U_x(t, c(t)+) = U_x(t, c(t)-) = G_x(t, c(t)-),
\end{align*}
for $c(t) \neq L$ such that $\frac{\partial}{\partial x} (U(t, x) - G(t, x)) |_{x=c(t), c(t)\neq L} = V_x^c(t, c(t) +) - V_x^c(t, c(t) - ) = 0
$; on the other hand, \eqref{UCEGID1} also implies that
\begin{align*}
& \mathbb{E}_{t, x} \left[  \int_0^{T-t} e^{-ru} \bigl( V_x^c\left( t+u, L+ \right) - V_x^c\left( t+u, L- \right) \bigr) I\{c(t+u) = L\} dl_u^L(X)\right] \\
 &=  \mathbb{E}_{t, x}\left[  \int_0^{T-t} e^{-ru} \bigl( U_x\left( t+u, c(t+u)+ \right) - G_x\left( t+u, c(t+u)- \right) \bigr) I\{ c(t+u) = L\} dl_u^c(X)\right],
\end{align*}
i.e. $V_x^c(t, L+) = U_x(t, L+)$ and $V_x^c(t, L-) = G_x(t, L-)$, for $c(t) = L$.

(iii) To show \eqref{UCEGID1}, we observe that
\begin{align}
& \mathbb{E}_{t, x} \left[ e^{-rs} U(t+s, X_{t+s})\right] - U(t, x) =- rK\mathbb{E}_{t, x}\left[ \int_0^s e^{-ru}  I\{L < X_{t+u}< c(t+u) \} du\right]\nonumber\\
&\qquad\qquad  - \mathbb{E}_{t, x} \left[  \int_0^s e^{-ru}  (rKZ + \sigma^2(X_{t+u})^2 Z_x)(t+u, X_{t+u}) I\{ X_{t+u}< c(t+u), X_{t+u} < L \} du \right] \nonumber\\
& \qquad\qquad -  \frac{1}{2} \int_0^s e^{-ru} (K-L) Z_x(t+u, L-){I}\{ c(t+u)> L \} d\mathbb{E}\left[l_u^L(X)\right]\nonumber\\
&\qquad \qquad + \frac{1}{2} \int_0^{ s } e^{-ru} {I}\{ c(t+u) = L \} \bigl( U_x\left( t+u, L+ \right) - G_x\left( t+u, L- \right) \bigr)  d \mathbb{E}_{t, x} \left[l_u^L(X) \right], \label{MPPOULB}
\end{align}
which follows from the definition of $U(t+s, X_{t+s})$ via \eqref{UOBUC}, the strong Markov property of $X$ and the smoothing property of conditional expectation as $\mathcal{F}_t \subseteq \mathcal{F}_{t+s}$.

With the above observation in mind, we are now ready to show that $U(t, x)=G(t, x)$ for all $x \leq c(t)$. We first consider the stopping time,
\begin{align*}
\sigma_c=\inf\{ 0\leq s \leq T-t: X_{t+s} \geq c(t+s) \},
\end{align*}
and note that, since $c$ solves \eqref{RRFBGB}, $U(t, c(t))=G(t, c(t))$ for all $0\leq t<T$ and because of the definition of $U$ in \eqref{UOBUC}, $U(T, x)=G(T, x)$. From \eqref{MPPOULB}, we see that  for $0<x\leq c(t)$,
\begin{align*}
&U(t, x)= \mathbb{E}_{t, x} \left[ e^{-r\sigma_c} U(t+\sigma_c, X_{t+\sigma_c}) \right] + rK\mathbb{E}_{t, x}\left[ \int_0^{\sigma_c}e^{-ru}  I\{ L < X_{t+u}< c(t+u) \} du\right]\nonumber\\
&\qquad\qquad  + \mathbb{E}_{t, x} \left[ \int_0^{\sigma_c} e^{-ru} (rKZ + \sigma^2(X_{t+u})^2 Z_x)(t+u, X_{t+u}) I\{ X_{t+u}< c(t+u), X_{t+u} < L \} du \right] \nonumber\\
& \qquad\qquad +  \frac{1}{2} \int_0^{\sigma_c} e^{-ru} (K-L) Z_x(t+u, L-){I}\{ c(t+u)\geq L \} d\mathbb{E}\left[l_u^L(X)\right] \\
&\qquad \qquad - \frac{1}{2} \int_0^{ \sigma_c } e^{-ru} {I}\{ c(t+u) = L \} \bigl( U_x\left( t+u, L+ \right) - G_x\left( t+u, L- \right) \bigr)  d \mathbb{E}_{t, x} \left[l_u^L(X) \right]\\
&\,\,\,=\mathbb{E}_{t, x} \left[ e^{-r\sigma_c} G(t+\sigma_c, X_{t+\sigma_c}) \right] + rK\mathbb{E}_{t, x}\left[ \int_0^{\sigma_c}e^{-ru}  I\{ X_{t+u}< c(t+u), X_{t+u} > L \} du\right]\nonumber\\
&\qquad\qquad  + \mathbb{E}_{t, x} \left[ \int_0^{\sigma_c} e^{-ru}  (rKZ + \sigma^2(X_{t+u})^2 Z_x)(t+u, X_{t+u}) I\{ X_{t+u}< c(t+u), X_{t+u} < L \} du \right] \nonumber\\
& \qquad\qquad +  \frac{1}{2} \int_0^{\sigma_c} e^{-ru} (K-L) Z_x(t+u, L-){I}\{ c(t+u)> L \} d\mathbb{E}\left[l_u^L(X)\right]\\
& \qquad \qquad - \frac{1}{2} \int_0^{ \sigma_c } e^{-ru} {I}\{ c(t+u) = L \} \bigl( U_x\left( t+u, L+ \right) - G_x\left( t+u, L- \right) \bigr)  d \mathbb{E}_{t, x} \left[l_u^L(X) \right] = G(t, x),
\end{align*}
where the last equality is because of \eqref{GKBTC} and the definition of $\sigma_c$, that is, for $u\in[0, \sigma_c]$, $I\{X_{t+u}<c(t+u)\}=1$ and in the local time term, $I\{c(t+u)<L\} = I\{c(t+u)<X_{t+u}\} = 0$; after which, the desired assertion follows; more precisely, \eqref{LSPTFOVC} equals 
\begin{align}
& \mathbb{E}_{t, x}\left[e^{-rs} V^c \left( t+s, X_{t+s} \right) \right] =  V^c(t, x) \nonumber\\
&\qquad  +   \mathbb{E}_{t, x}\left[ \int_0^s e^{-ru} \left( -r G + G_t + \mathbb{L}_X G \right)\left( t+u, X_{t+u} \right) I\{ X_{t+u} < c(t+u), X_{t+u} \neq L \} du\right] \nonumber\\
&\qquad + \mathbb{E}_{t, x}\left[ \frac{1}{2} \int_0^s e^{-ru} \bigl( U_x\left( t+u, c(t+u)+ \right) - G_x\left( t+u, c(t+u)- \right) \bigr) I\{ c(t+u) = L\} dl_u^c(X^x)\right]\nonumber\\
&\qquad + \mathbb{E}_{t, x} \left[ \frac{1}{2} \int_0^s e^{-ru} \bigl( G_x\left( t+u, L+ \right) - G_x\left( t+u, L- \right) \bigr) I\{c(t+u) > L\} dl_u^L(X)\right]. \label{LSPTFOVCII}
\end{align}

(iv) Before moving on, we pause for another direct consequence of \eqref{LSPTFOVCII}: $V^c(t, x)\leq V(t, x)$.
To see this relation, we begin by considering the stopping time
\begin{align*}
\tau_c=\inf\{ 0\leq s\leq T-t: X_{t+s} \leq c(t+s) \}.
\end{align*}
With $s$ replaced by $\tau_c$ on \eqref{LSPTFOVCII}, we find that
\begin{align*}
V^c(t, x)= \mathbb{E}_{t, x}\left[ e^{-r\tau_c} V^c (t+\tau_c, X_{t+\tau_c}) \right] = E_{t, x} \left( e^{-r\tau_c} G (t+\tau_c, X_{t+\tau_c}) \right)  \leq V(t, x),
\end{align*}
in the sense that the local time term is zero as I\{$X_{t+u} = L \leq c(t+u)\} = 0$ as $u\in[0, \tau_c]$, which together with the definition of $V$ in \eqref{OSPFLR} proves $V^c \leq V$.

(v) With the aid of the above discussions, now is time for the justification of the uniqueness of $b$. Here comes the first relation $c\geq b$ on $[0, T]$ and $c = b$ on $[t_*, T]$. Let us fix $(t, x)\in(0, T)\times(0, \infty)$ so that $x<b(t)\wedge c(t)$ and consider the stopping time
\begin{align*}
\sigma_b = \inf\{ 0\leq s\leq T-t: X_{t+s}\geq b(t+s) \}.
\end{align*}
Via inserting $\sigma_b$ in place of $s$ on equations \eqref{ETRPLGB} and \eqref{LSPTFOVCII}, we obtain
\begin{align}
&\mathbb{E}_{t, x} \left[e^{-r\sigma_b} V\left( t+\sigma_b, X_{t+\sigma_b} \right)\right] = G(t, x) - rK \mathbb{E}_{t, x}\left[ \int_0^{\sigma_b} e^{-ru} I\{ L < X_{t+u} < b(t+u) \}  du\right] \nonumber\\
&-\mathbb{E}_{t, x}\left[ \int_0^{\sigma_b} e^{-ru}\left( rKZ + \sigma^2 (X_{t+u})^2 Z_x \right)(t+u, X_{t+u}) I\{ X_{t+u} <b(t+u), X_{t+u} < L\} du \right]\nonumber\\
&- \frac{1}{2} \int_0^{\sigma_b} e^{-ru} {I}\{ b(t+u)> L \} (K-L)Z_x(t+u, L-) d \mathbb{E}_{t, x} \left[l_u^L(X) \right],\nonumber\\
&+ \frac{1}{2} \int_0^{ \sigma_b } e^{-ru} {I}\{ b(t+u) = L \} \bigl( V_x\left( t+u, L+ \right) - G_x\left( t+u, L- \right) \bigr)  d \mathbb{E}_{t, x} \left[l_u^L(X) \right] \\
&\mathbb{E}_{t, x} \left[e^{-r{\sigma_b}} V^c\left( t+\sigma_b, X_{t+\sigma_b} \right)\right) = G(t, x) - rK \mathbb{E}_{t, x}\left[ \int_0^{\sigma_b} e^{-ru} {I} \{ L < X_{t+u} < c(t+u) \} du\right] \nonumber\\
&-\mathbb{E}_{t, x}\left[ \int_0^{\sigma_b} e^{-ru}\left( rKZ + \sigma^2 (X_{t+u})^2 Z_x \right)(t+u, X_{t+u}) {I} \{ X_{t+u} < c(t+u), X_{t+u}<L \} du \right]\nonumber\\
&- \frac{1}{2} \int_0^{\sigma_b} e^{-ru} {I}\{ c(t+u)> L \} (K-L)Z_x(t+u, L-) d \mathbb{E}_{t, x} \left[l_u^L(X) \right]\nonumber\\
&+\frac{1}{2} \int_0^{ \sigma_b } e^{-ru} {I}\{ c(t+u) = L \} \bigl( U_x\left( t+u, L+ \right) - G_x\left( t+u, L- \right) \bigr)  d \mathbb{E}_{t, x} \left[l_u^L(X) \right] 
\end{align}
which, in view of the relation $V^c \leq V$, implies that
\begin{align}
 &rK \mathbb{E}_{t, x}\left[ \int_0^{\sigma_b} e^{-ru} (I\{X_{t+u} > L\} - I\{c(t+u) > X_{t+u} > L\} ) du\right] \nonumber\\
 & + \mathbb{E}_{t, x}\left[ \int_0^{\sigma_b} e^{-ru}\left( rKZ + \sigma^2 (X_{t+u})^2 Z_x \right)(t+u, X_{t+u}) (I\{ X_{t+u} < L \} - I\{X_{t+u} < c(t+u), X_{t+u} < L  \}) du \right] \nonumber\\
 & + \frac{1}{2} \int_0^{\sigma_b} e^{-ru} ( {I}\{ b(t+u)> L \}  - I\{c(t+u) > L\}) (K-L)Z_x(t+u, L-) d \mathbb{E}_{t, x} \left[l_u^L(X) \right]\nonumber \\
 & -\frac{1}{2} \int_0^{ \sigma_b } e^{-ru} ({I}\{ b(t+u) = L \} \bigl( V_x\left( t+u, L+ \right) - G_x\left( t+u, L- \right) \bigr)  d \mathbb{E}_{t, x} \left[l_u^L(X) \right]\nonumber \\
 & + \frac{1}{2} \int_0^{ \sigma_b } e^{-ru} ({I}\{ c(t+u) = L \} \bigl( U_x\left( t+u, L+ \right) - G_x\left( t+u, L- \right) \bigr)  d \mathbb{E}_{t, x} \left[l_u^L(X) \right] \leq 0,\label{VCLEQVLZ}
\end{align}
and for $u\in[0, \sigma_b]$,
\begin{align*}
&I\{X_{t+u} > L\} - I\{c(t+u) > X_{t+u} > L\} = I\{ X_{t+u} > L, X_{t+u} > c(t+u)\} \geq 0,\\
&I\{X_{t+u} < L\} - I\{ X_{t+u} < c(t+u), X_{t+u} < L \}  = I\{ X_{t+u} > c(t+u), X_{t+u} < L \} \geq 0, 
\end{align*}
and note that in the local time term, 
\begin{align*}
 I\{ b(t+u) > L \}  - I\{c(t+u) > L\}  = 0,
\end{align*}
since $b(t+u) > L$ implies $t+u\geq t_*$, which is precisely $c(t+u) = b(t+u) = B(t+u)$. Assume, on the contrary that $b(t) > c(t)$ for some $t\in[0, t_*)$. We know for $t\in[0, t_*)$, $b(t) \leq L$ and by the monotonicity of $b, c$, we have $c(t)\leq b(t)\leq L$. Then, in the local time term, the definition of $\sigma_b$ yields $I\{b(t+u)= L = X_{t+u}\} = 0$ and $I\{c(t+u)= L = X_{t+u}\} = 0$. That is, for $b(t) > c(t)$ for some $t\in[0, t_*)$, the right hand side of \eqref{VCLEQVLZ} is strictly positive because of the continuity of $b, c$, which is a contradiction. Therefore, we conclude that $c\geq b$ on $[0, T]$.

(vi) Finally, we assume that, there exists $t\in(0, t_*)$ such that $c(t)>b(t)$ and choose $x\in(b(t), c(t))$. Now consider the stopping time
\begin{align*}
\tau_b=\inf\{ 0 \leq u \leq T-t: X_{t+u}\leq b(t+u) \},
\end{align*}
and once again, insert $\tau_b$ in the place of $s$ on equations \eqref{ETRPLGB} and \eqref{LSPTFOVCII}, 
\begin{align}
&\mathbb{E}_{t, x} \left[e^{-r\tau_b} G\left( t+\tau_b, X_{t+\tau_b} \right)\right) = V(t, x) - rK \mathbb{E}_{t, x}\left[ \int_0^{\tau_b} e^{-ru} I\{b(t+u) > X_{t+u} > L\}  du\right] \nonumber\\
&-\mathbb{E}_{t, x}\left[ \int_0^{\tau_b} e^{-ru}\left( rKZ + \sigma^2 (X_{t+u})^2 Z_x \right)(t+u, X_{t+u}) I\{ X_{t+u} <b(t+u), X_{t+u} < L\} du \right]\nonumber\\
&- \frac{1}{2} \int_0^{\tau_b} e^{-ru} {I}\{ b(t+u)> L \} (K-L)Z_x(t+u, L-) d \mathbb{E}_{t, x} \left[l_u^L(X) \right],\nonumber\\
& + \frac{1}{2} \int_0^{ \tau_b } e^{-ru} {I}\{ b(t+u) = L \} \bigl( V_x\left( t+u, L+ \right) - G_x\left( t+u, L- \right) \bigr)  d \mathbb{E}_{t, x} \left[l_u^L(X) \right] \nonumber \\
&\mathbb{E}_{t, x} \left[e^{-r{\tau_b}} G\left( t+\tau_b, X_{t+\tau_b} \right)\right) = V^c(t, x) - rK \mathbb{E}_{t, x}\left[ \int_0^{\tau_b} e^{-ru} {I} \{ L < X_{t+u} < c(t+u) \} du\right] \nonumber\\
&-\mathbb{E}_{t, x}\left[ \int_0^{\tau_b} e^{-ru}\left( rKZ + \sigma^2 (X_{t+u})^2 Z_x \right)(t+u, X_{t+u}) {I} \{ X_{t+u} < c(t+u), X_{t+u}<L \} du \right]\nonumber\\
&- \frac{1}{2} \int_0^{\tau_b} e^{-ru} {I}\{ c(t+u)> L \} (K-L)Z_x(t+u, L-) d \mathbb{E}_{t, x} \left[l_u^L(X) \right]\nonumber\\
& + \frac{1}{2} \int_0^{ \tau_b } e^{-ru} {I}\{ c(t+u) = L \}\bigl( U_x\left( t+u, L+ \right) - G_x\left( t+u, L- \right) \bigr)  d \mathbb{E}_{t, x} \left[l_u^L(X) \right], \nonumber
\end{align}
which, in view of the relation $V^c \leq V$, entails that
\begin{align}
 &-rK \mathbb{E}_{t, x}\left[ \int_0^{\tau_b} e^{-ru} (I\{b(t+u) > X_{t+u} > L\} - I\{c(t+u) > X_{t+u} > L\} ) du\right] \nonumber\\
 & - \mathbb{E}_{t, x}\left[ \int_0^{\tau_b} e^{-ru}\left( rKZ + \sigma^2 (X_{t+u})^2 Z_x \right)(t+u, X_{t+u}) I\{ X_{t+u} < b(t+u), X_{t+u} < L\}  du \right]  \nonumber\\
& + \mathbb{E}_{t, x}\left[ \int_0^{\tau_b} e^{-ru}\left( rKZ + \sigma^2 (X_{t+u})^2 Z_x \right)(t+u, X_{t+u}) I\{ X_{t+u} < c(t+u), X_{t+u} < L \} du \right]\nonumber\\
 & - \frac{1}{2} \int_0^{\tau_b} e^{-ru} ( {I}\{ b(t+u)> L \}  - I\{c(t+u) > L\}) (K-L)Z_x(t+u, L-) d \mathbb{E}_{t, x} \left[l_u^L(X) \right]\nonumber\\
 & - \frac{1}{2} \int_0^{\tau_b} e^{-ru}  I\{c(t+u) = L\} 
\left( U_x(t+u, L+) - G_x(t+u, L-) \right) d \mathbb{E}_{t, x} \left[l_u^L(X) \right]
 \leq 0, \label{IDENTITY1}
\end{align}
where $I\{b(t+u) = L = X_{t+u}^x\} = 0$ and $I\{b(t+u) > L = X_{t+u}^x\} = 0$ by the definition of $\tau_b$ and $c>b$; however, for $u\in[0, \tau_b]$, by the continuity of $c$,
\begin{align*}
&-rK \mathbb{E}_{t, x}\left[ \int_0^{\tau_b} e^{-ru} (0 - I\{c(t+u) > X_{t+u} > L\} ) du\right] \\
 & - \mathbb{E}_{t, x}\left[ \int_0^{\tau_b} e^{-ru}\left( rKZ + \sigma^2 (X_{t+u})^2 Z_x \right)(t+u, X_{t+u}) (0 - I\{ X_{t+u} < c(t+u), X_{t+u} < L \}) du \right] \\
 & - \frac{1}{2} \int_0^{\tau_b} e^{-ru} (0  - I\{c(t+u) > L\}) (K-L)Z_x(t+u, L-) d \mathbb{E}_{t, x} \left[l_u^L(X) \right] \\
 &  - \frac{1}{2} \int_0^{\tau_b} e^{-ru}  I\{c(t+u) = L\} 
\left( U_x(t+u, L+) - G_x(t+u, L-) \right) d \mathbb{E}_{t, x} \left[l_u^L(X) \right] > 0,
\end{align*}
where 
\begin{align}
- (K-L) Z_x(t+u, L-)\leq U_x(t+u, L+) - G_x(t+u, L-) \leq 0
\end{align}
which, combining with \eqref{IDENTITY1}, leads to a contradiction. Therefore, such point cannot exist, which completes the proof.
\end{proof}

\begin{breakablealgorithm}
\caption{Algorithm Solving \eqref{RRFBGB} for the case $B(0)\leq L$}
\begin{algorithmic}
\State Input $K, L, r, \sigma$ and $T$.
\State Let $N = 1000$, $\Delta t = \frac{T}{N}$ and $t_i = T - i\Delta t$.
\For{$i$ from $0$ to $N$}
\State Find the optimal stopping boundary $B(t_i)$;
\If{$B(t_i) = L$}
\State let $t_* = t_i$ and $j = i$.
\EndIf
\EndFor
\For{$i$ from $j$ to $N$}
\State Find the unique maximiser $x^*(t_i)$;
\If{$i = j$ and $x^*(t_j)< L$}
\State let $t_{x^*} = t_j$ and $k = j$.
\EndIf
\If{$x^*(t_i) = L$}
\State let $t_{x^*} = t_i$ and $k = i$.
\EndIf
\EndFor
\For{$i$ from $j$ to $N$}
\State Solve PDE \eqref{FBPL1ee0} on $(L, \infty)$ via finite difference method to obtain $V_x(t_i, L+)$. (Remark \ref{VXTLPLUS})
\EndFor
\State Check Assumption \ref{TMOHTC5}.
\For{$i$ from $0$ to $N$}
\State $t_i = T-i\Delta t$
\If{$t_i\geq t_j$}
\State Let $b(t_i) = B(t_i)$
\EndIf
\If{$t_k\leq t_i\leq t_j$}
\State Let $b(t_i) = L$.
\EndIf
\State $M = i$, Vector $u = [1, \dots, M]\times\Delta t$ and $u[m]$ is its $m$-th entry. Note that $t_b$ is determined in the root-finding process! Search for the root of equation \eqref{RRFBGB} via Riemann sum and LOTUS:
\begin{align*}
&G(t_i, b(t_i)) = \mathbb{E}_{t_{i}, b(t_i)} \left[ e^{-r(T-t_i)}G(T, X_T)\right] \\
&+ rK \sum_{m=1}^M  \left(\frac{e^{-r u[m]}}{\sqrt{2\pi}} I\{b(t_i+u[m]) > L\}\int_{\frac{\log{\frac{L}{b(t_i)}} - \left(r - \frac{\sigma^2}{2}\right)u[m]}{\sigma\sqrt{u[m]}}}^{\frac{\log{\frac{b(t_i + u[m])}{b(t_i)}} - \left(r - \frac{\sigma^2}{2}\right)u[m]}{\sigma\sqrt{u[m]}}} e^{-\frac{z^2}{2}} dz \right)\Delta t + \text{remaining terms of \eqref{RRFBGB}}.
\end{align*}
\EndFor
\end{algorithmic}
\end{breakablealgorithm}

\begin{rem}
The numerical experiments are conducted by mainly following the numerical algorithm in \cite[Section 7]{DetempleandKitapbayev2018}. Note that as in \cite[Page 176]{Qiu2016}, such discrete time numerical scheme also introduces oscillations where smooth-fit condition fails and it can be improved by introducing smaller time steps. Since $t_b$ is unknown, from $t_{x^*}$ forwards, if $b(t_{i-1}) = L$, we always check whether $b(t_i) = L$ satisfies equation \ref{RRFBGB} first and only if it does not do we solve for a root with $b(t_i) < L$ and set $t_i = t_b$. 
This check interacts with the local-time term. Note that the local-time term contains the indicator function $I\{b(t+u) \geq L\}$, which is discontinuous at $b = L$. Since the boundary remains at $L$ on $[t_b, t_*]$, numerical errors place some nodes of the discretised boundary slightly above and others slightly below $L$, which causes oscillations during this time period. At time step $t_i$, the local-time integral over $[0, T-t_i]$ is split into the cells $[u[m-1], u[m]]$, $m = 1, \dots, M$ and if the indicator is evaluated at one point per cell, such a cell contributes either its full weight or nothing and these jumps feed back into the computed boundary as oscillations. To avoid this, we take $b$ piecewise linear (i.e. connecting all the discrete values as a straight line segments) on each cell, compute the sub-interval $[u_a[m], u_b[m]]$ on which $b \geq L$ and integrate only over it.
\end{rem}

\begin{figure}[H]
\centering
\subfloat[]{{\includegraphics[width=.5\textwidth]{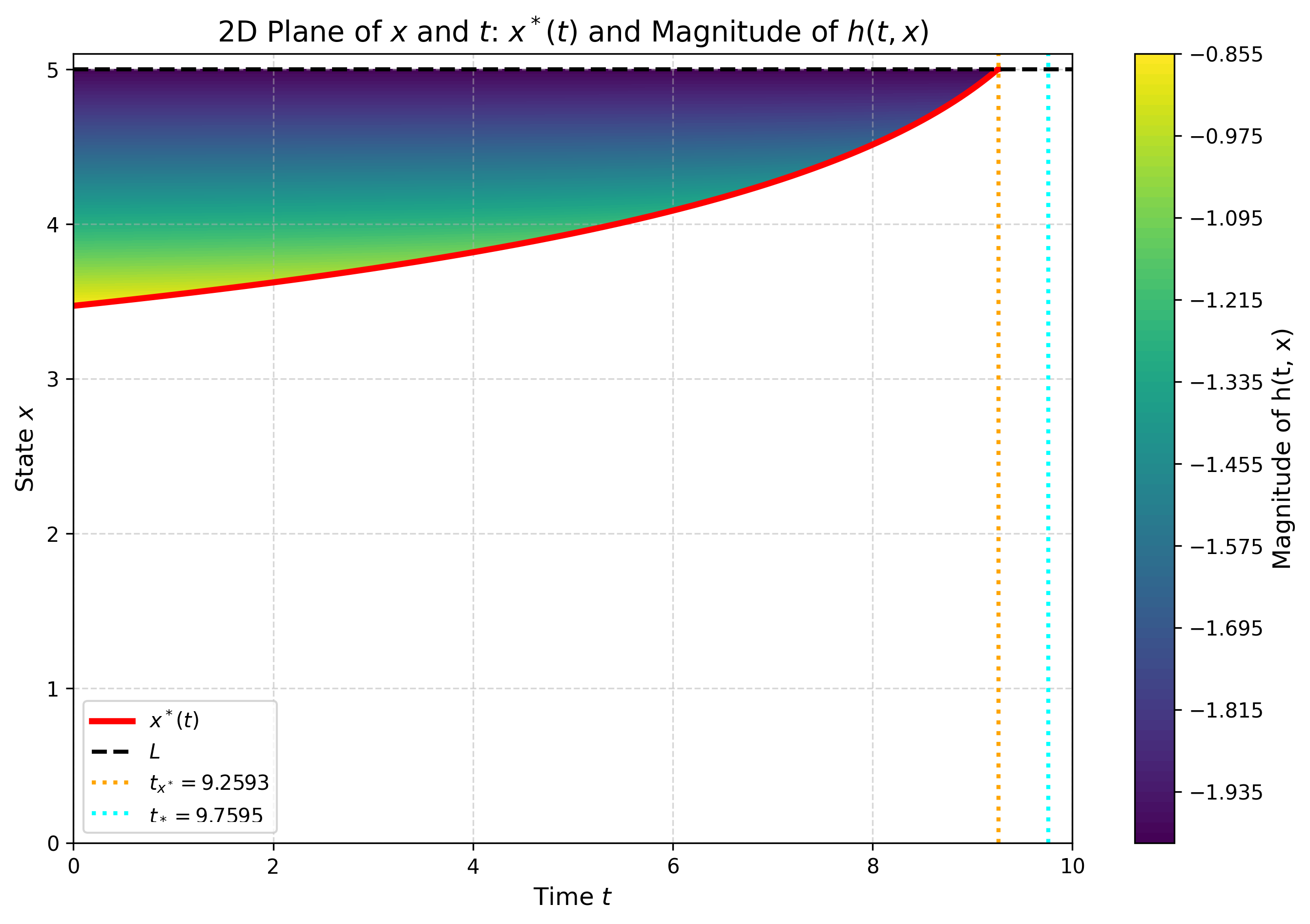} }}
\subfloat[]{{\includegraphics[width=.5\textwidth]{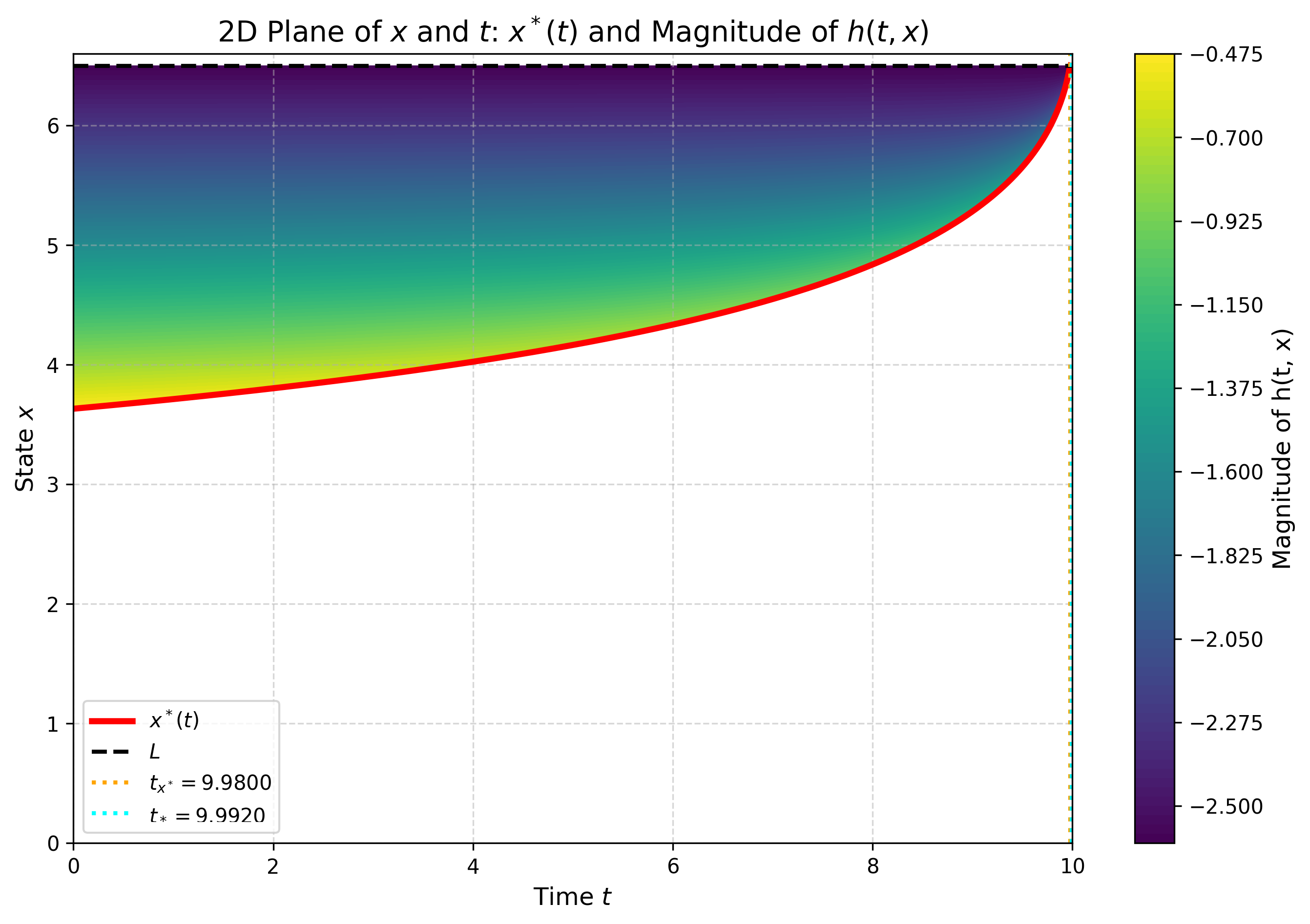} }}
\caption{These figures display $h(t, x)$ (defined in \eqref{CAM1}) in the state space $[0, t_*]\times[x^*(t), L]$ with the same chosen parameters $r=0.05$, $\sigma=0.4$, $K=7$, $T=10$ but different $L$. In particular, figures (a) and (b) correspond respectively to $L=5$ and $L=6.5$. }%
\label{OPSC4IFPICTURE}%
\end{figure}

\begin{figure}[H]
\centering
 \subfloat[]{{\includegraphics[width=.5\textwidth]{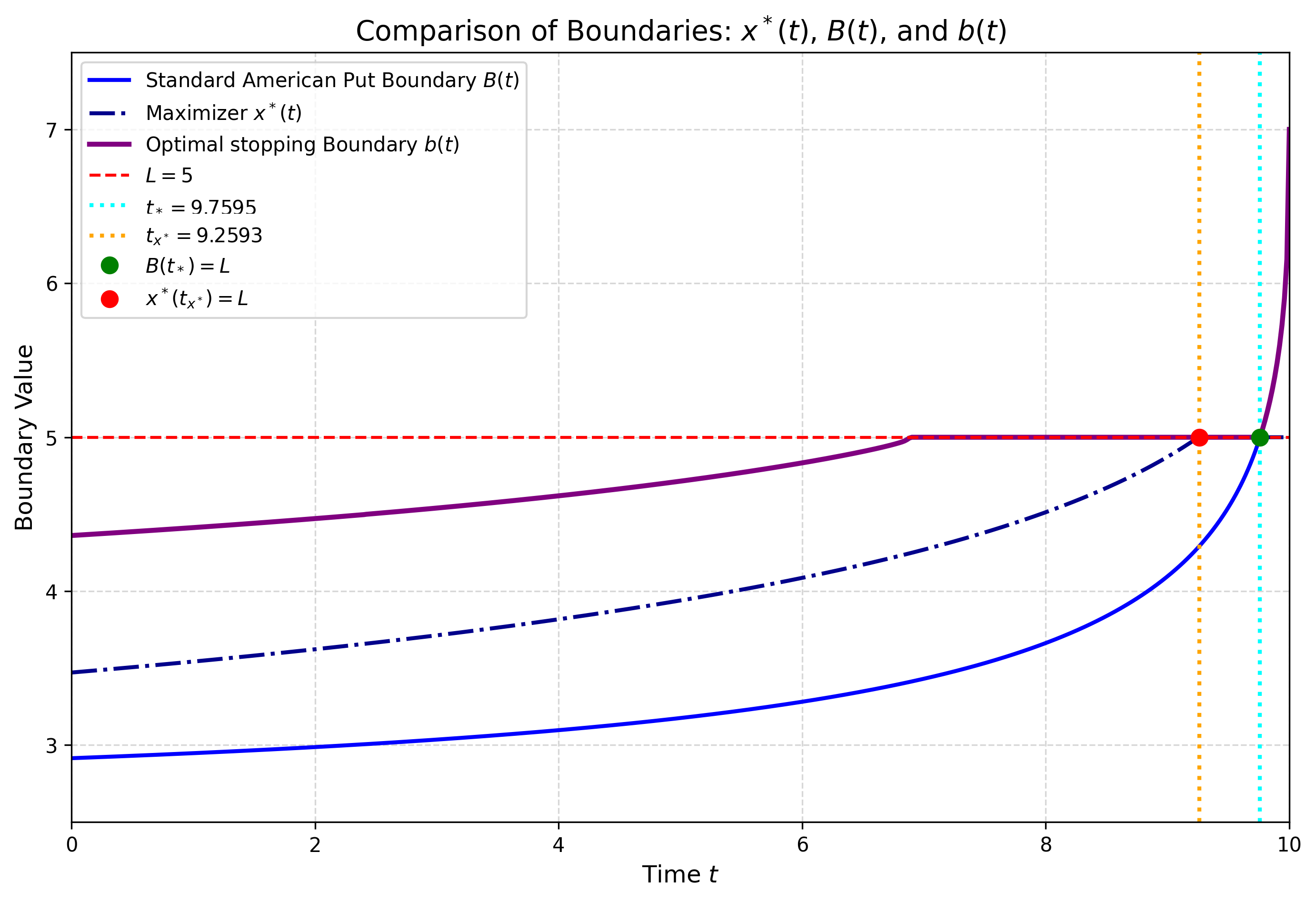}}}
 \subfloat[]{{\includegraphics[width=.5\textwidth]{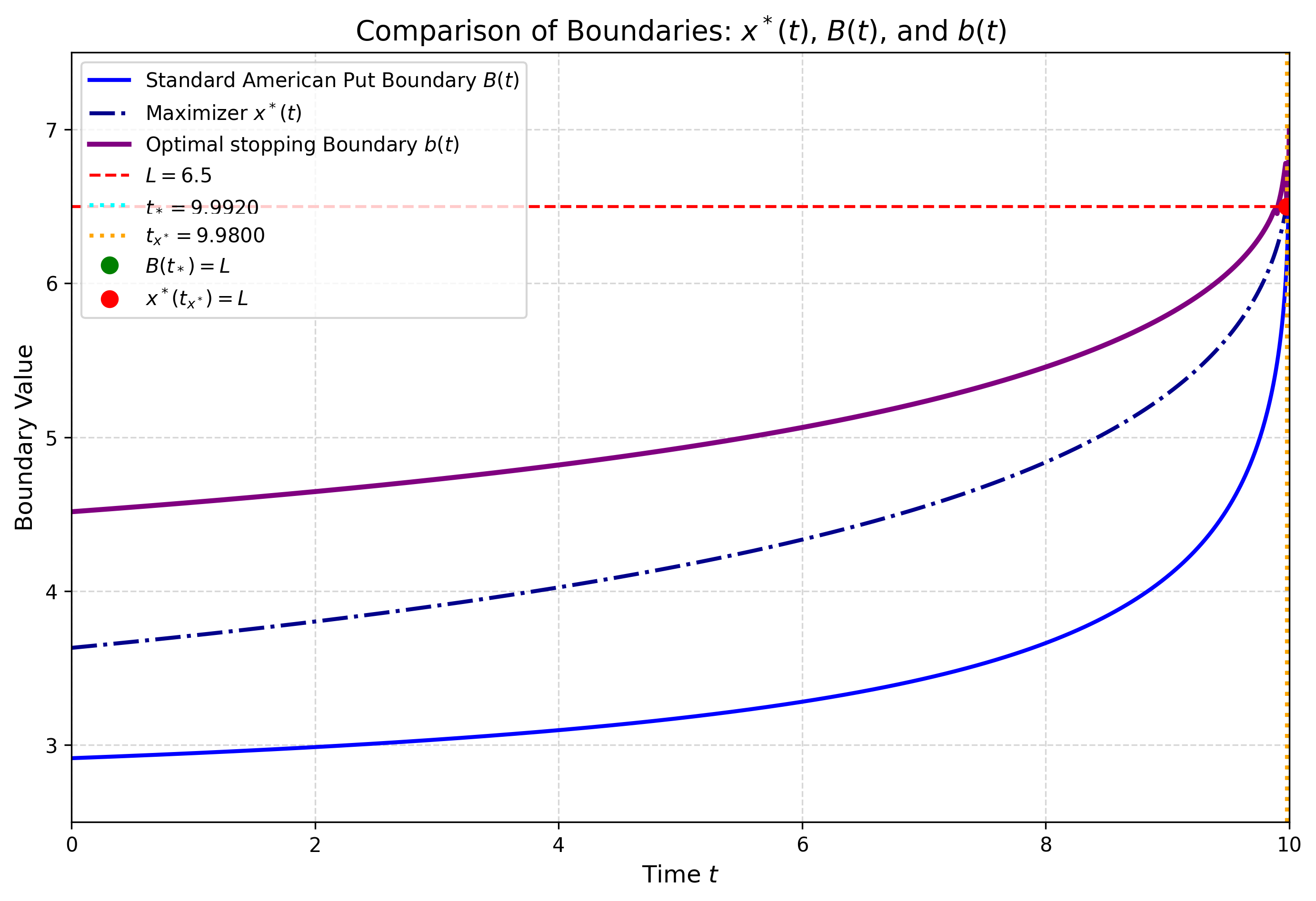}}}
 \caption{These figures display the maps $t\mapsto b(t)$ and $t\mapsto B(t)$ on $[0, T]$ with the same chosen parameters $r=0.05$, $\sigma=0.4$, $K=7$, $T=10$ but different $L$. In particular, figures (a) and (b) correspond respectively to $L=5$ and $L=6.5$. The function $t\mapsto x^*(t)$ is defined in Lemma \ref{TMOGIX}.}%
\label{OPSC4IFPICTUREe}%
\end{figure}

%
%
%
%

\vskip15pt
\begin{appendices}

\section{Auxiliary Results}\label{APFC5}

\subsection{Properties of the Az\'{e}ma Supermartingale }  \label{ALAP1}

For completeness, we include here some auxiliary computations related to the survival process.

\noindent {\bf Notations:}
\begin{align*}
& d_1=\frac{-\log{\frac{L}{x}}+\left( r-\frac{\sigma^2}{2} \right)(T-t)}{\sigma \sqrt{T-t}} \qquad \mathrm{and}\qquad d_2=\frac{-\log{\frac{L}{x}}-\left( r-\frac{\sigma^2}{2} \right)(T-t)}{\sigma \sqrt{T-t}}\\
& \alpha =\frac{2r}{\sigma^2}-1 \qquad \mathrm{and} \qquad  e^{-\frac{d_1^2}{2}}=\left(\frac{L}{x} \right)^\alpha e^{-\frac{d_2^2}{2}}.
\end{align*}


\noindent {\bf Derivatives with respect to time $t$:}
\begin{align}
& \frac{\partial}{\partial t} d_1=-\frac{1}{2} \sigma^{-1} (T-t)^{-\frac{3}{2}}   \log{\frac{L}{x}} -\frac{1}{2} \left(\frac{r}{\sigma}-\frac{\sigma}{2} \right) (T-t)^{-\frac{1}{2}},\nonumber \\
& \frac{\partial}{\partial t} d_2=-\frac{1}{2} \sigma^{-1} (T-t)^{-\frac{3}{2}}   \log{\frac{L}{x}} +\frac{1}{2} \left(\frac{r}{\sigma}-\frac{\sigma}{2} \right) (T-t)^{-\frac{1}{2}},\nonumber \\
& \frac{\partial}{\partial t} Z(t, x) =
\begin{cases}
 &-\sqrt{\frac{1}{2\pi}} \log{\left(\frac{L}{x}\right)} e^{-\frac{d_1^2}{2}} \sigma^{-1} (T-t)^{-\frac{3}{2}}, \qquad\qquad\,\, t<T,\\
 & 0, \qquad\qquad\qquad\qquad\qquad\qquad\qquad\qquad\qquad t=T.
  \end{cases}
  \label{A3}
\end{align}

\noindent {\bf The derivative with respect to $x$ for $(t, x)\in[0, T)\times(0, L)$}
\begin{align}
& \frac{\partial}{\partial x} d_1=\frac{\partial}{\partial x} d_2=x^{-1} \sigma^{-1} (T-t)^{-\frac{1}{2}}. \nonumber \\
& \frac{\partial}{\partial x} Z(t, x) = 
\begin{cases}
 \sqrt{\frac{2}{\pi}} e^{-\frac{d_1^2}{2}} x^{-1} \sigma^{-1} (T-t)^{-\frac{1}{2}} - \alpha \left(\frac{L}{x} \right)^\alpha x^{-1} \Phi(d_2),&x<L\\
 0, & x\geq L\,\, \text{or} \,\, t=T,
 \end{cases} \label{A4}\\
& \frac{\partial^2}{\partial x \partial t} Z(t, x) = \sqrt{\frac{1}{2\pi}} x^{-1} \sigma^{-1} (T-t)^{-\frac{3}{2}} e^{-\frac{d_1^2}{2}} \left( 1- \left( \log{\frac{L}{x}} \right)^2 (T-t)^{-1} \sigma^{-2} + \log{\frac{L}{x}} \sigma^{-1} \left( \frac{r}{\sigma} - \frac{\sigma}{2} \right) \right)\label{BODERX}
\end{align}

\noindent {\bf Bounds on the derivative with respect to $x$ for $(t, x)\in[0, t_\delta)\times(0, L)$ for $t_\delta\in[t_*, T)$}

\begin{align}
0& \leq \frac{\partial}{\partial x } Z(t, x)  =  \sqrt{\frac{2}{\pi}} e^{-\frac{d_1^2}{2}} x^{-1} \sigma^{-1} (T-t)^{-\frac{1}{2}} - \alpha \left(\frac{L}{x} \right)^\alpha x^{-1} \Phi(d_2) \nonumber\\
&< \sqrt{\frac{2}{\pi}} \sigma^{-1} (T-t_\delta)^{-\frac{1}{2}} x^{-1} e^{-\frac{\left( \frac{-\log{\frac{L}{x}}-\left( r-\frac{\sigma^2}{2} \right)T}{\sigma \sqrt{T-t_\delta}} \right)^2}{2}} - \alpha L^{\alpha} x ^{-\frac{2r}{\sigma^2}} \Phi(d_2)=C_2 < \infty. \label{BODERX11}
\end{align}
Note that by L'H\^{o}pital's rule (and that $d_2$ is a function of $x$ and $\frac{2r}{\sigma^2} < 1$, $\lim\limits_{x\to 0} \frac{x}{e^x} = 0$):
\begin{align*}
\lim_{x\to0} \frac{\Phi(d_2)}{x^{\frac{2r}{\sigma^2}}} = \lim_{x\to0} \frac{\sqrt{\frac{1}{{2\pi}}}e^{-\frac{d_2^2}{2}}x^{-1} \sigma^{-1} (T-t)^{-\frac{1}{2}}}{\frac{2r}{\sigma^2}x^{\frac{2r}{\sigma^2}-1}} = \sqrt{\frac{1}{2\pi}} \frac{\sigma}{2r} (T-t)^{-\frac{1}{2}} \lim_{x\to0} \frac{e^{-\frac{d_2^2}{2}}}{x^{\frac{2r}{\sigma^2}}} = 0
\end{align*}

\noindent {\bf Bounds on the derivative with respect to $t$ for $(t, x)\in[0, t_\delta]\times(0, \infty)$ for $t_\delta\in[t_*, T)$}
\begin{align}
0& \leq \big{|}\frac{\partial}{\partial t} Z(t, x)\big{|} =\big{|}\sqrt{\frac{1}{2\pi}} \log{\left(\frac{L}{x}\right)} e^{-\frac{d_1^2}{2}} \sigma^{-1} (T-t)^{-\frac{3}{2}}\big{|} \nonumber\\
&\leq \sqrt{\frac{1}{2\pi}} \sigma^{-1} (T-t_\delta)^{-\frac{3}{2}}\sigma {T} e^{-\frac{1}{2}}=C_3, \label{BoundforZt}
\end{align}
The constant $C_3$ is derive from first noticing that 
\begin{align}
\big{|}\frac{\partial}{\partial t} Z(t, x)\big{|}\leq \sqrt{\frac{1}{2\pi}}\sigma^{-1} (T-t_\delta)^{-\frac{3}{2}} \sup_{\log{\frac{L}{x}}} \log{\left(\frac{L}{x}\right)} e^{-\frac{d_1^2}{2}}, \label{BoundforZtp}
\end{align}
and by letting $l=\log{\frac{L}{x}}$, observe that since $\alpha<0$, we have $r-\frac{\sigma^2}{2}<0$ and
\begin{align*}
g(l)&:= l e^{-\frac{-\left(\frac{l - \left(r-\frac{\sigma^2}{2}\right)(T-t)}{\sigma\sqrt{T-t}}\right)^2}{2}}\leq le^{-\frac{l^2}{2\sigma^2 (T-t)}}.
\end{align*}
An application of first order condition w.r.t $l$ shows that for $l^2 = \sigma^2(T-t)$, $\sup\limits_{l}g(l) = \sigma^2(T-t)e^{-\frac{1}{2}}$, which together with \eqref{BoundforZtp}, yields $\big{|}\frac{\partial}{\partial t} Z(t, x)\big{|} \leq \sqrt{\frac{1}{2\pi}}(T-t_\delta)^{-\frac{3}{2}} \sigma Te^{-\frac{1}{2}}$. Note that as $x\to0$, $Z_t(t, x)\to 0$ uniformly in $t\in[0, t_*]$ as $e^{-\frac{d_1^2}{2}}$ decay to zero much faster than $\log{\frac{L}{x}}$; for $x\geq L$, $Z_t = 0$ since $Z = 1$, $Z_t(t, x)\to 0$ as $x\to L$, so $Z_t$ is continuous at $x = L$, and the bound $C_3$ holds on $[0, t_*]\times(0, \infty)$.

\noindent {\bf The derivative w.r.t time $t$ for $x<L, t<T$:}
\begin{align*}
-rK Z_t - \sigma^2 x^2 Z_{xt}=\sqrt{\frac{1}{2\pi}} e^{-\frac{d_1^2}{2}}  (T-t)^{-\frac{3}{2}} \left( \frac{ r (K-x) \log{\frac{L}{x}}}{\sigma} -\sigma x + \sigma x \left( \frac{ \log{\frac{L}{x}}}{\sigma \sqrt{T-t}} \right)^2 + \frac{\sigma x}{2} \log{\frac{L}{x}} \right)
\end{align*}
and note that the assumption $-rKZ_t-\sigma^2 x^2 Z_{xt}\leq0$ on set $\mathcal{S}$, together with $Z_t\leq0$, suggests that $Z_{xt}\geq\frac{-rKZ_t}{\sigma^2 x^2}\geq0$.

\noindent{\bf The expectation of the local time:}
\begin{align*}
&\lim_{\epsilon\to0} \frac{1}{\epsilon} \mathbb{P}\left( \log{ \frac{L-\epsilon}{y} } <  \left(r-\frac{\sigma^2}{2}\right) u + \sigma W_u  <   \log{\frac{L+\epsilon}{y}}  \right) = \frac{1}{\sigma \sqrt{u} L}\phi\left( \frac{\log{\frac{L}{y}-\left( r-\frac{\sigma^2}{2} \right)u}}{\sigma\sqrt{u}} \right).
\end{align*}
\begin{align*}
d \mathbb{E}_{t, x} \left[ l_u^L\right] &=\sigma^2 L^2 \lim_{\epsilon\to0} \frac{1}{\epsilon} \mathbb{P}\left( \log{ \frac{L-\epsilon}{x} } <  \left(r-\frac{\sigma^2}{2}\right) u + \sigma W_u  <   \log{\frac{L+\epsilon}{x}}  \right)du \\
&= \frac{\sigma L}{\sqrt{u} }\phi\left( \frac{\log{\frac{L}{x}-\left( r-\frac{\sigma^2}{2} \right)u}}{\sigma\sqrt{u}} \right) du,
\end{align*}
where $\phi(z)=\frac{1}{\sqrt{2\pi}} e^{-\frac{z^2}{2}}$ is the probability density function of the standard normal law.

\subsection{Proofs for Section \ref{TFBPC5}} \label{ALAP}

\begin{proof}[{\bf Proof of Lemma \ref{VILSCC5}}]
As in \cite[Page 991]{DuToitandPeskir2009}, by the dominated convergence theorem, the continuity of the gain function $G:[0,T)\times (0, \infty)\mapsto[0, K]$ and the continuity of the flow $x\mapsto X_{t+\tau}^x$, we see the map $(t, x)\mapsto \mathbb{E}_{t, x} \left[ e^{-r\tau} G\left(t+\tau, X_{t+\tau}\right) \right]$  is continuous and thus lower semicontinuous for every stopping time $\tau$ taking values on $[0, T-t)$. Using the fact that the supremum of an l.s.c function define an l.s.c function (see \cite[Remark 2.10, Page 48]{PeskirandShiryaev} and \cite[Page 8]{DuToitandPeskir2009}), Lemma \ref{VILSCC5} follows.
\end{proof}

\begin{proof}[{\bf Proof of Lemma \ref{tauDisoptimal}}] 
Let $\bar{\mathcal{D}}=\mathcal{D}\setminus\{(T, x): x\in(0, \infty)\}$ and \[\bar{\tau}_D=\inf\{s\in[0, T-t): (t+s, X_{t+s}^x)\in\bar{\mathcal{D}}\},\]
where by convention we set the infimum of the empty set to be infinite. Then note that function $G$ is continuous (thereby u.s.c) on $[0, T)\times(0, \infty)$, which, together with Lemma \ref{VILSCC5}, yield that $\bar \tau_D$ is optimal for the following problem
\begin{align*}
\bar{V}(t, x)=\sup_{\tau\in[0, T-t)} \mathbb{E}_{t, x} \left[ e^{-r\tau} G\left(t+\tau, X_{t+\tau}\right) \right].
\end{align*}
That is $\bar{V}(t, x) = \mathbb{E}_{t, x} \left[ e^{-r\bar{\tau}_D} G\left(t+\bar{\tau}_D, X_{t+\bar{\tau}_D}\right) \right]$ where the corresponding stopping set is $\bar{\mathcal{D}}$. Furthermore, by letting $\tilde{\mathbb{E}}_{t, x}\left[ e^{-r\tau} G\left(t+\tau, X_{t+\tau}\right)  \right] = \sup\limits_{\tau\in[0, T-t]}\mathbb{E}\left[ e^{-r\tau} G\left(t+\tau, X_{t+\tau}\right)  \right]$, observe that
\begin{align*}
V(t, x)&=\sup_{\tau\in[0, T-t]} \mathbb{E}_{t, x} \left[ e^{-r\tau} G\left(t+\tau, X_{t+\tau}\right) \right] \\
&=\sup_{\tau\in[0, T-t]} \mathbb{E}_{t, x} \left[ e^{-r\tau} G\left(t+\tau, X_{t+\tau}\right) I\{\tau<T-t\} + e^{-r(T-t)} G\left(T, X_{T}\right) I\{\tau=T-t\} \right]\\
&=\max\bigg\{ \sup_{\tau\in[0, T-t)} \mathbb{E}_{t, x} \left[ e^{-r\tau} G\left(t+\tau, X_{t+\tau}\right) \right],  \tilde{\mathbb{E}}_{t, x} \left[ e^{-r(T-t)} G\left(T, X_{T}\right)\right] \bigg\}\\
&=\max\bigg\{ \tilde{\mathbb{E}}_{t, x} \left[ e^{-r\bar{\tau}_D} G\left( t+\bar{\tau}_D, X_{t+\bar{\tau}_D} \right) \right], \tilde{\mathbb{E}}_{t, x} \left[ e^{-r(T-t)} G\left(T, X_{T}\right)\right] \bigg\},
\end{align*}
which, together with Lemma \ref{MAPOZwrttax}, implies that
\begin{align*}
V(t, x)=
\begin{cases}
&\mathbb{E}_{t, x} \left[ e^{-r\bar{\tau}_D} G\left( t+\bar{\tau}_D, X_{t+\bar{\tau}_D} \right) \right], \,\,\bar{\tau}_D<T-t,\\
&\mathbb{E}_{t, x} \left[ e^{-r(T-t)} G\left(T, X_{T}\right)\right],\qquad\,\,\, \bar{\tau}_D>T-t,
\end{cases}
\end{align*}
which implies that $\tau_D=\bar{\tau}_D\wedge (T-t)$, and the desired claim follows.
\end{proof}

\subsection{Proofs for Section \ref{TCASSC5B0GL}} \label{ALAPii}

\begin{proof}[{\bf Proof of Corollary \ref{GIUCIZL}}]
On the state space $[0, T]\times[L,\infty)$, $G(t, x)=(K-x)^+$. Now suppose that there exists a $\delta>0$ such that for all $(t_1, x_1), (t_2, x_2)\in[0, T]\times [L, \infty)$, $|t_1-t_2| + |x_1-x_2|<\delta$. Then, $|G(t_1, x_1) - G(t_2, x_2)|\leq |x_1-x_2|$ and by choosing $\epsilon = \delta$, we have $|G(t_1, x_1) - G(t_2, x_2)|<\epsilon$ for all $|x_1-x_2|<\delta$. Therefore, the uniform continuity follows.
\end{proof}

\subsection{Proofs for Section \ref{TCASSC5B0LL}} \label{ALAPiii}

\begin{proof}[{\bf Proof of Lemma \ref{TMOGIX}}](i) To show that $G$ is uniformly continuous on $[0, t_*]\times(0, L]$, it is sufficient to show that for every $\epsilon>0$, there exists a $\delta>0$ so that for all $(t_1, x_1), (t_2, x_2)\in[0, t_*]\times(0, L]$, $\lvert t_1- t_2\rvert + \lvert x_1- x_2\rvert <\delta $ implies $\lvert G(t_1, x_1)- G(t_2, x_2)\rvert<\epsilon$.

To this end, we notice that
\begin{align*}
\lvert G(t_1, x_1)- G(t_2, x_2)\rvert &\leq \lvert x_2- x_1 \rvert Z(t_1, x_1) + \lvert K- x_2 \rvert \lvert Z(t_1, x_1)- Z(t_2, x_2)\rvert\\
&< \lvert x_2- x_1 \rvert + K \lvert Z(t_1, x_1)- Z(t_2, x_2)\rvert.
\end{align*}
Then we observe that on $[0, t_*]\times(0, L]$, 
\begin{align*}
\lvert Z(t_1, x_1)- Z(t_2, x_2)\rvert &= \lvert Z(t_1, x_1)- Z(t_2, x_1) + Z(t_2, x_1) - Z(t_2, x_2)\rvert\\
&\leq  \lvert Z(t_1, x_1)- Z(t_2, x_1) \rvert + \lvert Z(t_2, x_1) - Z(t_2, x_2)\rvert\\
&\leq M \lvert t_1 - t_2 \rvert + N  \lvert x_1 - x_2 \rvert, 
\end{align*}
where the second inequality holds true by the mean value theorem and the fact that the partial derivatives of $Z_x(t, x)$ and $Z_t(t, x)$ are bounded, by constants $N$ and $M$, on the set $[0, t_*]\times(0, L]$ (see \eqref{A3}, \eqref{A4}).  Hence,
\begin{align*}
\lvert G(t_1, x_1)- G(t_2, x_2)\rvert<  K  M \lvert t_1 - t_2 \rvert + ( K N + 1) \lvert x_1 - x_2 \rvert, 
\end{align*}
and we can now choose $\delta=\frac{\epsilon}{\max\{KM, KN+1\}}$ and verify that if $(t_1, x_1), (t_2, x_2)\in [0, t_*]\times(0, L]$ satisfy $\lvert t_1- t_2\rvert + \lvert x_1- x_2\rvert <\delta$, then
\begin{align*}
\lvert G(t_1, x_1)- G(t_2, x_2)\rvert < \epsilon,
\end{align*}
thus proving statement (i) as desired.

(ii) The uniform continuity of $G$ means that for every $\epsilon>0$, there exists a $\delta>0$ such that for all $t_1, t_2 \in [0, t_*]$ and $x\in(0, L]$, 
\begin{align*}
\text{$\lvert t_1 - t_2 \rvert < \delta \,\,$ 
 implies  $\,\, \lvert G(t_1, x) - G(t_2, x) \rvert < \epsilon$ },
\end{align*}
which, after rearranging, equals
\begin{align*}
G(t_2, x) - \epsilon  \leq G(t_1, x) \leq G(t_2, x) + \epsilon,
\end{align*}
so we must have, 
\begin{align*}
\max_{x\in(0, L]}  G(t_1, x)&\geq \max_{x\in(0, L]} G(t_2, x)-\epsilon,\\
\max_{x\in(0, L]}  G(t_1, x)&\leq \max_{x\in(0, L]} G(t_2, x)+\epsilon,
\end{align*}
after which, the conclusion follows from
\begin{align*}
\lvert  \max_{x\in(0, L]}  G(t_1, x) - \max_{x\in(0, L]}  G(t_2, x) \rvert < \epsilon.
\end{align*}

(iii) Assume, by contradiction, that statement (iii) is false.
To say that, $t\mapsto x^*(t)$ is not continuous on $[0, t_*]$ means that there exists a sequence $(t_n)\subseteq [0, t_*]$ where $(t_n)\to t$ 
such that $x^*(t_n)$ does not converge $x^*(t)$. According to statement (ii) and the definition of $x^*(t)$, we have
\begin{align*}
\lim_{t_n\to t} \max_{x\in(0, L]}G(t_n, x) \to  \max_{x\in(0, L]}G(t, x) = G (t, x^*(t)),
\end{align*}
but the assumption is 
\begin{align*}
\lim_{t_n\to t} \max_{x\in(0, L]}G(t_n, x) =  \max_{x\in(0, L]}G(t, x) = G\left(t, \lim_{t_n\to t} x^*(t_n)\right) \neq G (t, x^*(t)),
\end{align*}
which is a contradiction, we therefore conclude that statement (iii) holds true.
\end{proof}

\begin{Lemma} \label{xisugxxin}
Let $F(b) = Z(t, Le^{-b})$. Then, 

(i) the map $b\mapsto \log{F(b)}$ is concave on $[0, \infty)$;

(ii) for each $t<T$, the map $x\mapsto (K-x)Z(t, x)$ has a unique maximiser $x^*(t)$ on $(0, L]$;

(iii) if $x^*(t) < L$, then $G_{xx}(t, x^*(t)) < 0$.
\end{Lemma}
\begin{proof}
In this proof, we slightly abuse notation $b$, which is not the optimal stopping boundary $b$ in the main text.
(1) Let $Y_s = \left(r - \frac{\sigma^2}{2}\right)s + \sigma W_s$ and $H_b = \inf\{s: Y_s \geq b\}$ so that by \cite[Pages 759 and 760]{Shiryaev1999}, 
\begin{align*}
&\{H_b \leq T-t\} = \{\max_{s\in[0, T-t]} Y_s \geq b\},\\
&\mathbb{P}(H_b\leq T-t) = \mathbb{P}\left( \max_{s\in[0, T-t]} Y_s \geq b \right) = Z(t, Le^{-b}).
\end{align*}
Fix $\delta > 0$, observe that 
\begin{align}
\frac{F(b+\delta)}{F(b)} = \frac{Z(t, Le^{-(b+\delta)})}{Z(t, Le^{-b})} = \frac{\mathbb{P}(H_{b+\delta} \leq T-t)}{\mathbb{P}(H_{b} \leq T-t)} = \frac{\mathbb{P}(H_{b+\delta} \leq T-t, H_b \leq T-t)}{\mathbb{P}(H_{b} \leq T-t)} \label{FirstID}
\end{align}
where the last equality is due to the fact that $\{H_{b+\delta}\leq T-t\}\subset\{H_{b}\leq T-t\}$. By the strong Markov property, it follows that $H_{b+\delta} = H_b + H'_\delta$ on $H_b<\infty$, where $H'_\delta$ (independent of $H_b$ and identically distributed as $H_\delta$) is the first hitting time of $\delta$ by the restarted process $Y'_u = Y_{H_b + u}-b$ such that by the conditional probability formula
\begin{align}
\eqref{FirstID} = \mathbb{P}(H_{b+\delta}\leq T-t | H_b\leq T-t) = \mathbb{P}(H'_\delta\leq T-t-H_b | H_b\leq T-t).
\end{align}
Note that if $\frac{F(b+\delta)}{F(b)}$ is decreasing in $b$ for every $\delta>0$, that is, if the map $b\mapsto \mathbb{P}(H'_\delta\leq T-t-H_b | H_b\leq T-t)$ is decreasing, then $\log{F(b)}$ is concave (log-concavity). Note that for $b=0$, $F(0) = 1$.

For $b>0$, $H_b$ has probability density function $f_{H_b}(s) = \frac{b}{\sigma\sqrt{2\pi s^{3}}} e^{-\frac{\left( b - \left(r-\frac{\sigma^2}{2}\right)s \right)^2}{2\sigma^2 s}}$ and for $b'>b$, let
\begin{align*}
l(s) = \frac{f_{H_b'}(s)}{f_{H_b}(s)} = \frac{b'}{b} e^{-\frac{(b')^2 - b^2}{2\sigma^2s} + \frac{\left(r-\frac{\sigma^2}{2}\right)(b'-b)}{\sigma^2}},
\end{align*}
which implies that $s\mapsto l(s)$ is increasing. Now by letting 
\begin{align*}
A_b = \int_0^s f_{H_b}(r) dr, \, B_b = \int_s^{T-t} f_{H_b}(r)dr, \,A_{b'} =  \int_0^s f_{H_{b'}}(r) dr, \, B_{b'} =  \int_s^{T-t} f_{H_{b'}}(r) dr
\end{align*}
which, from the definition and monotonicity of $l(s)$, suggests that
\begin{align*}
A_{b'} &= \int_0^s l(r) f_{H_b}(r) dr \leq l(s)\int_0^s f_{H_b}(r) dr = l(s)A_b,\\
B_{b'} &= \int_s^{T-t} l(r) f_{H_b}(r) dr \geq l(s)\int_s^{T-t} f_{H_b}(r) dr = l(s)B_b.
\end{align*}
This leads to the following inequalities
\begin{align}
\frac{B_{b'}}{A_{b'}} \geq \frac{B_b}{A_b},\qquad \frac{B_{b'}}{A_{b'}+B_{b'}} = \frac{\frac{B_{b'}}{A_{b'}}}{1+\frac{B_{b'}}{A_{b'}}}\geq \frac{\frac{B_b}{A_b}}{1+\frac{B_b}{A_b}} = \frac{B_{b}}{A_{b}+B_{b}},
\end{align}
where the second inequality follows from the map $x\mapsto \frac{x}{1+x}$ being increasing for $x\in(0, \infty)$ and is actually
\begin{align}
&\frac{\mathbb{P}(s\leq H_{b'}\leq T-t)}{\mathbb{P}(0\leq H_{b'}\leq s) + \mathbb{P}(s\leq H_{b'}\leq T-t)} = \frac{\mathbb{P}(s\leq H_{b'}\leq T-t)}{\mathbb{P}(0\leq H_{b'}\leq T-t)}\nonumber\\
& \geq \frac{\mathbb{P}(s\leq H_{b}\leq T-t)}{\mathbb{P}(0\leq H_{b}\leq s) + \mathbb{P}(s\leq H_{b}\leq T-t)} =  \frac{\mathbb{P}(s\leq H_{b}\leq T-t)}{\mathbb{P}(0\leq H_{b}\leq T-t)}
\end{align}
which, by the conditional probability formula, gives us 
\begin{align}
\mathbb{P}( H_{b'}\geq s | H_{b'}\leq T-t) \geq \mathbb{P}( H_{b}\geq s | H_{b}\leq T-t),
\end{align}
such that
\begin{align}
\mathbb{P}( H_{b'}\leq T-t-s | H_{b'}\leq T-t) \leq \mathbb{P}( H_{b}\leq T-t - s | H_{b}\leq T-t).\label{SecondID}
\end{align}
We see that for $H_{b'+\delta}$ and $H_{b+\delta}$
\begin{align*}
&\mathbb{P}( H'_{\delta}\leq T-t-H_{b'} | H_{b'}\leq T-t) = \int_0^{T-t} f_{H'_\delta}(s) \mathbb{P}( H_{b'}\leq T-t-s | H_{b'}\leq T-t)ds \\
& \leq \mathbb{P}( H'_{\delta} \leq T-t - H_b | H_{b}\leq T-t) = \int_0^{T-t} f_{H'_\delta}(s) \mathbb{P}( H_{b}\leq T-t-s | H_{b}\leq T-t)ds,
\end{align*}
and statement (i) follows.

(2) Write the gain function $G(t, Le^{-b}) = (K-Le^{-b})F(b)$ such that $\log{G} = \log{(K-Le^{-b})} + \log{F(b)}$. Note that $\log{(K-Le^{-b})}$ is strictly concave on b, which together with statement (i) yields the strict concavity of $\log{G}$ and thereby, the quasi-concavity of $G$ on $b$. A strict quasi-concavity function can have at most one local maximiser, which automatically makes it a unique global maximiser on $b\in[0, \infty)$. Note also that as $b=0$, $G = (K-L)$ and as $b\to\infty$, $G(t, Le^{-b}) \to 0$, which implies that two geometric possibilities: (a) it rises to a single unique peak and then strictly decreases to $0$; (b) it is strictly decreasing from the start, peaking at the boundary $b=0$( i.e. $x=L$), and strictly decreasing to $0$. This proves statement (ii).

(3) For $b>0$ and $t<T$, $F(b) > 0$ is $C^{\infty}$ and hence, $\log{F(b)}$ is $C^2$ on $(0, \infty)$. Step (1) shows that the map $b\mapsto \frac{F(b+\delta)}{F(b)}$ is decreasing for every $\delta>0$ such that the map $b\mapsto \log{F(b+\delta)} - \log{F(b)}$ is decreasing as well, which, for $b'>b$, is
\begin{align*}
 \log{F(b'+\delta)} - \log{F(b')} \leq  \log{F(b+\delta)} - \log{F(b)},
\end{align*}
of which, we further divide both sides with $\delta$ and take the limit as $\delta\to0$,
\begin{align}
\frac{d\log{F(b)}}{db}|_{b=b'} \leq \frac{d\log{F(b)}}{db}|_{b=b}
\end{align}
and thereby, meaning that the first derivative is decreasing and that $(\log{F(b)})''\leq 0$ on $(0, \infty)$. Next, we compute the second derivative of $\log{G}(t, Le^{-b})$ for fixed $t<T$,
\begin{align}
\log{G}'' = (\log(K-Le^{-b}) + \log{F(b)})'' \leq  (\log(K-Le^{-b}))'' < 0, \label{ThirdID}
\end{align}
where the last inequality is due to the strict concavity of $\log(K-Le^{-b})$. Notice that
\begin{align}
(\log{G})_{xx} = -\frac{G_x}{G^2} + \frac{G_{xx}}{G},
\end{align}
at a critical point, say $b^*$, we must have $G_x(t, b^*) = 0$ for fixed $t<T$ and that by \eqref{ThirdID}, $G_{xx}(t, b^*) < 0$ as $G(t, b^*) > 0$. Finally, we see that $b = \log{L} - \log{x}$, we have $db = -\frac{1}{x}dx$ such that $G_{bb}(t, Le^{-b}) = x G_x(t, x) + x^2 G_{xx}(t, x)$, once again, at the critical point we must have $G_x(t, x^*(t)) = 0$, which leads to $G_{xx}(t, x^*(t)) < 0$ for $x^*(t) < L$. This proves statement (iii).
\end{proof}

\subsection{Proofs for Section \ref{TOSPRC5}} \label{ALAPiv}
\begin{proof}[{\bf Proof of Lemma \ref{LSPTFFLLB}}]
By the definition of $F$, 
\begin{align*}
F(t, x) = 
\begin{cases}
e^{-rt} V(t, x),&(t, x)\in\mathcal{C},\\
e^{-rt} G(t, x),&(t, x)\in\mathcal{D}_1\cup\mathcal{D}_2\cup\mathcal{D}_3,
\end{cases}
\end{align*}
statement (a) is immediate. 

To prove statement (b), it suffices to show that $F_t + \mathbb{L}_X F$ is locally bounded on $(\mathcal{C}\cup \mathcal{D}_1 \cup \mathcal{D}_2\cup \mathcal{D}_3)\cap\mathcal{K}$ for each compact set $\mathcal{K}$ in $[0, T]\times(0, \infty)$, see \cite[Page 427]{PeskirandShiryaev}. Since
\begin{align*}
F_t + \mathbb{L}_X F=
\begin{cases}
0,  & (t, x)\in\mathcal{C},\\
-r K,  & (t, x)\in \mathcal{D}_3,\\
(-r K Z - \sigma^2 x^2 Z_x)(t, x),   & (t, x)\in\mathcal{D}_1 \cup \mathcal{D}_2 ,
\end{cases}
\end{align*}
where the last case is locally bounded as it is a continuous function on $\{\mathcal{D}_1 \cup \mathcal{D}_2 \} \cap \mathcal{K}$.

The maps $t \mapsto F_x(t, b(t)+)$ for $t\in[0, T)$ and $t\mapsto F_x(t, L\pm)$ for $t\in[0, t_b]$ are continuous as the result of the fact that $V$ is $C^{1, 2}$ on the continuation set $\mathcal{C}$, see \cite[Page 8]{DetempleandKitapbayev2018} and \cite[Page 131]{PeskirandShiryaev}. Moreover, the map $t\mapsto F_x(t, b(t)-)$ is continuous as $F_x(t, b(t)-)=e^{-rt} G_x(t, b(t)-)$ for $t\in[0, T]$ and the map $t\mapsto b(t)$ is continuous, which, together with $F_x(t, L\pm) = e^{-rt} G_x(t, L\pm)$ for $t\in[t^b, T]$, justifies statements (c) and (d).

Finally, statement (e) follows via
\begin{align*}
F_{xx}(t, x)=
\begin{cases}
\frac{2e^{-rt}}{\sigma^2 x^2} \left( r V - r x V_x - V_t \right)(t, x), & (t, x)\in\mathcal{C},\\
0, & (t, x)\in \mathcal{D}_3,\\\
\frac{2e^{-rt}}{\sigma^2 x^2} \left( -G_t - rx G_x - \sigma^2 x^2 Z_x - rx Z \right) (t, x),  & (t, x)\in\mathcal{D}_1 \cup \mathcal{D}_2,
\end{cases}
\end{align*}
where the second case indicates that $F$ is convex on $\mathcal{D}_3$ (see \cite[Page 526]{Peskir2005AC}), which is a stronger condition than (e), the first case is immediate from the value function being non-negative and $C^{1, 2}$; and in the third case, $G_t(t, x)\leq0$ and the remaining terms are continuous.
\end{proof}

\end{appendices}

\end{document}